\documentclass[11pt]{article}
\usepackage[margin=1in]{geometry}
\usepackage[utf8]{inputenc}
\usepackage{amsmath,amssymb,amsthm}
\usepackage{algorithm,algpseudocode}
\usepackage{hyperref}
\usepackage{DP-macros}

\newtheorem{theorem}{Theorem}
\newtheorem{lemma}[theorem]{Lemma}

\newtheorem{claim}[theorem]{Claim}

\newtheorem{corollary}[theorem]{Corollary}

\newtheorem{definition}[theorem]{Definition}

\newcommand{\ip}[2]{\langle #1, #2\rangle }
\newcommand{\cut}[1]{}

\usepackage{color}

\renewcommand{\hat}{\widehat}

\title{General Gaussian Noise Mechanisms and Their Optimality for Unbiased Mean Estimation}
\author{Aleksandar Nikolov\thanks{Department of Computer Science,
    University of Toronto, email: \url{anikolov@cs.toronto.edu}} \and Haohua Tang\thanks{Department of Computer Science, University of Toronto, email: \url{haohua.tang@mail.utoronto.ca}}}
\date{}
\begin{document}
\maketitle

\begin{abstract}
    We investigate unbiased high-dimensional mean estimators in differential privacy. We consider differentially private mechanisms whose expected output equals the mean of the input dataset, for every dataset drawn from a fixed bounded domain \(K\) in \(\R^d\). A classical approach to private mean estimation is to compute the true mean and add unbiased, but possibly correlated, Gaussian noise to it. In the first part of this paper, we study the optimal error achievable by a Gaussian noise mechanism for a given domain \(K\) when the error is measured in the \(\ell_p\) norm for some \(p \ge 2\). We give algorithms that compute the optimal covariance for the Gaussian noise for a given \(K\) under suitable assumptions, and prove a number of nice geometric properties of the optimal error. These results generalize the theory of factorization mechanisms from domains \(K\) that are symmetric and finite (or, equivalently, symmetric polytopes) to arbitrary bounded domains.

  In the second part of the paper we show that Gaussian noise mechanisms achieve nearly optimal error among all private unbiased mean estimation mechanisms in a very strong sense. In particular, for \emph{every input dataset}, an unbiased mean estimator satisfying concentrated differential privacy introduces approximately at least as much error as the best Gaussian noise mechanism. We extend this result to local differential privacy, and to approximate differential privacy, but for the latter the error lower bound holds either for a dataset or for a neighboring dataset, and this relaxation is necessary.
\end{abstract}

\section{Introduction}
Unbiased estimation is a classical topic in statistics, and an elegant
theory of the existence and optimality of unbiased estimators, and
methods for constructing unbiased estimators and proving lower bounds
on their variance have been developed over the last century. We refer
the reader to the monographs on this topic by Nikulin and
Voinov~\cite{VoinovNikulin-vol1,VoinovNikulin-vol2}, and other
standard statistical texts such as~\cite{MathStats09}. One highlight
of this theory is the existence of uniformly minimum variance unbiased
estimators (UMVUE), i.e., estimators whose variance is smaller than
any other unbiased estimator for \emph{every} value of the
parameter. Such estimators can be derived via the Rao-Blackwell and
Lehmann-Scheff\'e theorems.  In addition to their mathematical
tractability, unbiased estimators also have the nice property that
averaging several estimators decreases the mean squared
error. Moreover, in some cases unbiased estimators are also minimax
optimal, e.g., the empirical mean in many settings.

In this paper we consider the basic problem of mean estimation under
the additional constraint that the privacy of the data must be
protected. In particular, we study estimators computed by a randomized
algorithm \(\mech\) (a mechanism), that satisfies \((\eps,
\delta)\)-differential privacy~\cite{DworkMNS06}. 
We focus on the setting where the dataset \(X = (x_1, \ldots, x_n)\) consists of a sequence
of \(n\) data points from a domain \(K \subseteq \R^d\), and our goal
is to estimate the mean \(\mu(X) := \frac1n \sum_{i = 1}^n x_i\) using an \emph{unbiased} differentially private
mechanism \(\mech\). Formally, we use the following definition of
unbiased mechanisms.
\begin{definition}
  A mechanism \(\mech\) that takes as input datasets of \(n\) points
  in some domain \(K\subseteq \R^d,\) and
  outputs a vector in \(\R^d\) is unbiased 
  over \(K\subseteq \R^d\) (for the empirical mean) if, for every dataset \(X \in K^n\), we have
  \(
  \E[\mech(X)] = \mu(X),
  \)
  where the expectation is over the randomness of \(\mech\).
\end{definition}
We note that this definition captures being \emph{empirically}
unbiased, i.e., unbiased with respect to the dataset. This is a
desirable property when multiple differentially private analyses are
performed on the same dataset, as being unbiased allows decreasing the
mean squared error by averaging. Moreover, bias has been raised as a
concern when publishing official statistics such as census data. The
bias of differentially private mean estimation algorithms in this
context was examined, for example, in~\cite{ZhuHF21,ZFHDT23}. It is
also possible to define unbiased algorithms in a distributional
setting, by assuming that \(X\) contains i.i.d.~samples from an
unknown distribution \(P\), and \(\E[\mech(X)]\) equals the mean of
\(P\), with the expectation taken over the randomness in choosing
\(X\), and the randomness used by \(\mech\).  We give such a
definition in Section~\ref{sect:dist}, and extend our results to
this setting as well.

The prototypical example of an unbiased differentially private
mechanisms is given by oblivious, or noise-adding mechanisms, i.e.,
mechanisms that output \(\mech(X) := \mu(X) + Z\) for some mean \(0\)
random variable \(Z \in \R^d\) drawn from a fixed distribution that's
independent of \(X\). A particularly important oblivious mechanism is
the Gaussian noise mechanism~\cite{DworkMNS06,DworkKMMN06}, for which
\(Z\) is a mean \(0\) Gaussian random variable in \(\R^d\) with
covariance matrix proportional to the identity, and scaled
proportionally to the \(\ell_2\) diameter of the domain \(K\). Another family of
oblivious mechanisms is given by the matrix mechanism from~\cite{MM},
and, more generally, factorization
mechanisms~\cite{NTZ,factormech}. Factorization mechanisms are defined
for a finite domain \(K := \{\pm w_1, \ldots, \pm w_N\} \subseteq \R^d\), where a
dataset $X$ can be represented by a histogram vector $h \in [0,1]^N$,
defined by letting $h_i$ equal the difference between the fraction of
points in $X$ equal to $w_i$ and the fraction of points equal to
$-w_i$. Then \(\mu(X)\) can be written as $Wh$, where \(W\) is the
$d\times N$ matrix with columns $w_1, \ldots, w_N$. A
factorization mechanism chooses (usually by solving an optimization
problem) matrices $L$ and $R$ for which $W = LR$, and outputs
$L(Rh + Z) = \mu(X) + LZ$ for a mean \(0\) Gaussian noise random
variable \(Z\) with covariance matrix proportional to the
identity, and scaled proportionally to the maximum \(\ell_2\) norm of
a column of the matrix \(R\). Thus, factorization mechanisms can be seen as either a
post-processing of the Gaussian noise mechanism, or a method of
achieving privacy by adding correlated mean \(0\) Gaussian
noise. Factorization mechanisms have received a lot of attention in
differential privacy, both from the viewpoint of theoretical
analysis~\cite{MM,NTZ,factormech,HU22,HUU22,MRT22}, and from a more
applied and empirical viewpoint~\cite{HDMM,CMRT22}.

There are, also, natural and widely used non-oblivious unbiased
mechanisms. For example, each iteration of the differentially private
stochastic gradient descent algorithm for convex
minimization~\cite{BassilyST14} uses an unbiased mechanism to compute
a private unbiased estimate of the current gradient. The mechanism
first samples a subset of the data points, and then adds Gaussian
noise to the mean of the sample. The subsampling effectively adds
noise to the mean in a way that's not independent from the dataset
\(X\).  Another example of a non-oblivious unbiased mechanism is
randomized response~\cite{Warner65}. In its simplest form, randomized
response is defined for the domain \(K := \{0,1\}\), and involves
releasing, for each data point \(x_i\) in \(X\), \(x_i\) with
probability \(\frac{e^\eps}{1+e^\eps}\) and \(1-x_i\) with probability
\(\frac{1}{1 + e^\eps}\). It is easy to compute an unbiased estimator
of \(\mu(X)\) from these released points, but the additive error of
this estimator is not independent of \(X\). There are also other
mechanisms based on randomized response used for higher-dimensional
mean estimation in the local model of differential privacy, see
e.g.~\cite{DJW,cdp,factormech}. They are, likewise, unbiased but not
oblivious.

Among the mechanisms mentioned above, in high dimensional settings the
mechanisms that add (possibly correlated) Gaussian noise, i.e., the
Gaussian noise mechanism and factorization mechanisms based on it,
tend to give the lowest error for a given dataset size. Let us call
mechanisms that add unbiased but potentially correlated Gaussian noise
general Gaussian noise mechanisms. This paper is
motivated by the following questions.
\begin{enumerate}
\item What is the best error achievable by a general Gaussian noise mechanism
for a given domain \(K\) and a given measure of error? Relatedly,
  can this error, and the optimal covariance matrix for the noise be
  computed efficiently?
  
\item Are general Gaussian noise mechanisms indeed optimal among all unbiased
  private mean estimators? Can their error be improved on some
  ``nice'' datasets?
\end{enumerate}

To answer the first question, we study the best \(\ell_p\) error
achievable by a differentially private Gaussian noise mechanism, i.e., a mechanism
\(\mech_\Sigma(X) := \mu(X) + Z\) for \(Z \sim N(0, \Sigma)\), when
\(X \in K^n\) for a bounded domain \(K\). Here, \(N(\mu,\Sigma)\) is
the Gaussian distribution with mean $\mu$ and covariance matrix
\(\Sigma\). 
Let us denote the unit Euclidean ball in \(\R^d\) by
\(B_2^d := \{x \in \R^d: \|x\|_2 \le 1\}\), where \(\|x\|_p :=
(|x_1|^p + \ldots + |x_d|^p)^{1/p}\) is the standard \(\ell_p\) norm
on \(\R^d\). For a positive semidefinite matrix $M\in \mathbb{R}^{d\times d}$ and a
real number $p \ge 1$, we
define $\trp{p}(M):=(\sum_i^d{M_{ii}^p})^{1/p}.$ We define
\(\trp{\infty}(M) := \max_{i=1}^d M_{ii}\). Note that
$\trp{p}(M)$ is simply the $\ell_p$-norm of the diagonal entries of
$M$. Finally, for sets $K$ and $L$, let us write
$$K\shinclu L \iff \exists v\in \R^d, K+v\subseteq L$$
We now define the following key
quantity.

\begin{definition}
For a bounded set ${K}\subseteq \mathbb{R}^d$ and \(p \in [2, \infty]\), we define
\[
  \Gamma_p({K}):=\inf\left\{\sqrt{\trp{p/2}(AA^T)}: K\shinclu AB_2^d\right\},
\]
where the infimum is over \(d\times d\) matrices \(A\), and 
\end{definition}

The definition of \(\Gamma_p\) is motivated by the next theorem. Its proof
is given in Section~\ref{sect:gengauss}
\begin{theorem}\label{thm:ub}
  For any \(p \in [2, \infty]\), any \(\eps >0\), any
  \(\delta \le e^{-\eps}\), and any bounded set \(K \subseteq \R^d\),
  there exists a mechanism \(\mech\) that is unbiased over \(K\), and,
  for any \(X \in K^n\) achieves
  \[
    (\E\|\mech(X) - \mu(X)\|_p^2)^{1/2} \lesssim
    \frac{\sqrt{\min\{p,\log(2d)\}\log(1/\delta)}\ \Gamma_p(K)}{\eps n},
  \]
  and satisfies \((\eps, \delta)\)-differential privacy. The mechanism
  outputs \(\mu(X) + Z\), where \(Z\) is a mean \(0\) Gaussian random
  variable with covariance matrix proportional to $AA^T$, for a matrix
  \(A\) such that \(K\shinclu AB_2^d\) and \(\sqrt{\trp{p/2}(AA^T)}
  \lesssim \Gamma_p(K).\)
\end{theorem}
In addition, it is not hard to also show that this bound on \(
(\E\|\mech(X) - \mu(X)\|_p^2)^{1/2}\) is also tight up to the
\(\sqrt{\min\{p,\log(2d)\}}\) factor (see the proof of
Lemma~\ref{lm:err-trace}). Thus, \(\Gamma_p(K)\) nearly captures the
best $\ell_p$ error we can achieve by a general Gaussian noise mechanism for
mean estimation over \(K\).  

In the special case when \(p \in \{2,\infty\}\) and \(K\) is a
finite set symmetric around \(0\), the general Gaussian mechanism
above is equivalent to a factorization mechanism, as we show in
Section~\ref{sect:fact}. Factorization mechanisms, as defined, for
example, in~\cite{MM,factormech}, can be suboptimal for non-symmetric
\(K\). A trivial example is a singleton \(K = \{w\}\) for some \(w\neq
0\), for which any
factorization mechanism would add non-zero noise, but the optimal
mechanism just outputs \(\mu(X) := w\) with no noise. This issue is
the reason we allow \(K\) to be shifted in the definition of
\(\Gamma_p(K)\), and one can similarly modify factorization mechanisms
by allowing a shift of \(K\), i.e., shifting the columns of the matrix
\(W\) by a fixed vector. It is also not difficult to extend standard
factorization mechanism formulations from minimizing \(\ell_p\) error
for \(p \in \{2, \infty\}\), to minimizing \(\ell_p\) error for any
\(p \ge 2\). Independently from our work, this was also done in a
recent paper by Xiao, He, Zhang, and Kifer, who, in addition, also considered error measures that are convex
functions of the per-coordinate variances~\cite{XHZK23}. These modification
allow deriving factorization mechanisms equivalent to the
general Gaussian mechanism achieving error \(\Gamma_p(K)\) when \(K\)
is finite, as shown in Theorem~\ref{thm:Gammap-as-factorization}
below.

Extending factorization mechanisms to infinite \(K\), or \(K\)
 of size exponential in the dimension \(d\), however, is more
challenging. A fundamental issue is that the matrices \(W\) and \(R\)
involved in the definition of a factorization mechanism are infinite
or exponentially sized in these cases. Heuristic solutions tailored to
specific structured \(K\) have been proposed, for example
in~\cite{HDMM,XHZK23}. With the definition of general Gaussian mechanisms,
and of \(\Gamma_p(K)\), we take a different approach, moving away from
factorizations and instead focusing on optimizing the covariance
matrix of the noise. This is the role of the matrix \(A\) in the definition of
\(\Gamma_p(K)\): it can be seen as a proxy for the covariance matrix of
the noise, which is, per Theorem~\ref{thm:ub}, proportional to
\(AA^T\). Equivalently, our definition of
\(\Gamma_p(K)\) can be thought of as formulating a factorization
mechanism only in terms of the left matrix of the factorization,
without explicitly writing the right matrix. The benefit in this approach is that the
covariance matrix has size \(d \times d\), independently of the size
of \(K\). Of course, optimizing the covariance matrix may still be
computationally expensive, depending on how complicated \(K\)
is. Nevertheless, we show that finding an optimal \(A\) approximately
achieving \(\Gamma_p(K)\) can be done in polynomial time under natural
conditions. In particular, we show that \(\Gamma_p(K)\) equals the
value of the convex optimization problem
\begin{align*}
  \Gamma_p(K)^2 = &\min \trp{p/2}(M)\\
                  &\text{s.t.}\notag\\
                  &(x + v)^T M^{-1} (x+v) \le 1 \ \ \forall x \in K,\\
                  &M \succ 0, v \in \R^d.
\end{align*}
In Section~\ref{sect:alg} we prove this equivalence, and show that
this optimization problem can be approximately solved in polynomial
time using the ellipsoid method, assuming the existence of an oracle
that approximately solves the quadratic maximization problem
\(\max_{x \in K}(x + v)^T M^{-1} (x+v)\) for a given \(v \in \R^d\)
and a given positive definite matrix \(M\). (The notation
\(M \succ 0\) above means that \(M\) is positive definite.)  Beyond
finite \(K\), such oracles exist for many classes of \(K\), e.g.,
affine images of \(\ell_p\) and Schatten-\(p\) balls when \(p \ge 2\),
and affine images of other symmetric norms and unitarily invariant
matrix norms: see~\cite{BLN21} and Section~\ref{sect:alg} for more
information. As one example, we can compute \(\Gamma_p(K)\) over
zonotopes: sets of the type \(K := [u_1, v_1] + \ldots + [u_N, v_N]\)
where \(u_1, \ldots, u_N \in \R^d\), and \(v_1, \ldots, v_N \in \R^d\)
are given explicitly, and \([u_i, v_i]\) is the line segment joining
\(u_i\) and \(v_i\). Such sets are simply affine images of the
$\ell_\infty$ ball $[-1,+1]^N$, and for them the maximization problem
\(\max_{x \in K}(x + v)^T M^{-1} (x+v)\) can be solved approximately using
algorithmic versions of Grothendieck's
inequality~\cite{grothendieck,AlonN04}.  These computational results
are the first general theoretical guarantees that allow optimizing the
Gaussian noise covariance matrix for domains \(K\) that are not
explicitly presented finite sets or polytopes in vertex
representation.

In addition to computational tractability, we also prove a number of properties of
\(\Gamma_p(K)\) which significantly generalize known facts about
factorization norms and factorization mechanisms. Let us highlight
some of these properties, whose proofs are given in Section~\ref{sect:Gammap}:
\begin{itemize}
\item \(\Gamma_p(K)\) behaves like a norm on convex sets: it is
  monotone under inclusion, absolutely homogeneous under scaling, i.e., \(\Gamma_p(tK) = |t|\Gamma_p(K)\), and satisfies the
  triangle inequality with respect to Minkowski sum, i.e.,
  $\Gamma_{p}(K+L)\leq \Gamma_{p}(K)+\Gamma_{p}(L)$.

\item \(\Gamma_p(K)\) admits a nice dual characterization. For
  example, for \(p=2\) we have
  \[
    \Gamma_2(K) = \sup\{\tr(\cov(P)^{1/2}): P \in \Delta(K)\}.
  \]
  Above, \(\Delta(K)\) is the set of probability
  measures supported on \(K\), \(\cov(P)\) is the covariance matrix of
  a probability distribution \(P\), and \(\cov(P)^{1/2}\) is its
  positive semidefinite square root of \(\cov(P)\). This characterization (and its
  generalization to \(p > 2\), both given in
  Section~\ref{sect:duality}) is particularly useful for proving lower
  bounds on \(\Gamma_p(K)\).
\end{itemize}

Next we turn to the question of the optimality of general Gaussian noise
mechanisms. Our main result is the following theorem. 
\begin{theorem}\label{thm:main}
  Let \(c > 0\) be a small enough absolute constant, let \(p \in [2,
  \infty]\), and let $\mech$ be an
  $(\eps, \delta)$-differentially private mechanism that is unbiased
  over a bounded set \(K \subseteq \R^d\). If \(\eps \le c\), and
  \(\delta \le \min\left\{\frac{c}{n},\frac{c\eps^2}{d^2}\right\}\),
  then the following holds. For any dataset \(X \in K^n\), there exists a
  neighboring dataset \(X' \in K^n\) (which may equal \(X\)) for which
  \[
    \sqrt{\E\left[\norm{\mech(X')-\mu(X')}_p^2\right]} \gtrsim
    \frac{\Gamma_p(K)}{n{\eps}}.
  \]
\end{theorem}
Above, the notation \(A \gtrsim B\)
for two quantities \(A\) and \(B\) is used to mean that there exists
an absolute constant \(c > 0\) such that \(A \ge cB\). 

Theorem~\ref{thm:main}, together with Theorem~\ref{thm:ub}, shows that, for any \(\ell_p\) norm for \(p \ge
2\), and any
unbiased mechanism \(\mech\) over any domain \(K\), there is a
general Gaussian noise mechanism that has
\(\ell_p\) error not much larger than that of \(\mech\). Moreover,
this is true in an \emph{instance optimal} sense: every dataset \(X\)
has a neighbor \(X'\) for which the correlated Gaussian noise
mechanism has smaller error (up to small factors).  This somewhat
complicated version of instance optimality is necessary, since, for
any dataset \(X\), there is an unbiased \((0, \delta)\)-differentially
private mechanism that has error \(0\) on \(X\). Roughly speaking,
this mechanism outputs \(\mu(X)\) on \(X\), and on any other dataset
outputs \(\mu(X)\) with probability \(1-\delta\) and some other output
with probability \(\delta\), chosen to make the mechanism
unbiased. The exact construction is given in Section~\ref{sect:counterexample}. This illustrates the key difficulty in proving a result such
as Theorem~\ref{thm:main}: one has to rule out the possibility that a
mechanism can ``cheat'' by hard-coding the true answer for one
dataset, thus outperforming an ``honest'' mechanism on that
dataset. The construction in Section~\ref{sect:counterexample} shows
that this type of cheating is possible even for unbiased mechanisms.

We note that Theorem~\ref{thm:main} is a simplification of our main
result, and what we prove is actually stronger: we show that either
the error of \(\mech\) on \(X\) is comparable to that of a general
Gaussian noise mechanism, or there is a neighboring dataset \(X'\) for which
the error grows with \(\frac{1}{\sqrt{\delta}}\). Since usually
\(\delta\) is chosen to be very small, this means that, on every
input, \(\mech\) is either approximately dominated by a general Gaussian noise mechanism, or has huge
error in the neighborhood of the input. It is worth also  noting that
in the case of $d=1$, adding Laplace noise gives smaller error than
Gaussian noise. In higher dimensions, however, the Laplace noise
mechanism, or any mechanism that satisfies pure differential priacy
(i.e., $\delta = 0$), typically has much larger error than the
Gaussian noise mechanism.

Theorem~\ref{thm:main} strengthens results showing the optimality of
Gaussian noise mechanisms among oblivious mechanisms~\cite{factormech}
to the more general class of unbiased mechanisms. As mentioned above,
some natural unbiased mechanisms fail to be oblivious. We also
consider the class of unbiased mechanisms more natural and robust than
the class of oblivious mechanisms. It is worth noting further that Gaussian
noise mechanisms are known to be approximately optimal \emph{in the worst
  case} for sufficiently large datasets among all differentially
private mechanisms. This follows, for example, by reductions from
general mechanisms to oblivious mechanisms
in~\cite{BhaskaraDKT12,factormech}.

Independently from our work, unbiased mechanisms for private mean
estimation were also recently investigated by Kamath, Mouzakis,
Regehr, Singhal, Steinke, and Ullman~\cite{trilemma23}. They focus on
mechanisms that take as input independent samples from an unknown one-dimensional
distribution in some family, e.g., distributions with
bounded \(k\)-th moments, and study the trade-off between the bias of
the mechanism with respect to the distribution's mean, and its
variance. Their results are incomparable to ours: on the one hand, they study
one-dimensional mean estimation and prove worst-case (i.e., minimax)
lower bounds, rather than instance per instance lower bounds; on the
other hand, they give a tight trade-off between bias and variance in
their setting, whereas we only consider unbiased mechanisms. They also
show that for Gaussian mean estimation no purely private (i.e.,
\((\eps, \delta)\)-differentially private with \(\delta = 0\))
unbiased mechanism can achieve finite variance.  

The kind of instance optimality guarantee shown in
Theorem~\ref{thm:main} is reminiscent of the theory of uniformly
minimum variance unbiased estimators in statistics. Our result is also
related to results by Asi and Duchi~\cite{asi20optimal,AsiD20}, and by
Huang, Liang, and Yi~\cite{HuangLY21} who also study instance
optimality for unbiased differentially private algorithms. Asi and
Duchi also focus on notions of being unbiased defined with respect to
the dataset, but only treat pure differential privacy, i.e., the
\(\delta = 0\) setting.  The results in~\cite{asi20optimal} are
tailored to one-dimensional estimation, and the lower bounds proved
there are not sufficiently strong to prove the result in
Theorem~\ref{thm:main}. The results in~\cite{AsiD20} do extend to
higher-dimensional problems, but use a non-standard definition of
unbiased mechanisms that is incomparable with the more standard
definition we use. Their lower bounds also rely on some strong
regularity assumptions that we do not make. Finally, the work of
Huang, Liang, and Yi~\cite{HuangLY21} is not restricted to unbiased
mechanisms, and, similarly to our results, considers mean estimation and
optimality for the  neighborhood of every dataset (in fact only
considering datasets resulting from removing points). In their results,
however, optimality is proved only up to a factor of at least the
square root of the dimension, making them less interesting in high
dimensions. 

Let also mention that there are other approaches
to instance optimality in differential privacy. One approach is based on
local minimax rates, initiated by Ruan and
Duchi~\cite{DuchiRuan18}, and Asi and Duchi
in~\cite{asi20optimal}, and explored further
in~\cite{DongYi22,MSU22}. The local minimax rates framework is
not suitable for proving strong instance optimality for high
dimensional problems, as noted in~\cite{MSU22}. Another approach,
similar to that of~\cite{HuangLY21}, based on optimality with respect
to subsets of a dataset, was considered in~\cite{DKSS23}. Their
results are also restricted to one-dimensional problems. In general,
existing instance optimality results tend to only be meaningful in low
dimensional settings, and our work is a rare example of instance
optimality up to small factors for a high-dimensional problem.

In addition to Theorem~\ref{thm:main}, we also show analogous, and, in
fact, stronger results for other variants of differential privacy. For
concentrated differential privacy~\cite{DworkR16,BunS16}, our result
holds for \emph{every} dataset \(X \in K^n\). This is also the case
for local differential privacy~\cite{EvfimievskiGS03,KLNRS}, recalled
in the next definition.

\begin{definition}
  A (non-interactive) locally differentially private mechanism
  \(\mech\) is defined by a tuple of randomized algorithms \(\mathcal{A}, \mathcal{R}_1,
  \ldots, \mathcal{R}_n\), where each \(\mathcal{R}_i\), called a
  local randomizer, receives a single data point from a domain \(K\),
  and is \((\eps, 0)\)-differentially private with respect to that
  point, and the algorithm's output on a dataset \(X := (x_1, \ldots,
  x_n)\) is defined by
  \[
    \mech(X) := \mathcal{A}(\mathcal{R}_1(x_1), \ldots, \mathcal{R}_n(x_n)),
  \]
  i.e., the postprocessing of the outputs of the randomizers
  \(\mathcal{R}_i\) by the aggregator \(\mathcal{A}\).
  Moreover, each algorithm uses independent randomness. 
\end{definition}
Local differential privacy captures settings in which there is no
trusted central authority, and instead every data owner ensures the
privacy of their own data. Mean estimation algorithms in local
differential privacy are typically based on randomized response, and
are not oblivious even when they are unbiased. The recent work~\cite{asi2022optimal} studies unbiased
mechanisms for mean estimation in this model when \(K = B_2^d\), and
gives a tight characterization of algorithms that achieve optimal
error in \(\ell_2\). Here we extend their results to arbitrary bounded
domains \(K\), and to error measured in other norms, albeit with a
somewhat less tight characterization. Our characterization shows that
that the instance optimal unbiased locally private mechanism for mean
estimation  has every local agent add a linear transformation of (non-oblivious)
subgaussian noise. The precise results are given in Section~\ref{sect:ldp}.

As mentioned above, we also prove a result analogous to
Theorem~\ref{thm:main} for mechanisms that are unbiased in a
distributional sense, i.e., where the expectation of the mechanism's
output, given an input dataset drawn i.i.d.~from some
distribution over the domain \(K\), matches the mean of the
distribution. In that case we show that, for every distribution \(P\)
on \(K\), the \(\ell_p\) error of every distributionally unbiased
mechanism is at least \(\gtrsim \frac{\Gamma_p(K)}{n{\eps}}\) either
on \(P\), or on some
distribution \(Q\) that is at distance at most \(\frac1n\) from \(P\)
in total variation distance. Our result thus shows that, in the distributional setting,
too, unbiased mechanisms are dominated by general Gaussian noise mechanisms in
the neighborhood of every input distribution. The precise results are
given in Section~\ref{sect:dist}.

Using the properties of the \(\Gamma_p\) function that we establish,
we can prove lower bounds on \(\Gamma_p(K)\) for domains \(K\) that
naturally appear in applications. Together with
Theorem~\ref{thm:main}, these lower bounds imply concrete lower bounds
on the \(\ell_p\) error of any unbiased mechanism that hold in the
neighborhood of every dataset. Analogous results would for
distributionally unbiased mechanisms and unbiased locally
differentially private mechanisms also follow via the appropriate
variant of Theorem~\ref{thm:main}. First, we state one such lower
bound for estimating moment tensors.
\begin{theorem}\label{thm:lb-tensor}
  Let \(c > 0\) be a small enough absolute constant, let \(p \in [2,
  \infty]\) and let \(\ell\) be a positive integer. Let $\mech$ be an
  $(\eps, \delta)$-differentially private mechanism that takes as
  input datasets in \((B_2^d)^n\). Suppose that for every dataset
  \(X := (x_1, \ldots, x_n)\),
  \[
    \E[\mech(X)] = M_\ell(X) := \frac1n \sum_{i=1}^n {x_i^{\otimes \ell}}.
  \]
  If \(\eps \le c\), and
  \(\delta \le \min\left\{\frac{c}{n},\frac{c\eps^2}{d^{2\ell}}\right\}\),
  then, for any dataset \(X \in (B_2^d)^n\), there exists a
  neighboring dataset \(X' \in (B_2^d)^n\) (which may equal \(X\)) for which
  \[
    \sqrt{\E\left[\norm{\mech(X')-M_\ell(X')}_p^2\right]} \gtrsim
    \frac{1}{\eps n} \left(\frac{d}{\ell}\right)^{\ell/p}.
  \]
\end{theorem}
This lower bound implies that estimating the \(\ell\)-th moment tensor
of a dataset in \(B_2^d\) via an \((\eps, \delta)\)-differentially
private unbiased mechanism requires \(\ell_2\) error at least on the
order \(\frac{d^{\ell/2}}{\eps n}\) for small enough \(\eps\) and
\(\delta\) and constant \(\ell\). This lower bound is nearly matched
by the basic Gaussian noise mechanism, which adds independent
Gaussian noise to each coordinate of the tensor. At the same time, the
projection mechanism~\cite{NTZ} allows \(\ell_2\) error on the scale of
\(\frac{d^{1/4}\log(1/\delta)^{1/4}}{\sqrt{\eps n}}\) for any constant
\(\ell\), which is much smaller for \(n \ll d^{\ell - \frac12}\). This
follows from an analysis similar to that in~\cite{conjunctions}: see
the recent paper~\cite{DLY22} for the argument in the \(\ell=2\)
case. Theorem~\ref{thm:lb-tensor} thus illustrates the cost of using
private unbiased mechanisms: while they produce answers that are
accurate in expectation, they can incur much more error than biased
algorithms, and this is true in the neighborhood of every input. We
note that Theorem~\ref{thm:lb-tensor} can easily be extended to unbiased estimates of the
covariance, and also to a distributional setting, and show a similar
gap between biased and unbiased mechanisms.

Analogous techniques also imply tight upper and lower bounds on
estimating \(\ell\)-way marginals on \(d\)-dimensional binary data
when \(\ell = O(1)\). Before we state these bounds, let us recall the
general connection between query release and mean estimation. Suppose
that $Q = (q_1, \ldots, q_k)$ is a sequence of statistical queries on
a universe \(\uni\), also known as a workload. This means that each
\(q_i\) is specified by a function \(q_i: \uni \to \R\), and,
overloading notation, its value on a dataset
\(\ds := (x_1, \ldots, x_n) \in \uni^n\) is defined by
\( q_i(\ds) := \frac1n \sum_{i =1}^n q_i(\ds).  \) Overloading
notation again, we can define
\(Q(\ds) := (q_1(\ds), \ldots, q_k(\ds))\) to be the sequence of true
answers to the queries in \(Q\) on the dataset \(X\). When \(n=1,\)
i.e., \(X\) consists of the single data point \(x \in \uni,\) we write
\(Q(x)\) rather than \(Q((x)).\) The problem of privately releasing an
approximation to \(Q(\ds)\) is equivalent to mean estimation over the
set \(K_Q := \{Q(x): x \in \uni\}\) in the following sense. Given a
dataset \(X  = (x_1, \ldots, x_n) \in \uni^n\), we can construct a
dataset \(f(X) := (Q(x_1), \ldots, Q(x_n)) \in K_Q^n\). Clearly,
\(\mu(f(X)) = Q(X)\), and any differentially private algorithm \(\mech\) for
(unbiased) mean estimation over \(K_Q\) gives an (unbiased) differentially
private algorithm for releasing \(Q(X)\), simply by running
\(\mech(f(X))\). We can also choose an inverse \(g\) of \(f\) by choosing, for
each \(y \in K_Q\), some \(x \in \uni\) such that \(Q(x) = y\), and
defining \(g(Y)\) for \(Y = (y_1, \ldots, y_n) \in K_Q^n\) as the
function that replaces each \(y_i\) with the chosen \(x_i\) giving
\(Q(x_i) = y_i\). This shows, in turn, that an (unbiased) differentially private
mechanism \(\mech\) that releases \(Q(X)\) gives an (unbiased) private mean
estimation algorithm \(\mech(f(Y))\). These reductions preserve the
privacy parameters, the property of being unbiased, and the error. 

Let us now specialize this discussion to releasing \(\ell\)-way
marginal queries. Let \(\queries^{\text{marg}}_{d,\ell}\) be the
statistical queries over the universe \(\uni := \{0,1\}^d\) where each
query \(\query_{s, \beta}\) in \(\queries^{\text{marg}}_{d,\ell}\) is
defined by a sequence of \(\ell\) indices
\(s := (i_1, \ldots, i_\ell) \in [d]^\ell\) and a sequence of \(\ell\)
bits \(\beta := (\beta_1, \ldots, \beta_\ell) \in \{0,1\}^\ell\), and has
value
\(\query_{s,\beta}(x) = \prod_{j = 1}^\ell |x_{i_j} - \beta_{j}|\) on
every \(x \in \uni\). For an easier to read notation, let us write \(K^{\text{marg}}_{d,\ell}
:= K_{\queries^{\text{marg}}_{d,\ell}}\). By analyzing
\(\Gamma_p(K^{\text{marg}}_{d,\ell})\), we derive the following bound
on the error necessary to release unbiased estimates of the
\(\ell\)-way marginal queries.

\begin{theorem}\label{thm:marginals}
  Let \(c > 0\) be a small enough absolute constant, let \(p \in [2,
  \infty]\) and let \(\ell\) be a positive integer. Let $\mech$ be an
  $(\eps, \delta)$-differentially private mechanism that takes as
  input datasets in \((\{0,1\}^d)^n\). Suppose that for every dataset
  \(X := (x_1, \ldots, x_n)\),
  \[
    \E[\mech(X)] = \queries^{\text{marg}}_{d,\ell}(X).
  \]
  If \(\eps \le c\), and
  \(\delta \le \min\left\{\frac{c}{n},\frac{c\eps^2}{(2d)^{2\ell}}\right\}\),
  then, for any dataset \(X \in (\{0,1\}^d)^n\), there exists a
  neighboring dataset \(X' \in (\{0,1\}^d)^n\) (which may equal \(X\)) for which
  \[
    \sqrt{\E\left[\norm{\mech(X')-\queries^{\text{marg}}_{d,\ell}(X')}_p^2\right]} \gtrsim
    \frac{d^{\frac{\ell}{2} + \frac{\ell}{p}}}{(2\sqrt{2} \ell)^\ell}.
  \]
\end{theorem}

Theorem~\ref{thm:marginals} shows a similar gap as
Theorem~\ref{thm:lb-tensor} between the optimal error achievable by
unbiased and biased mechanisms for releasing marginals. For constant
\(\ell\), the lower bound in Theorem~\ref{thm:marginals} nearly
matches the error achievable by adding i.i.d.~Gaussian noise. By
contrast, the error achieved by the projection mechanism or the
private multiplicative weights mechanism, which can be biased, is much
smaller for moderate values of \(n\) and \(\ell > 1\): for example,
the projection mechanism achieves error on the order of
\(\frac{d^{1/4}\log(1/\delta)^{1/4}}{\sqrt{\eps n}}\) for any constant
\(\ell\)~\cite{conjunctions}. In the case of \(\ell = 1\),
Theorem~\ref{thm:marginals} gives lower bounds on the error achieved
by unbiased mechanisms for one-way marginals that match, up to the
dependence on \(\delta\), the lower bounds against all
\((\eps, \delta)\)-differentially private algorithms proved via
fingerprinting codes~\cite{BunUV14}.  While our lower bounds are for
restricted mechanisms, they hold in the neighborhood of every input
dataset, rather than for a worst case dataset as the fingerprinting
lower bounds.

The proofs of Theorems~\ref{thm:lb-tensor}~and~\ref{thm:marginals},
and more precise upper and lower bounds on
\(\Gamma_p(K^{\text{marg}}_{d,\ell})\) are presented in Section~\ref{sect:lbs}.

\paragraph{Techniques.} In terms of techniques for proving Theorem~\ref{thm:main} and its
extensions, we combine a technique from \cite{factormech}, developed
for proving lower bounds on oblivious mechanisms,
with classical results from the statistical theory of unbiased
estimation. The key insight from~\cite{factormech} is that, in order
to prove a theorem like Theorem~\ref{thm:main}, it is sufficient to
show that, for every
unit vector \(\theta \in \R^d\), the variance of \(\theta^T \mech(X)\)
is bounded below in terms of the width of \(K\) in the direction of
\(\theta\). This reduction is explained in
Section~\ref{sect:high2low}. While proving a lower bound on the worst-case variance of a
one-dimensional private mechanism like \(\theta^T \mech(X)\) is easy,
the challenge is that the lower bound must hold for a fixed \(X\) that
is not allowed to vary with \(\theta\). This is trivial for oblivious
mechanisms, but not for unbiased mechanisms. Nevertheless, we show
that the classical Hammersley-Chapman-Robins (HCR) bound~\cite{H50,CR51} implies the
one-dimensional variance lower bounds we need for pure and
concentrated differential privacy. The situation is more subtle for
approximate differential privacy, i.e., \((\eps,
\delta)\)-differential privacy for \(\delta > 0\). The main technical
issue is that applying the HCR bound requires proving an upper bound
on the \(\chi^2\) divergence between the output
distributions \(\mech(X)\) and \(\mech(X')\) of a differentially private mechanism \(\mech\) on two
neighboring datasets \(X\) and \(X'\). No such finite bound need exist
for \((\eps,\delta)\)-differentially private mechanisms when \(\delta
> 0\).  To get around
this issue, we modify one of the output distributions \(\mech(X)\) and
\(\mech(X')\) so that the the
\(\chi^2\) divergence becomes bounded, and, moreover, the expectations
of the two distributions does not change much, unless one of the two
distributions already has huge variance. Then we can carry out a
win-win analysis: either one of \(\mech(X)\) or \(\mech(X')\) has huge
variance, or the HCR bound can be applied to them.

\paragraph{Notation.} As already noted, we use \(\|x\|_p := (|x_1|^p +
\ldots + |x_d|^p)^{1/p}\) for the \(\ell_p\) norm of a vector \(x \in
\R^d\), and \(B_p^d := \{x \in \R^d: \|x\|_p \le 1\}\) for the
corresponding unit ball. We write the standard inner product in
\(\R^d\) as \(\ip{x}{y} := x_1 y_1 + \ldots + x_d y_d = x^T y\) for
\(x,y \in \R^d\). For a $d\times N$ matrix $M$, we define the $\ell_p
\to \ell_q$ operator norm by $\|M\|_{p\to q} := \sup_{x \in \R^N: x
  \neq 0} \frac{\|Mx\|_q}{\|x\|_p}$. Note that \(\|M\|_{1\to 2}\)
equals the largest \(\ell_2\) norm of a column of $M$, and
\(\|M\|_{2\to \infty}\) equals the largest \(\ell_2\) norm of a row of
\(M\). We also define
the Frobenius (or Hilbert-Schmidt) norm \(\|M\|_F := \tr(M^T
M)^{1/2}\). Note that this is just the \(\ell_2\) norm of $M$ treated
as a vector. 

For a \(d\times d\) matrix \(M\), we use
\(M\succeq 0\) to denote that \(M\) is positive semidefinite, i.e.,
\(M\) is symmetric and satisfies \(x^T M x \ge 0\) for all \(x \in
\R^d\). If \(M\) is also positive definite, i.e., positive
semidefinite and non-singular, we write \(M \succ 0\). We write \(A
\succeq B\) and \(B \preceq A\) when \(A-B \succeq 0\) for two
\(d\times d\) matrices \(A\) and \(B\). We write \(\sqrt{M}\) or
\(M^{1/2}\) for the principle square root of a positive semidefinite
matrix \(M\), i.e., \(\sqrt{M}\) is a positive semidefinite matrix
such that \((\sqrt{M})^2 = M\).

For a probability distribution \(P\), we use \(X \sim P\) to denote
the fact that the random variable \(X\) is distributed according to
\(P\). We use \(\E_{X \sim P}[f(X)]\) to denote the expectation of the
function \(f(X)\) when \(X\) is a random variable distributed
according to \(P\). We use \(\cov(P)\) to denote the covariance matrix of a
distribution \(P\) on \(\R^d\).

\section{Gaussian Noise Mechanisms}
\label{sect:Gammap}

In this section, we introduce the general Gaussian noise mechanism for
estimating means in an arbitrary bounded domain \(K\). This mechanism
 generalizes known factorization mechanisms, as we discuss
later on in the section. The mechanism's error bound generalizes
factorization norms: we prove this fact, and some other important
properties in this section, as well.

\subsection{Preliminaries on Concentrated Differential Privacy}

Our mechanism's privacy guarantees are most cleanly stated in the
language of concentrated differential privacy. We recall the
definition of this variant of differential privacy here, together with
some basic properties of it.

Before we define concentrated differential privacy, it is convenient
to define a ``ratio of probability densities'' for general probability
distributions. We use the following (standard) definition.
\begin{definition}
  For two probability distribution \(P\) and \(Q\) over the same
  ground set, we define
  \(\frac{dP}{dQ}\) as follows. Let \(R = \frac{P +
    Q}{2}\) be a reference distribution, and denote by
  \(\frac{dP}{dR}, \frac{dQ}{dR}\) the Radon-Nykodim derivatives of
  \(P\) and \(Q\) with respect to \(R\). Then we take \(\frac{dP}{dQ}
  := \frac{dP/dR}{dQ/dR}\) where the ratio is defined to be \(\infty\)
  if the denominator is 0 while the numerator is positive. 
\end{definition}

We also recall the definition of R\'enyi divergence.
s
\begin{definition}
  For two probability distributions \(P\) and \(Q\) over the same
  ground set, and a real number \(\alpha > 1\), the R\'enyi divergence
  \(D_\alpha(P\|Q)\) of order \(\alpha\) is defined by
  \[
    D_\alpha(P\|Q) := \frac{1}{\alpha-1} \ln \E_{X \sim Q}
    \Braket{\left(\frac{dP}{dQ}(X)\right)^\alpha}. 
  \]
\end{definition}

We are now ready to define zero-concentrated differential privacy,
following~\cite{BunS16}.
\begin{definition}
  A mechanism \(\mech\) satisfies \(\rho\)-zero concentrated differential
  privacy (\(\rho\)-zCDP) if, for all neighboring datasets \(X,X'\),
  we have
  \[
    D_\alpha(\mech(X)\|\mech(X')) \le \rho \alpha.
  \]
\end{definition}

Concentrated differential privacy satisfies many nice properties: it
has a simple and optimal composition theorem, is invariant under
post-processing, and implies some protection to small groups in
addition to protecting the privacy of individuals. We refer
to~\cite{BunS16} for details. The properties we need are stated in the
following lemmas, and proofs can be found in~\cite{BunS16}.

\begin{lemma}\label{lm:cdp-composition}
  Suppose that \(\mech_1\) is a \(\rho_1\)-zCDP mechanism, and, for
  every \(y\) in the range of \(\mech_1\),
  \(\mech_2(y,\cdot)\) is a \(\rho_2\)-zCDP mechanism.  Then the
  composition \(\mech\) defined on dataset \(X\) by \(\mech(X) :=
  \mech_2(\mech_1(X),X)\) satisfies \((\rho_1 + \rho_2)\)-zCDP.

  In particular, if \(\mech\) satisfies \(\rho\)-zCDP, and \(\mathcal{A}\)
  is a randomized algorithm defined on the range of \(\mech\), then
  the post-processed mechanism defined on dataset \(X\) by
  \(\mathcal{A}(\mech(X))\) satisfies \(\rho\)-zCDP as well. 
\end{lemma}

\begin{lemma}\label{lm:cdp-to-adp}
  If a mechanism \(\mech\) satisfies \(\rho\)-zCDP, then, for any
  \(\delta > 0\), \(\mech\) also satisfies \((\rho + 2\sqrt{\rho
    \log(1/\delta)}, \delta)\)-differential privacy.
\end{lemma}

\begin{lemma}\label{lm:cdp-gaussian}
  Suppose that \(f:K^n \to \R^d\) is a function on size \(n\) datasets
  drawn from the domain \(K\) with \(\ell_2\) sensitivity at most
  \(\Delta\), i.e., for any two neighboring
  datasets \(X\) and \(X'\) we have
  \(
  \norm{f(X) - f(X')}_2 \le \Delta.
  \)
  Then the mechanism that on input \(X\) outputs \(\mech(X) := f(X) +
  Z\) for \(Z \sim N(0,\sigma^2 I)\) satisfies \(\frac{\Delta^2}{2\sigma^2}\)-zCDP.
\end{lemma}

\subsection{A Gaussian Noise Mechanism for General Domains}
\label{sect:gengauss}

Recall the notation $\trp{p}(M)$, defined in the Introduction: for a positive semidefinite matrix $M\in \mathbb{R}^{d\times d}$ and a
real number $p \ge 1$, we
have \(\trp{p}(M):=(\sum_i^d{M_{ii}^p})^{1/p}.\) Moreover, we have
\(\trp{\infty}(M) := \max_{i=1}^d M_{ii}\). Recall also that
$\trp{p}(M)$ is simply the $\ell_p$-norm of the diagonal entries of
$M$.
The next lemma notes a few useful properties of \(\trp{p}(M)\) which follow
from this observation.
\begin{lemma}\label{lm:trp-norm}
  The function \(\trp{p}\) satisfies the following properties for any
  \(p \ge 1\):
  \begin{enumerate}
  \item for any positive semidefinite matrix \(M\), \(\trp{p}(M) = 0\)
    implies \(M = 0\);
  \item for any real number \(t \ge 0\), and any positive semidefinite
    matrix \(M\), \(\trp{p}(tM) = t\ \trp{p}(M)\);
  \item for any two positive semidefinite matrices \(M_1\), and
    \(M_2\), \(\trp{p}(M_1 + M_2) \le \trp{p}(M_1) + \trp{p}(M_2)\);
  \item for any two \(1 \le p \le q \le \infty\), and any \(d \times
    d\) positive semidefinite matrix \(M\),
    \(
    \trp{q}(M) \le \trp{p}(M) \le d^{\frac1q - \frac1p}\trp{q}(M).
    \)
  \end{enumerate}
\end{lemma}
\begin{proof}
  All except the first property are immediate from the observation
  that when \(M\) is a positive semidefinite matrix, \(\trp{p}(M)\) is
  the \(\ell_p\) norm of its diagonal entries. This observation also
  shows that when \(\trp{p}(M) = 0\) and \(M \succeq 0\), the diagonal
  entries of \(M\) are 0. But, since the largest absolute value of any
  entry of a positive semidefinite matrix is achieved on the diagonal,
  this also implies that \(M = 0\). 
\end{proof}

Next recall our notation for inclusion of sets up to shifting: 
for subsets $K$ and $L$ of \(\R^d\), we write
\[K\shinclu L \iff \exists v\in \R^d, K+v\subseteq L.\]
Finally, we recall the \(\Gamma_p(K)\) function defined in the
Introduction as
\[\Gamma_p({K}):=\inf\left\{\sqrt{\trp{p/2}(AA^T)}:
   K\shinclu AB_2^d\right\},\]
where the infimum is over \(d\times d\) matrices \(A\).

The next theorem is the core of the proof of Theorem~\ref{thm:ub} from
the Introduction, and is a slight generalization of Corollary 2.8 from~\cite{NTZ}.
 \begin{theorem}\label{thm:gen-gaussian}
 Suppose that \(p \in [2, \infty]\),  that \(\cset \subseteq\R^d\) is a
 bounded set, and that for some \(d\times d\) matrix \(A\), \(\cset  \shinclu AB_2^d\). Then the mechanism \(\mech\) that, on input \(X \in \cset^n\), outputs 
 \(
 \mech(X) := \mu(X) + Z,
 \)
 where \(Z \sim N(0,\frac{4}{\eps^2n^2} AA^T)\), satisfies
 \(\frac{\eps^2}{2}\)-zCDP.
 In particular, \(\mech\) is an unbiased \(\frac{\eps^2}{2}\)-zCDP
 mechanism \(\mech\) that, for any dataset \(X \in \cset^n\) achieves
 \[
   (\E\|\mech(X) - \mu(X)\|_p^2)^{1/2} \lesssim
   \frac{\sqrt{\min\{p,\log(2d)\}}\ \trp{p/2}(AA^T)^{1/2}}{\eps n}.
 \]
 \end{theorem}
\begin{proof}
  Let us fix some set \(\cset' \subseteq B_2^d\) and vector \(v \in \R^d\) such that \(A\cset' - v =
  \cset\). \(K'\) and \(v\) exist
  by the assumption that
  \[
    \cset \shinclu AB_2^d \iff
    \exists v \in \R^d: \cset +v \subseteq  AB_2^d
    \iff \exists v \in \R^d: \cset \subseteq AB_2^d - v.
  \]
  Define a map \(f:\cset \to \cset'\) that inverts the surjective map from $x \in
  \cset'$ to $Ax - v \in K$. I.e., we choose $f$ so that \(Af(x)
  - v = x\) for each \(x \in \cset\).
  Then, for any dataset
  \(X := (x_1, \ldots, x_n) \in \cset^n\) we can define the dataset
  \(f(X) := (f(x_1), \ldots, f(x_n)) \in (\cset')^n\). Notice that
  this map sends neighboring datasets to neighboring datasets. The
  output of the
  mechanism \(\mech(X)\) is distributed identically to the following post-processing of
  the Gaussian mechanism:
  \[
    A(\mu(f(X)) + Z') - v,
  \]
  where \(Z' \sim N(0, \frac{4}{\eps^2 n^2})\). Since all elements
  of \(f(X)\) lie in \(B_2^d\), the \(\ell_2\) sensitivity of \(\mu(f(X))\) is
  bounded by \(\frac{2}{n}\), and the Gaussian mechanism \(\mu(f(X))
  + Z'\) satisfies \(\frac{\eps^2}{2}\)-zCDP by Lemma~\ref{lm:cdp-gaussian}. The privacy guarantee for \(\mech(X)\) then
  follows from the post-processing property of zCDP (Lemma~\ref{lm:cdp-composition}).
  
  To bound the \(\ell_p\) error of the mechanism \(\mech\), we use
  Jensen's inequality to write
  \begin{align}
    (\E\|\mech(X) - \mu(X)\|_p^2)^{1/2} = (\E \|Z\|_p^2)^{1/2}
    \le (\E \|Z\|_p^p)^{1/p}
    &\lesssim 
      \sqrt{p}\ \left( \sum_{i = 1}^d \var[Z_i]^{p/2}\right)^{1/p}\notag\\
    &\lesssim \frac{\sqrt{p}\ \trp{p/2}(AA^T)^{1/2}}{\eps n}.\label{eq:smallp-err}
  \end{align}
  Above, we used the standard estimate that the \(p\)-th absolute
  moment of a Gaussian \(G \sim N(0,\sigma^2)\) satisfies
  \((\E|G|^p)^{1/p} \lesssim \sqrt{p}\sigma\).  This
  estimate is, in fact, true for any \(\sigma\)-subgaussian random
  variable~\cite[Chapter 2]{Vershynin-HDP}.

  This completes the proof
  of the error bound for $p \le \max\{2,\ln d\}$. Next we consider the case $p
  \ge \max\{2,\ln d\}$. Recall the inequality $d^{\frac1p -
    \frac1q}\|x\|_q \le \|x\|_p \le \|x\|_q$, valid for any $1 \le q
  \le p$ and any $x \in \R^d$. For $q = \max\{2,\ln d\}$ and any $p \ge q$ we have
  $\|x\|_p \le \|x\|_q$ for all $x \in \R^d$. Moreover, for any
  $d\times d$ positive semidefinite matrix $M$, we have $\trp{q/2}{(M)}
  \le d^{\frac2q} \trp{p/2}{(M)} \le e^2 \trp{p/2}(M)$ by Lemma~\ref{lm:trp-norm}.
  Therefore, for $p\ge q$,
  \begin{align*}
    (\E\|\mech(X) - \mu(X)\|_p^2)^{1/2} = (\E \|Z\|_p^2)^{1/2}
    \le (\E \|Z\|_q^2)^{1/2}
    &\lesssim
    \frac{\sqrt{q}\ \trp{q/2}(AA^T)^{1/2}}{\eps n}\\
    &\lesssim
    \frac{\sqrt{q}\ \trp{p/2}(AA^T)^{1/2}}{\eps n},
  \end{align*}
  where we used \eqref{eq:smallp-err} in the penultimate
  inequality. Since we chose $q = \max\{2,\ln d\}\lesssim \log(2d)$,
  this completes the proof.
\end{proof}

Taking \(A\) to achieve \(\Gamma_p(K)\) in Theorem~\ref{thm:gen-gaussian}, and also using
Lemma~\ref{lm:cdp-to-adp}, gives Theorem~\ref{thm:ub}. We also have
the following corollary for zCDP.
\begin{corollary}\label{cor:ub-cdp}
  For any \(p \in [2, \infty]\), any \(\eps >0\), and any bounded set \(K \subseteq
  \R^d\), there exists a mechanism \(\mech\) that is unbiased over
  \(K\), for any \(X \in K^n\) achieves
  \[
    (\E\|\mech(X) - \mu(X)\|_p^2)^{1/2} \lesssim
    \frac{\sqrt{\min\{p,\log(2d)\}}\ \Gamma_p(K)}{\eps n},
  \]
  and satisfies \(\frac{\eps^2}{2}\)-zCDP.
\end{corollary}

We also give a variant of this result for local differential
privacy. 

\begin{theorem}\label{thm:ub-ldp}
  For any \(p \in [2, \infty]\), any \(\eps >0\), and any bounded set \(K \subseteq
  \R^d\), there exists a local mechanism \(\mech\) that is unbiased over
  \(K\), for any \(X \in K^n\) achieves
  \[
    (\E\|\mech(X) - \mu(X)\|_p^2)^{1/2} \lesssim
    \frac{\sqrt{p}(e^\eps + 1)\Gamma_p(K)}{(e^\eps - 1)\sqrt{n}},
  \]
  and satisfies \(\eps\)-local differential privacy. 
\end{theorem}
\begin{proof}
  The proof is analogous to the proof of
  Theorem~\ref{thm:gen-gaussian}, except, instead of using the
  Gaussian noise mechanism, each local agent uses the local point
  release randomizer from~\cite{cdp}, similiraly to how it is used in~\cite{factormech}. In particular, on input \(x_i
  \in K\), agent \(i\) uses the randomizer with \(f(x_i)\) and
  multiplies the answer by the matrix \(A\) that achieves \(\Gamma_p(K)\); the
  aggregator then averages the outputs of the randomizers. The error
  analysis follows analogously to Theorem~\ref{thm:gen-gaussian},
  since the output of each randomizer is \(O\left(\frac{e^\eps + 1}{e^\eps -
    1}\right)\)-subgaussian,\footnote{The paper~\cite{cdp} gave a
  slightly different bound but it is easy to see that the bound we
  claim follows via the same methods.} so the average of the
randomizer outputs is \(O\left(\frac{e^\eps + 1}{(e^\eps -
    1)\sqrt{n}}\right)\)-subgaussian.
\end{proof}

\subsection{Basic Properties of $\Gamma_p(K)$}

In this subsection we prove some important properties of the
\(\Gamma_p\) function. First, we introduce a key definition in
convex geometry.
\begin{definition}
  We define the support function \(h_K:\R^d \to \R\) of a set \(K \subseteq \R^d\) by
  \(
    h_K(\theta) := \sup_{x \in K}\ip{x}{\theta}.
  \)
\end{definition}

Let us note here several important identities involving the support
function. Observe first that, using \(\clconv K\) for the closed convex hull of \(K\),
\begin{equation}\label{eq:supp-convhull}
  h_K(\theta) = h_{\clconv K}(\theta)
\end{equation}
for all \(\theta \in \R^d\). Next, we recall that for any two closed convex sets \(K\) and \(L\) in
\(\R^d\), we have
\begin{equation}\label{eq:supp-contain}
  K\subseteq L \iff \forall \theta \in \R^d: h_K(\theta) \le
    h_L(\theta).
\end{equation}
The forward implication is trivial, whereas the backwards implication
follows from the hyperplane separation theorem: see Chapter 13
of~\cite{Rockafellar} for a proof. Moreover, if only \(L\) is convex
and closed, and \(K\) is arbitrary, it holds that \(K \subseteq L\) if
and only if \(\clconv K \subseteq L\). Together with
\eqref{eq:supp-convhull}, this means that \eqref{eq:supp-contain}
holds for \(K\) arbitrary and \(L\) closed and convex.

Recall that the Minkowski sum of sets
\(K\) and \(L\) in \(\R^d\) is defined by
by \(K + L :=\{x+y: x \in K, y \in L\}\). We have
\[
  \forall \theta \in \R^d: h_{K+L}(\theta) = h_K(\theta) + h_L(\theta).
\]
This identity follows easily from the definition of the support
function.

Finally, let us calculate the support function of a
centrally symmetric ellipsoid \(E := AB_2^d\), i.e., a linear image of
the Euclidean ball: for any \(\theta \in \R^d\), we have
\begin{align*}
  h_{E}(\theta) &= \max_{x \in E} \ip{x}{\theta}
                  = \max_{x \in B_2^d}\ip{Ax}{\theta}
                  = \max_{x \in B_2^d}\ip{x}{A^T\theta}
                  = \|A^T\theta\|_2,
\end{align*}
where the final equality follows by the Cauchy-Schwarz inequality,
which is tight for $x = \frac{A^T\theta}{\|A^T\theta\|_2}$.

First we show that, for any ellipsoid \(E := AB_2^d + u\),
\(\Gamma_p(E)\) is achieved by the same matrix \(A\) defining the
ellipsoid.
\begin{lemma}\label{lm:Gammap-ellipsoid}
  For any ellipsoid \(E := AB_2^d + u\), where \(A\) is a \(d\times
  d\) matrix, \(\Gamma_p(E) = \sqrt{\trp{p/2}(AA^T)}\). 
\end{lemma}
\begin{proof}
  Notice first that, since \(E \shinclu AB_2^d\) by assumption,
  \(\Gamma_p(E) \le \sqrt{\trp{p/2}(AA^T)}\). It remains the show the
  reverse inequality.
  
  Let \(E + v \subseteq A'B_2^d\) for some \(d\times d\) matrix \(A'\)
  and some vector \(u \in \R^d\). This is equivalent to having
  \[
    h_{E + u}(\theta) = \|A^T\theta\|_2 + \ip{u+v}{\theta}
    \le
    h_{A'B_2^d}(\theta) = \|(A')^T \theta\|_2
  \]
  for all \(\theta \in \R^d\). Taking the maximum of \(h_{E +
    u}(\theta)\) and \(h_{E + u}(-\theta)\), we see that this is also
  equivalent to having
  \[
    \|A^T\theta\|_2 + |\ip{u+v}{\theta}| \le \|(A')^T \theta\|_2,
  \]
  for all \(\theta \in \R^d\), and, in turn, this implies
  \(
  \|A^T\theta\|_2 \le \|(A')^T \theta\|_2
  \) holds for all \(\theta \in \R^d\). This last inequality is then
  equivalent to \(AA^T \preceq (A')(A')^T\), which implies that all
  diagonal entries of \(AA^T\) are bounded by the corresponding
  diagonal entries of \((A')(A')^T\). Then,
  \(
  \trp{p/2}(AA^T) \le \trp{p/2}((A')(A')^T).
  \)
  Taking the infimum of this inequality over all \(A'\) such that \(E \shinclu A'B_2^d\)
  proves that   \(\Gamma_p(E) \ge \sqrt{\trp{p/2}(AA^T)}\), as we wanted.
\end{proof}

The next property of \(\Gamma_p\) that we prove is that it behaves
like a norm, i.e., it satisfies the triangle
inequality with respect to Minkowski sum, and is homogeneous. We also
show that it is monotone with respect to inclusion. We use this
property repeatedly in the rest of the paper.

\begin{theorem}\label{thm:triangle}
  The function \(\Gamma_p\) satisfies the following properties:
  \begin{enumerate}
  \item (\emph{invariance with respect to convex hulls}) for any
    bonded set \(K \subseteq \R^d\), \(\Gamma_p(K) = \Gamma_p(\clconv K)\);
  \item (\emph{monotonicity}) whenever \(K \shinclu \clconv L\), we have \(\Gamma_p(K) \le \Gamma_p(L)\);
  \item (\emph{homogeneity}) for any bounded set \(K \subseteq \R^d\), and any \(t \in
    \R\), \(\Gamma_p(tK) = |t| \Gamma_p(K)\);
  \item (\emph{triangle inequality}) for any bounded sets $K, L\subseteq \mathbb{R}^d$,
$\Gamma_{p}(K+L)\leq \Gamma_{p}(K)+\Gamma_{p}(L)$.
  \end{enumerate}
\end{theorem}
\begin{proof}
  The invariance property holds because, since an ellipsoid \(AB_2^d\)
  is convex and closed,
  \(K \subseteq AB_2^d \iff \clconv K \subseteq AB_2^d\). 
  
  The monotonicity property is obvious from the definitions.

  For homogeneity, we just notice that if \(K \shinclu AB_2^d\), then
  \(tK \shinclu tAB_2^d\), and \(\sqrt{\trp{p/2}((t A)(tA)^T)} = |t|
  \sqrt{\trp{p/2}(AA^T)}\). 

  It remains to prove the triangle inequality.
  Let ${E_K := A_K B_2^d}, {E_L := A_L B_2^d}$ be such that $K \subseteq_\leftrightarrow {E_K},
  L\subseteq_\leftrightarrow {E_L}$, and
  \(\Gamma_p(K) = \sqrt{\trp{p/2}(A_K A_K^T)}\),
  \(\Gamma_p(L) = \sqrt{\trp{p/2}(A_L A_L^T)}\). 
  Let $\alpha := \frac{\Gamma_p(K) + \Gamma_p(L)}{\Gamma_p(K)}$, $\beta := \frac{\Gamma_p(K) + \Gamma_p(L)}{\Gamma_p(L)}$.
  We define $A= \sqrt{\alpha A_KA_K^T+\beta A_LA_L^T}$ and
  $E=AB_2^d$. Then, for any \(\theta \in \R^d\),
  \begin{align*}
    h_{E}(\theta)&=\sqrt{\theta^TA^2\theta}\\
                 &=\sqrt{\alpha \theta^TA_KA_K^T\theta+\beta \theta^TA_LA_L^T\theta}\\
                 &=\sqrt{\alpha h_{E_K}(\theta)^2+\beta h_{E_L}(\theta)^2}\\
                 &\geq \frac{h_{E_L}(\theta)+h_{E_K}(\theta)}{\sqrt{\frac{1}{\alpha}+\frac{1}{\beta }}}\\
                 &=h_{E_K}(\theta)+h_{E_L}(\theta)\\
                 &=h_{E_K + E_L}(\theta),
  \end{align*}
  where the inequality follows by Cauchy-Schwarz. This means that 
  \[
    K + L \shinclu E_K + E_L \subseteq E = AB_2^d,
  \]
  and, therefore, \(\Gamma_p(K+L) \le \sqrt{\trp{p/2}(A^2)}\). We can
  then calculate, using Lemma~\ref{lm:trp-norm}, that
  \begin{align*}
    \sqrt{\tr_{p/2}(A^2)}
    &= \sqrt{ \tr_{p/2}(\alpha A_KA_K^T+\beta A_LA_L^T)}\\
    &\leq \sqrt{\alpha \tr_{p/2}(A_KA_K^T)+\beta \tr_{p/2}(A_LA_L^T)}\\
    &= \sqrt{\alpha \Gamma_p(K)^2 + \beta \Gamma_p(L)^2}\\
    &=\Gamma_{p}(K)+\Gamma_{p}(L).
  \end{align*}
  This completes the proof of the triangle inequality.
\end{proof}

The next lemma shows that shifting $K$ by a vector $v$ in the
definition of \(\Gamma_p(K)\) is not necessary when $K$ is centrally
symmetric. We then use this fact to get an alternative formulation
of $\Gamma_p(K)$ as the minimum of the $\Gamma_p$ functional over
different symmetrizations of $K$.

\begin{lemma}\label{lm:centering}
  Suppose that \(K \subseteq \R^d\) is bounded. If \(K\) is symmetric around \(0\), i.e., \(K = -K\), then
  \[
    \Gamma_p(K) = \inf\left\{\sqrt{\trp{p/2}(AA^T)}: K \subseteq AB_2^d\right\}.
  \]
  Otherwise,
  \[
    \Gamma_p(K) 
    =
    \inf\left\{\sqrt{\trp{p/2}(AA^T)}: (K+v) \cup (-K - v) \subseteq AB_2^d\right\}
    =
    \inf_{v \in \R^d} \Gamma_p((K+v) \cup (-K - v)).        
  \]
\end{lemma}
\begin{proof}
  Whenever \(K\) is symmetric around
  \(0\), and \(K \shinclu AB_2^d\), we have \(K \subseteq AB_2^d\) as
  well, which proves the first claim in the lemma. Indeed, if \(K + v \subseteq AB_2^d\) for some \(v \in
  \R^d\), then by the symmetry of \(AB_2^d\) and of \(K\) we have \[K
  - v = -K  - v = -(K+v) \subseteq -AB_2^d = AB_2^d.\] Using the
  convexity and symmetry of \(AB_2^d\), we have
  \[
    K \subseteq \frac12 (K + v) + \frac12 (K - v) \subseteq \frac12
    AB_2^d + \frac{1}{2} AB_2^d = AB_2^d,
  \]
  as we claimed.
  

  The next claim follows from  the observation
  that, for any \(d\times d\) matrix \(A\), \(K+v \subseteq AB_2^d\)
  is equivalent to \((K+v) \cup (-K - v) \subseteq
  AB_2^d\) because \(AB_2^d\) is symmetric around 0. Then
  $K \shinclu AB_2^d$ is equivalent to the existence of a vector $v
  \in \R^d$ such that \((K+v) \cup (-K - v)\subseteq
  AB_2^d\). We have
  \begin{align*}
    \Gamma_p(K)
    &= \inf \left\{\sqrt{\trp{p/2}(AA^T)}: K\shinclu
      AB_2^d\right\}\\
    &= \inf_{v \in \R^d} \inf \left\{\sqrt{\trp{p/2}(AA^T)}: K + v \subseteq AB_2^d\right\}\\
    &= \inf_{v \in \R^d} \inf\left\{\sqrt{\trp{p/2}(AA^T)}: (K+v) \cup (-K - v) \subseteq AB_2^d\right\}\\
    &=     \inf_{v \in \R^d} \Gamma_p((K+v) \cup (-K - v)),
  \end{align*}
  where the final equality follows from the first claim in the lemma,
  since $(K+v) \cup (-K - v)$ is symmetric around $0$.
\end{proof}

\subsection{Connection to Factorization}
\label{sect:fact}

In this subsection we show that, in the case when \(K = \{\pm w_1,
\ldots, \pm w_N\} \subseteq \R^d\) is a finite
symmetric set, the quantities \(\Gamma_2(K)\) and
\(\Gamma_\infty(K)\) can be equivalently formulated in terms of the
factorization norms $\gamma_F(W)$ and $\gamma_2(W)$ of the matrix \(W
:= (w_i)_{i = 1}^N\). These norms have been
studied in prior work on factorization mechanisms in differential privacy. The
\(\gamma_2\) norm is classical in functional analysis: see, e.g., the
book by Tomczak-Jaegermann~\cite{TJ-book}. It was first applied to
differential privacy implicitly in~\cite{NTZ-stoc,NTZ}, and more explicitly
in~\cite{nikolov-thesis}. The \(\gamma_F\) norm is implicit in the work on the
matrix mechanism~\cite{LiHRMM10}, and the notation we use is
from~\cite{factormech}, albeit with different normalization. 
We define natural analogs of these quantities
that correspond to \(\Gamma_p\) for any \(p \in [2,
\infty]\). Our general
formulation of Gaussian noise mechanisms thus generalizes factorization
mechanisms to more general domains and more general measures of
error. 

First we recall the definitions of the $\gamma_F$
and $\gamma_2$ factorization norms, and introduce a definition of a
family of factorization norms parameterized by $p \in [2,\infty]$ that
we later show correspond to $\Gamma_p$. 
\begin{definition}
  The \(\gamma_2\) and the \(\gamma_F\) factorization norms of a \(d
  \times N\) real matrix \(W\) are defined\footnote{In~\cite{factormech} the
    \(\gamma_F\) norm is normalized differently.} as
  \begin{align*}
    \gamma_2(W) &:= \inf\{\|A\|_{2\to\infty} \|C\|_{1\to 2}: AC =
                  W\}\\
    \gamma_F(W) &:= \inf\{\|A\|_{F} \|C\|_{1\to 2}: AC = W\}.
  \end{align*}
  More generally, we define, for \(p \in [2,\infty]\),
  \[
    \gamma_{(p)}(W) :=
    \inf\left\{\sqrt{\trp{p/2}(AA^T)} \|C\|_{1\to 2}: AC = W \right\},
  \]
  where \(\gamma_{(2)}(W) = \gamma_F(W)\) and \(\gamma_{(\infty)}(W)
  = \gamma_2(W)\). 
\end{definition}


We have the following connection between \(\Gamma_p\) and these
factorization norms.

\begin{theorem}\label{thm:Gammap-as-factorization}
  For any \(p \in [2,\infty]\), and any \(d\times N\) real matrix
  \(W\) with columns \(w_1, \ldots, w_N\), for the set \(K^{\text{sym}} := \{\pm
  w_1, \ldots, \pm w_N\}\) we have
  \[
    \Gamma_p(K^{\text{sym}}) = \gamma_{(p)}(W).
  \]
  Moreover, for the set \(K := \{w_1, \ldots, w_N\}\) we have
  \[
    \Gamma_p(K) = \inf_{v \in \R^d} \gamma_{(p)}(W + v1^T),
  \]
  where \(1\) is the \(N\)-dimensional all-ones vector.
\end{theorem}
\begin{proof}
  First we show that \(\Gamma_p(K^{\text{sym}}) \le \gamma_{(p)}(W).\) Let
  \(AC = W\) be an optimal factorization of \(W\). Since we can always
  multiply $A$ and divide $C$ by the same constant, we can assume, without loss of
  generality, that \(\|C\|_{1\to 2} = 1\) and
  \(\sqrt{\trp{p/2}(AA^T)} = \gamma_{(p)}(W)\). Furthermore, we claim
  that we can assume that
  the inner dimension $k$ of \(A\) and \(C\) is \(d\). In fact, it is
  enough to show that $k$ can be taken to be at most
  $d$, since we can then add $0$ columns to $A$ and $0$ rows to $C$ to
  ensure $k=d$.
  To see that $k\le d$ without loss of generality, let us first
  orthogonally project the rows of \(A\)
  and columns of \(C\) to the intersection of the row span of \(A\)
  and the column span of \(C\). This transformation still gives us a valid
  optimal factorization $W = AC$, since it does not change the product
  $AC$, and does not increase \(\sqrt{\trp{p/2}(AA^T)} \|C\|_{1\to
    2}\). We can now assume that the column span of $C$ equals the
  rowspan of $A$, and, therefore, both matrices have rank equal to the
  rank of $W$. After an appropriate change of
  basis, we can then make sure that the inner dimension $k$ of $A$ and $C$
  equals their rank, so that $k \le \mathrm{rank}\ W \le d$.

  Now the inequality \(\Gamma_p(K^{\text{sym}} ) \le \gamma_{(p)}(W)\) is
  straightforward. We have that \(CB_1^N
  \subseteq B_2^d\) since all columns of \(C\) have norm at most
  \(\|C\|_{1\to 2} = 1\). Therefore, \(K^{\text{sym}}  = A(CB_1^N) \subseteq
  AB_2^d\). The inequality follows.

  Next we show that \(\gamma_{(p)}(W) \le \Gamma_p(K^{\text{sym}})\). Since
  \(K^{\text{sym}}\) is centrally symmetric around \(0\), by
  Lemma~\ref{lm:centering} there exists some matrix \(A\) such that
  \(\Gamma_p(K^{\text{sym}}) = \sqrt{\trp{p/2}(AA^T)}\) and \(K^{\text{sym}}  \subseteq
  AB_2^d.\) 
  Let us set
  \(c_i\) to be some point in \(B_2^d\) such that \(Ac_i = w_i\). Such
  a point \(c_i\) exists because \(w_i \in K^{\text{sym}} \). Defining the matrix \(C :=
  (c_i)_{i=1}^N\), we then have the factorization \(W = AC\), with
  \(\|C\|_{1\to 2} = \max_{i = 1}^N\|c_i\|_2 \le 1\). 

  The claim after ``moreover'' follows from what we just proved and
  Lemma~\ref{lm:centering}.
\end{proof}

Notice that \((W_1 + W_2)B_1^N \subseteq W_1B_1^N + W_2 B_2^N\). Then,
a triangle inequality for \(\gamma_{(p)}\), as well as homogeneity, follow from
Theorem~\ref{thm:triangle}. The fact that \(\gamma_{(p)}(W) = 0\) only
if \(W = 0\) follows from the observation that \(\trp{p/2}(AA^T) = 0\)
implies the diagonal of \(AA^T\) is \(0\), which implies that \(A =
0\). This verifies that $\gamma_{(p)}$ is, indeed, a norm on matrices.

We note that, instead of defining the functionals $\Gamma_p$ on
bounded subsets of $\R^d$, we could have taken the approach of
defining the $\gamma_{(p)}$ factorization norms above, and
generalizing them to operators between (finite or infinite
dimensional) normed vector spaces. In our opinion, our approach is
more transparent for our applications, and handles non-symmetric
domains $K$ more easily.

\subsection{Duality}\label{sect:duality}

Our goal in this section is to derive a dual characterization of
\(\Gamma_p(K)\) as a maximization problem over probability
distributions on \(K\). We first carry this out for \(p=2\), and then
reduce the general case to \(p=2\). This dual characterization is
useful in the proofs of some of our later results.

Let us introduce some notation before we state our main duality
result. 
\begin{definition}
  For a compact set \(K \subseteq \R^d\), we define \(\Delta(K)\) to
  be the set of Borel regular measures supported on \(K\). 
\end{definition}

\begin{definition}
  For a probability measure \(P\) over \(\R^d\), we use the notation
  \(\cov(P)\) for the covariance matrix of \(P\), i.e.,
  \[
    \cov(P) := \E_{Y \sim P}[(Y - \E[Y])(Y - \E[Y])^T].
  \]
\end{definition}

The following theorem is our dual characterization of \(\Gamma_p(K)\)
as a problem of maximizing the covariance of probability distributions
over \(K\). It generalizes the known dual characterizations of
the \(\gamma_2\) and \(\gamma_F\) factorization norms~\cite{LeeSS08,HUU22}.
\begin{theorem}\label{thm:duality}
    Let \(K \subseteq \R^d\) be a bounded set, and let \(p \in (2,
    \infty]\).  Then, for \(q := \frac{p}{p-2}\) we have the identity
  \begin{equation}\label{eq:gammap-duality}
    \Gamma_p(K) = \sup\{\tr((D\cov(P)D)^{1/2}):D \text{ diagonal}, D \succeq 0, \trp{q}{(D^2)} = 1, P \in \Delta(K)\}.
  \end{equation}
  Moreover, if \(K\) is symmetric around \(0\) (i.e., \(K = -K\)),
  then 
  \begin{equation}\label{eq:gammap-duality-sym}
    \Gamma_p(K) = \sup\{\tr((D\E_{X\sim P}[XX^T]D)^{1/2}):D \text{ diagonal}, D \succeq 0, \trp{q}{(D^2)} = 1, P \in \Delta(K)\}.
  \end{equation}
\end{theorem}

Let us also state one of the two key lemmas used in the proof of
Theorem~\ref{thm:duality}, since it is useful in several of our lower
bound proofs. This lemma expresses \(\Gamma_p(K)\) in terms of the
\(\Gamma_2\) functional applied to rescalings of \(K\).
\begin{lemma}\label{lm:Gammap-via-Gamma2}
  Let \(K \subseteq \R^d\) be a bounded set, and let \(p \in (2,
  \infty]\). Then, for \(q := \frac{p}{p-2}\) we have the identity
  \begin{equation}\label{eq:Gammap-via-Gamma2}
    \Gamma_p(K) = \max\{\Gamma_2(DK): D \text{ diagonal}, D \succeq 0,
    \trp{q}{(D^2)} = 1\}.
  \end{equation}
\end{lemma}

Our proofs of Lemma~\ref{lm:Gammap-via-Gamma2} and
Theorem~\ref{thm:duality} use Sion's minimax theorem and the
compactness of $\Delta(K)$ in the weak* topology. For readability, we
defer the proof itself to Appendix~\ref{app:duality}.

\subsection{Computational Efficiency}
\label{sect:alg}

Next, we show that, under appropriate assumptions, \(\Gamma_p\) can
be efficiently approximated. It is known that the $\gamma_F$ and $\gamma_2$ norms
can be computed efficiently using semidefinite
programming~\cite{LeeSS08,factormech,HUU22}. Via
Theorem~\ref{thm:Gammap-as-factorization} this implies that
\(\Gamma_p(K)\) can be efficiently approximated for $p \in \{2,
\infty\}$ when \(K\) is a centrally symmetric finite set or, equivalently, a
centrally symmetric polytope in vertex representation. The same
techniques can be used to extend this result to arbitrary finite \(K\)
or arbitrary polytopes in vertex representation. Here we further
show that \(\Gamma_p(K)\) can be approximated efficiently for these
\(K\) and \emph{all} \(p \ge 2\). Moreover, we give a general
condition under which \(\Gamma_p(K)\) is efficiently approximable,
which is satisfied by several families of sets \(K\) that go far beyond
polytopes in vertex representation. 

Note that the \(\gamma_2\) norm as defined in Section~\ref{sect:fact}
is the special case the more general \(\gamma_2\) norm of linear
operators between Banach spaces when the \(d \times N\) matrix \(W\)
is treated as a linear operator from \(\ell_1^N\) to
\(\ell_\infty^d\).  Our results in this section are related to work of
Bhattiprolu, Lee, and Naor~\cite{BLN21}, who formulated the general
\(\gamma_2\) norm as a convex optimization problem, and showed that
this optimization problem can be solved efficiently given quadratic
oracles as defined in Definition~\ref{defn:oracle} below. While
similar in spirit and techniques, our results do not follow directly
from theirs, because the definition of \(\Gamma_p(K)\) allows for
non-symmetric sets \(K\) that do not correspond to the unit balls of
Banach spaces. Moreover, even for symmetric sets \(K\), \(\Gamma_p(K)\)
cannot be written as the \(\gamma_2\) of an operator, unless
\(p = \infty\).

We first prove a lemma showing that the minimization defining
\(\Gamma_p(K)\) can be restricted to positive definite matrices.

\begin{lemma}\label{lm:posdef}
  For any bounded set \(K \subseteq \R^d\), we have
    \begin{align}
    \Gamma_p(K) &=\inf\left\{\sqrt{\trp{p/2}{(M)}}: M\succ 0, K\shinclu
      \sqrt{M}B_2^d\right\}\label{eq:pd}\\
    &=\inf\left\{\sqrt{\trp{p/2}{(M)}}: v \in -\clconv K, M\succ 0, (x+v)^T M^{-1}(x+v)
      \le 1 \ \forall x \in K\right\}.\label{eq:inv}
    \end{align}
    If \(K\) is symmetric around \(0\), then \eqref{eq:pd} and
    \eqref{eq:inv} hold with, respectively, \(\shinclu\) replaced by
    \(\subseteq\), and \(v = 0\).
\end{lemma}
\begin{proof}
  We first show that
  \begin{equation}\label{eq:psd}
    \Gamma_p(K) =\inf\left\{\sqrt{\trp{p/2}(M)}: M\succeq 0, K\shinclu \sqrt{M}B_2^d\right\}.
  \end{equation}
  To see this, note that, for any matrix \(A\), the ellipsoids
  \(AB_2^d\) and \(\sqrt{AA^T}B_2^d\) are equal because \(\sqrt{AA^T}
  = AU\) for some \(d\times d\) unitary matrix \(U\), and, by
  unitarity, \(UB_2^d = B_2^d\).

  Next we show that \(M\) on the right hand side of \eqref{eq:psd} can
  also be taken to be positive definite.
  Take any \(\beta > 0\) and any \(M \succeq 0\) such that \(K \shinclu \sqrt{M}B_2^d\) and
  define \(M' := M + \beta I\succ 0\). Then \(\sqrt{M}B_2^d \subseteq
  \sqrt{M'}B_2^d\) since \(M' \succ M\), so, for any \(\theta \in \R^d\),
  \[
    h_{\sqrt{M'}B_2^d}(\theta) = \sqrt{\theta^T M'\theta}
    \ge
    \sqrt{\theta^T M\theta}
    = h_{\sqrt{M}B_2^d}(\theta).
  \]
  This  means that \(K\shinclu \sqrt{M'}B_2^d\) as well. At the same
  time, by the triangle inequality for \(\ell_{p/2}\), \(\trp{p/2}(M')
  \le \trp{p/2}(M) + \trp{p/2}(\beta I) = \trp{p/2}(M) + \beta
  d^{2/p}\). Taking \(\beta\) to \(0\) shows that the infimum in
  \eqref{eq:psd} equals the infimum in \eqref{eq:pd}.

  To prove~\eqref{eq:inv}, let us take any \(M \succ 0\) such that \(K+v\subseteq
  \sqrt{M}B_2^d\). This is equivalent to
  \(M^{-1/2} (K+v) \subseteq B_2^d\), i.e., \(\max_{x \in K}
  \sqrt{(x+v)^T M^{-1}(x+v)} \le 1\).  Moreover, a standard argument
  shows that, for any \(v \in \R^d\), if \(-v' \in \arg \min\{(y+v)^T
  M^{-1} (y+v): y \in  \clconv K\}\), then, for any \(x \in K\), \((x+v')^T M^{-1}(x+v') \le
  (x+v)^T M^{-1}(x+v)\). This inequality can be proved analogously to
   inequality (3.6)~of~\cite{jlquery} as follows. The
  gradient of \((y+v)^T M^{-1} (y+v)\) with respect to \(y\) is
  \(2M^{-1} (y+v)\), so, by the first order optimality conditions, for
  all \(x \in K\),
  \[
    \ip{x+v'}{2M^{-1}(-v' +v)} = 2(x+v')M^{-1}(v-v')\ge 0.
  \]
  Therefore,
  \begin{align*}
    (x+v)^T M^{-1}(x+v) - (x+v')^T M^{-1}(x+v')
    &=
    2(x+v')^TM^{-1}(v-v') + (v-v')^TM^{-1}(v-v')\\
    &\ge (v-v')^TM^{-1}(v-v') \ge 0.
  \end{align*}
  In summary, \(K+v \subseteq \sqrt{M}B_2^d\) is equivalent to
  \((x+v')^TM^{-1}(x+v') \le 1\) for all \(x \in K\) and some \(v' \in
  -\clconv K\).
  Substituting into \eqref{eq:pd} finishes the proof of \eqref{eq:inv}.

  The claim for \(K\) symmetric around \(0\) follows analogously using
  Lemma~\ref{lm:centering}. 
\end{proof}

The following definition underlies our condition for \(\Gamma_p(K)\)
being efficiently approximable.

\begin{definition}\label{defn:oracle}
  We say that a set \(K \subseteq \R^d\) has an \(\alpha\)-approximate
  quadratic oracle with running time \(T\) if there exists an
  algorithm that, on input a positive semidefinite \(d\times d\)
  matrix \(M\), and a vector \(v \in \R^d\), in time \(T\) outputs
  some \(\tilde{x} \in K\) such that
  \begin{equation}\label{eq:quadr-approx}
    (\tilde{x}+v)^TM(\tilde{x}+v) \ge \frac{1}{\alpha^2} \sup_{x \in K}({x}+v)^TM({x}+v).
  \end{equation}
  If \(K\) is symmetric around \(0\) (i.e., \(K = -K\)), then \eqref{eq:quadr-approx} only
  needs to hold for \(v = 0\).  
  The running time \(T = T(M,v,K)\) can be a function of \(K\) and the bit
  representation length of \(M\) and \(v\).
\end{definition}

We now state our main theorem guaranteeing that \(\Gamma_p(K)\) can be
approximated efficiently given a quadratic oracle for \(K\).

\begin{theorem}\label{thm:opt}
  Suppose that, for some \(R > 0\), the set \(K\subseteq R B_2^d\) has
  an \(\alpha\)-approximate quadratic oracle with running time
  \(T\). For any \(p \in [2, \infty]\), and any \(\beta > 0\), in time
  polynomial in \(T, d, \log(R/\beta)\), we can compute a matrix \(A\) and vector \(v\in
  \R^d\) such that \(K + v \subseteq AB_2^d\), and
  \[
    \sqrt{\trp{p/2}(AA^T)} \le \alpha\Gamma_p(K) + \beta
  \]
\end{theorem}
\begin{proof}
  Let $W$ be the affine span of $K$.
  We write \(\Gamma_p(K)\) as the optimal value of a convex
  optimization problem as
  \begin{align}
    \Gamma_p(K)^2 = &\min \trp{p/2}(M)\label{eq:cp-obj}\\
                    &\text{s.t.}\notag\\
                    &(x + v)^T M^{-1} (x+v) \le 1 \ \ \forall x \in K,\label{eq:cp-contain}\\
                    &M \succ 0, v \in RB_2^d.\label{eq:cp-psd}
  \end{align}
  This identity follows from \eqref{eq:inv} in Lemma~\ref{lm:posdef}
  since \(-\clconv K \subseteq RB_2^d\) follows from the assumption
  that \(K \subseteq RB_2^d\) and the fact that \(RB_2^d\) is closed
  and convex. Moreover, Lemma~\ref{lm:centering} shows that when \(K =
  -K\) we can fix \(v:= 0\).

  Let us verify that the optimization problem
  \eqref{eq:cp-obj}--\eqref{eq:cp-psd} is convex. The objective
  \(\trp{p/2}(M)\) is convex on positive semidefinite matrices \(M\)
  since, as we already observed, it equals the \(\ell_{p/2}\) norm of
  the diagonal entries of \(M\). For each \(x\), the constraint \((x +
  v)^T M^{-1} (x+v) \le 1\) is jointly convex in \(M\) and \(v\) by Theorem~1 of
  \cite{Ando79} (see also Section 3.1 of \cite{Carlen10}). Finally,
  the constraints \(M \succ 0\) and \(v \in R B_2^d\) are clearly
  convex.

  To finish the proof, we claim that, under the assumptions of the
  theorem, \eqref{eq:cp-obj}--\eqref{eq:cp-psd} can be solved
  approximately using the ellipsoid method. The details of the
  argument are deferred to Appendix~\ref{app:opt}.
\end{proof}

To complete this subsection, we would like to identify cases when
\(K\) has an efficient approximate quadratic oracle. Clearly, if \(K\) is
finite, we have such an oracle with running time polynomial in the
size of \(K\). This observation can be easily extended to polytopes in
vertex representation. Using Grothendieck's inequality and its
extensions, we can also show there exists approximate quadratic
oracles for the \(\ell_p\) balls for \(p \in [2,\infty]\) and their
affine images. 

The next lemma shows a few equivalent conditions for the existence of
an approximate quadratic oracle. These conditions can be useful when
implementing such an oracle for some \(K\).

\begin{lemma}\label{lm:oracle-equiv}
  The following are all equivalent:
  \begin{itemize}
  \item \(K\) has an \(\alpha\)-approximate quadratic
    oracle;
    
  \item there exists an
    algorithm that, on input a \(d\times d\)
    matrix \(A\), and a vector \(v \in \R^d\), outputs
    some \(\tilde{x} \in K\) such that
    \begin{equation}\label{eq:opnorm-approx}
      \|A(\tilde{x}+v)\|_2 \ge \frac{1}{{\alpha}}
      \sup_{x \in K}\|A({x}+v)\|_2. 
    \end{equation}
    
  \item there exists an
    algorithm that, on input a positive semidefinite \(d\times d\)
    matrix \(M\), and a vector \(v \in \R^d\), outputs
    some \(\tilde{x},\tilde{y}
    \in K\) such that
    \begin{equation}
      \label{eq:bilin-approx}
      (\tilde{x}+v)^TM(\tilde{y}+v) \ge \frac{1}{\alpha^2} \sup_{x,y \in K}({x}+v)^TM({y}+v).
    \end{equation}
  \end{itemize}
  If \(K = -K\), then the conditions only need to hold for \(v =0\).
  If one of these conditions is satisfied by an oracle with running
  time \(T\), then the other conditions are satisfied by an oracle
  with running time \(T'\) which is polynomial in \(T,d \)
  and in the bit complexity of, respectively, \(A\) or \(M\).

  Finally, if \(K\) has an \(\alpha\)-approximate quadratic
  oracle with running time \(T\), then, for any full rank matrix \(F\), the
  linear image \(FK\) has an \(\alpha\)-approximate quadratic
  oracle with running time \(T'\), where \(T'\) is polynomial in \(T, d\)
  and in the bit complexity of \(F\) and \(M\).
\end{lemma}
\begin{proof}
  The equivalence between \eqref{eq:quadr-approx} and \eqref{eq:opnorm-approx} is immediate from the
  assumption that \(M\) is positive semidefinite, since that
  assumption implies that, for any \(x \in K\), 
  \[
    (x+v)^T M (x+v) = \ip{\sqrt{M}(x+v)}{\sqrt{M}(x+v)} =
    \|\sqrt{M}(x+v)\|_2^2.
  \]
  To prove equivalence between \eqref{eq:quadr-approx} and \eqref{eq:bilin-approx}
  note that, for any \(x,y \in K\), by the Cauchy-Schwarz
  inequality we have
  \begin{align*}
    (x+v)^T M(y+v) &= \ip{\sqrt{M}(x+v)}{\sqrt{M}(y+v)}\\
    &\le \|\sqrt{M}(x+v)\|_2\|\sqrt{M}(x+v)\|_2\\
    &\le \max\{({x}+v)^TM({x}+v), ({y}+v)^TM({y}+v)\},
  \end{align*}
  where, in the last inequality we used that the geometric mean of two non-negative
  numbers is no bigger than their maximum. Therefore, the right hand
  sides of \eqref{eq:quadr-approx} and \eqref{eq:bilin-approx} are
  equal, and if \(\tilde{x}\) satisfies \eqref{eq:quadr-approx}, then
  \(\tilde{x}, \tilde{y}:=\tilde{x}\) satisfy
  \eqref{eq:bilin-approx}. In the other direction, if \(\tilde{x}\) and \(\tilde{y}\) that satisfy
  \eqref{eq:bilin-approx}, then \(\arg
  \max\{(\tilde{x}+v)^TM(\tilde{x}+v), (\tilde{y}+v)^TM(\tilde{y}+v)\}\) satisfies
  \eqref{eq:quadr-approx}. The running time guarantee for the
  equivalences is apparent from the proof.

  To show the final claim, observe that we can implement an oracle for
  \(FK\) by calling, on input \(M\) and \(v\), the oracle for \(K\)
  with input \(M:= F^T M F\) and any \(v'\) such that \(Fv' =
  v\). This is because, for any \(x \in FK\),
  \(
  (x+v)^T M (x+v) = (x' + v')^T F^T M F^T (x' + v')
  \)
  for some \(x' \in K\) such that \(Fx' = x\). 
\end{proof}

Using Theorem~\ref{thm:opt} and plugging in known results for
quadratic optimization over convex sets, we are able to show that
\(\Gamma_p(K)\) can be efficiently approximated for several rich
classes of bodies \(K\). To see the definitions of these classes, and
details about oracle implementations, refer to the papers cited in the
proof of the next theorem.

\begin{theorem}\label{thm:cases}
  Let \(p \in [2,\infty]\), and let \(\beta > 0\) be given
  constants. Suppose that \(K \subseteq RB_2^d\)  for some \(R > 0\)
  and that \(\clconv K\) contains a Euclidean ball of radius \(r > 0\) for
  some \(r < R\). Suppose also that we are given a membership oracle for \(K\) running in
  time \(T\), and \(K\) belongs to one of the
  following classes of sets in \(\R^d\):
  \begin{enumerate}
  \item finite sets \(K\), and polytopes \(K\) given in vertex
    representation, where \(T\) is at least, respectively, \(|K|\) or
    the number of vertices of \(K\);\label{it:finite}

  \item the \(\ell_p\) balls \(B_p^d := \{x \in \R^d: \|x\|_p \le
    1\}\) for \(p \in [2,\infty]\); \label{it:ellp}

  \item the unit ball of a 2-convex sign-invariant norm;\label{it:signs}

  \item the unit ball of a norm symmetric under coordinate
    permutations, provided that the dual of the norm has bounded
    cotype-2 constant; \label{it:sym}

  \item the operator norm (also known as the Schatten-\(\infty\) norm)
    unit ball $B_{op}^d$, consisting of \(d \times d\) matrices whose
    largest singular value is bounded by  \(1\); \label{it:noncomm}

  \item the unit ball of a norm on matrices that's invariant under
    unitary transformations, provided that the dual of the norm has
    bounded cotype-2 constant.\label{it:unitary}
    
  \end{enumerate}
  Then, there exists a constant \(\alpha\ge 1\) such that, in time
  polynomial in \(T, d, \log\left(\frac{R}{r\beta}\right)\), we can
  compute a matrix \(A\) and a vector \(v\in \R^d\) for which \(K + v \subseteq AB_2^d\), and
  \[
    \sqrt{\trp{p/2}(AA^T)} \le  (\alpha + \beta)\Gamma_p(K).
  \]
  The constant \(\alpha\) equals \(1\) in case~\ref{it:finite}, and may depend on the cotype constant in cases~\ref{it:sym}~and~\ref{it:unitary}, but is otherwise an
  absolute constant.  
  
  The same guarantee also holds if we replace \(K\) with \(FK\) for
  any matrix \(F\), with the running time now also polynomial in the
  bit complexity of \(F\). 
\end{theorem}
\begin{proof}
  All cases of the theorem follow from Theorem~\ref{thm:opt} after
  showing that there exists an appropriate quadratic oracle. We argue
  that such oracles exist next.
  
  In the case of finite \(K\), we can implement an exact quadratic oracle oracle by simply
  enumerating over the elements of \(K\). Since, \((x+v)^T M (x+v)\)
  is convex in \(x\) for \(M \succeq 0\), we see that for any polytope
  \(K\), \(\max_{x \in K}({x}+v)^TM({x}+v)\) is achieved at a
  vertex. Therefore, we can reduce the case of a polytope in vertex
  representation to the finite case. This establishes the case
  \ref{it:finite}.

  The other cases follow by using either formulation
  \eqref{eq:opnorm-approx} or \eqref{eq:bilin-approx} of the quadratic
  oracle, and known approximation algorithms for quadratic
  maximization. The relevant results are as follows.
  \begin{itemize}
  \item The \(\ell_\infty\) ball case \([-1,+1]^d\) is implied by
    Grothendieck's inequality~\cite{grothendieck}, whose algorithmic version was first
    given by Alon and Naor~\cite{AlonN04}. Here we only need the inequality for
    positive semidefinite \(M\), and an algorithmic version of this
    easier case was given by Nesterov~\cite{Nesterov98}.

  \item Krivine observed that the case
    of other \(\ell_p\) balls for \(p \ge 2\) follows from the
    \(\ell_\infty\) case~\cite{Krivine74}. A sketch of the simple argument is given in
    the survey by Khot and Naor~\cite{KhotNaor12}. In fact, the same argument of
    Krivine also works
    for the sign-invariant 2-convex norms, i.e.,
    case~\ref{it:signs}, as explained in~\cite{BLN21}. This was also
    discovered by Nesterov with a different, algorithmic
    proof~\cite{Nesterov98}.  

  \item The operator ball case is implied by the non-commutative
    Grothendieck inequality, first proved by Pisier~\cite{Pisier-noncomm}. An algorithmic
    version was given by Naor, Regev, and Vidick~\cite{NRV-noncomm}.

  \item For the cases \ref{it:signs}, \ref{it:sym}, and
    \ref{it:unitary}, see Section 7 of~\cite{BLN21}.
  \end{itemize}

  The statement about linear images \(FK\) of \(K\) holds by
  Theorem~\ref{thm:opt}, Lemma~\ref{lm:oracle-equiv}, and the oracles
  described above.

  Note that Theorem~\ref{thm:opt} only gives an additive approximation
  \(\beta\), rather than a multiplicative one. Nevertheless, since we
  assumed that \(\clconv K\) contains a Euclidean ball of radius
  \(r\), by
  Lemma~\ref{lm:Gammap-ellipsoid} and Theorem~\ref{thm:triangle}, we
  have
  \[
    \Gamma_p(K) = \Gamma_p(\clconv K)) \ge \Gamma_p(rB_2^d) = rd^{1/p}.
  \]
  Therefore, an additive approximation of \(\beta\) implies a
  multiplicative approximation of \(\frac{\beta}{rd^{1/p}}\). 
\end{proof}

\section{High-dimensional Lower Bound from One-dimensional Marginals}
\label{sect:high2low}

In this section we give a framework for deriving lower bounds on the
\(\ell_p\) error of an unbiased mean estimation mechanism from lower
bounds on the variance of its one-dimensional marginals. This
framework is essentially the same as the one proposed
in~\cite{factormech} for oblivious mechanisms, with some small
improvements. In particular, here we generalize the framework
in~\cite{factormech} to not necessarily symmetric domains, and to
error measured in the \(\ell_p\) norm for \(p \in [2,\infty]\). We
also give slightly different, easier proofs of some of the main
claims.

In the following, we use the notation $\cov(\mech(X))$ for the
covariance matrix of the output distribution of the mechanism $\mech$
on input dataset $X$.
The next lemma gives a lower bound on the \(\ell_p\) error in terms of
a function of the covariance matrix. 
\begin{lemma}\label{lm:err-trace}
  For any \(p \in [2,\infty]\), and any unbiased mechanism \(\mech\)
  over \(\cset \subseteq \R^d\), and any input dataset \(X \in K^n\),
  we have
  $$\E \left[{\norm{\mech (X)-{\mu(X)}}_p^2}\right]^{1/2}\geq \sqrt{\trp{p/2}(\cov(\mech(X)))}$$
\end{lemma}
\begin{proof}
  Using Jensen's inequality and the fact that the norm
  \(\|\cdot\|_{p/2}\) is convex for \(p \ge 2\), we have
\begin{align*}
    \E \left[{\norm{\mech (X)-{\mu(X)}}_p^2}\right]&=\E\left[\left(\sum_{i=1}^d{\left(\mech (X)_i-{\mu(X)}_i\right)}^{2\cdot\frac{p}{2}}\right)^{\frac{2}{p}}\right]\\
    &\geq \left(\sum_{i=1}^d{\E\left[{\left(\mech (X)_i-{\mu(X)}_i\right)^2}\right]^{\frac{p}{2}}}\right)^{\frac{2}{p}}\\
    &=\left(\sum_{i=1}^d{\var[\mech (X)_i]^{\frac{p}{2}}}\right)^{\frac{2}{p}}\\
    &=\trp{p/2}(\cov(\mech(X))).
\end{align*}
Taking the square root on both sides finishes the proof.
\end{proof}

The next lemma is key to our framework. Before we state it, we
introduce notation for the width of a set in a given direction. 
\begin{definition}
  We define the width function \(w_K:\R^d \to \R\) of a set
  \(K\subseteq \R^d\) by
  $$w_K(\theta) := \sup_{x \in K} \ip{x}{\theta} - \inf_{x \in K} \ip{x}{\theta}=
  h_K(\theta)+h_K(-\theta).$$
\end{definition}

\begin{lemma}\label{lm:cov-containment}
  Let \(c > 0\), let $K\subseteq \R^d$, let \(X \in K^n\) be a
  dataset, and let \(\mech\) be an unbiased mechanism over \(K\).
  If, for all
  $\theta\in \R^d$, $\mech$ satisfies that 
  $$\sqrt{\var[\theta^T \mech(X)]}\geq c\ w_K(\theta),$$
  then $K\shinclu \frac1c \sqrt{\cov(\mech(X))}B_2^d$.
\end{lemma}
\begin{proof}
  Let us denote \(E := \sqrt{\cov(\mech(X))}B_2^d\). Then, for each
  \(\theta \in \R^d\), we have
  \begin{align*}
    h_E(\theta) &= \norm{\sqrt{\cov(\mech(X))} \theta}_2
    = \sqrt{\theta^T \cov(\mech(X))\theta}
    = \sqrt{\var[\theta^T \mech(X)]}.
  \end{align*}
  Together with the assumption of the lemma, this means that
  \(
  h_E(\theta) \ge c\ w_K(\theta)
  \)
  for all \(\theta \in \R^d\). Let then \(v \in K\) be arbitrary, and
  note that
  \[
    h_{K-v}(\theta) = h_K(\theta) - \ip{v}{\theta}
    = \max_{x \in K}\ip{x}{\theta} - \ip{v}{\theta} \ge 0, 
  \]
  for all \(\theta \in \R^d\). Therefore,
  \[
    h_{K-v}(\theta) \le h_{K-v}(\theta) + h_{K-v}(-\theta)
    = w_{K-v}(\theta) = w_{K}(\theta).
  \]
  Then, for all $\theta \in \R^d$,
  \(
  h_{K-v}(\theta) \le \frac{1}{c} h_E(\theta) = h_{(1/c)E}(\theta)
  \)
  which is equivalent to \(K-v \subseteq \frac1c E\).
\end{proof}

Combining the two lemmas, we have the following lemma that allows us
to reduce proving our lower bounds to proving one-dimensional lower
bounds on variance.

\begin{lemma}\label{lm:1-to-d}
  Let \(c > 0\), let $K\subseteq \R^d$, let \(X \in K^n\) be a
  dataset, and let \(\mech\) be an unbiased mechanism over \(K\).
  If, for all $\theta\in \R^d$, $\mech$ satisfies that 
  $$\sqrt{\var[\theta^T \mech(X)]}\geq c\ w_K(\theta),$$
  then we have that
  \[
    {\E\left[\|\mech(X) - \mu(X)\|_p^2\right]^{1/2}}
    \ge c\ \Gamma_p(K). 
  \]
\end{lemma}
\begin{proof}
  By Lemma~\ref{lm:cov-containment}, and the definition of
  \(\Gamma_p(\cdot)\), we have
  \[
  \Gamma_p(K) \le \frac1c \sqrt{\trp{p/2}(\cov(\mech(X)))}.
  \]
  On the other hand, by Lemma~\ref{lm:err-trace},
  \[
    \sqrt{\trp{p/2}(\cov(\mech(X)))} \le \E \left[{\norm{\mech (X)-{\mu(X)}}_p^2}\right]^{1/2}.
  \]
  Combining the two inequalities and multiplying through by \(c\)
  gives the lemma.
\end{proof}

\section{Lower Bound for Pure and Concentrated Differential Privacy}

We first show that for any dataset \(X\) we can find a neighboring
dataset \(X'\) so \(\mu(X)\) and \(\mu(X')\) are far in a given
direction. The following simple geometric lemma is helpful for that goal.

\begin{lemma}\label{lm:farpoints}
  For any $\beta > 0$, any bounded \(K\subseteq \R^d\), any \(x \in K\), and any
  \(\theta \in \R^d\), there exists a point \(x' \in K\) such
  that
  \(
  |\ip{\theta}{x- x'}| \ge \frac{(1-\beta)w_K(\theta)}{2}.
  \)
\end{lemma}
\begin{proof}
  Let $x_{+}$ be such that $\ip{\theta}{x_+} \ge 
  h_K(\theta) - \frac{\beta w_K(\theta)}{2}$ and let $x_{-}$ be such that $\ip{-\theta}{x_+}
  \ge h_K(-\theta) - \frac{\beta w_K(\theta)}{2}$. 
  Thus,
  \[
    |\ip{\theta}{x-x_+}| + |\ip{\theta}{x-x_-}| \ge \ip{\theta}{x_+ -
      x_-} \ge  (1-\beta)w_K(\theta)
  \]
  Then, it must be true that either $$|{\ip{\theta}{x-x_{+}}}|\geq
\frac{(1-\beta) w_K(\theta)}{2},$$ or $$|\ip{\theta}{x-x_{-}}|\geq
\frac{(1-\beta) w_K(\theta)}{2},$$ and we can choose \(x' \in \{x_+, x_-\}\) accordingly.
\end{proof}

Next we use Lemma~\ref{lm:farpoints} to construct a neighboring
dataset \(X'\) for any dataset \(X\) so that the means of \(X\) and
\(X'\) are far in the direction of \(\theta\). 
\begin{lemma}\label{lm:neighboring}
    For any \(\beta > 0\), any bounded \(K \subseteq \R^d\), any
    \(\theta \in \R^d\), and any dataset $X \in K^n$, there exists a neighboring dataset $X'$ such that $$|\ip{\theta}{\mu(X)-\mu(X')}|\geq \frac{(1-\beta)w_K(\theta)}{2n}$$.
\end{lemma}
\begin{proof}
For any given $X$, we change only the first data point to
construct $X'$. Let $x_1\in K$ be the first data point of $X$, and,
take \(x'_1\) to be the point \(x'\) guaranteed by
Lemma~\ref{lm:farpoints} used with \(x := x_1\). Then we set \(X'
:= (x'_1, x_2 \ldots, x_n)\). 
We have
\begin{align*}
    |\ip{\theta}{\mu(X)-\mu(X')}|&=|{(\mu(\theta^TX)-\mu(\theta^TX'))}|
  =\left|\frac{1}{n}\ip{\theta}{x_1-x'_1}\right|
  \geq \frac{(1-\beta)w_K(\theta)}{2n},
\end{align*}
where we use the notation \(\theta^T X := (\theta^T x_1, \ldots,
\theta^T x_n)\), and, similarly, \(\theta^T X' := (\theta^T x'_1, \ldots,
\theta^T x_n)\). This completes the proof.
\end{proof}

Recall that, for two probability distributions \(P\) and \(Q\), defined on the same
ground set, the \(\chi^2\)-divergence between them is defined by
\[
\Chisq(P\|Q) := \E_{X \sim Q}\Braket{\braket{\frac{dP}{dQ}(X) - 1}^2}.
\]
The following lemma, bounding the \(\chi^2\)-divergence between the
output distributions of an \(\eps\)-differentially private mechanism
run on two neighboring datasets, is likely well-known. We include a brief proof
sketch, as we were unable to find a reference for the tight bound.
\begin{lemma}\label{lm:eps-to-chi2}
  Suppose that \(\mech\) is an \(\eps\)-differentially private
  mechanism, and \(X, X'\) are two neighboring datasets. Let \(P\) be
  the probability distribution of \(\mech(X)\), and \(Q\) the
  probability distribution of \(\mech(X')\). Then $\Chisq(P\|Q) \leq
  e^{-\eps}(e^\eps-1)^2$. 
\end{lemma}
\begin{proof}
  Let \(P', Q'\) be the probability distributions on \(\{0,1\}\) such that
  \begin{align*}
    &\Pr_{P'}[0] = \Pr_{Q'}[1] = \frac{e^\eps}{1+e^\eps},
    &\Pr_{P'}[1] = \Pr_{Q'}[0] = \frac{1}{1+e^\eps}.
  \end{align*}
  (In other words, \(P',Q'\) are the probability distributions
  defining the randomized response mechanism.)
  Let \(A\) be a random bit distributed according to \(P'\), and \(B\) a
  random bit distributed according to \(Q'\). Then, there exists a
  randomized function (i.e., a Markov kernel) \({F}\) such that \({F}(A)\) is
  distributed identically to \(P\), and \({F}(B)\) is
  distributed identically to \(Q\). This follows from a theorem of
  Blackwell, as observed in~\cite{KOV17}. By the data processing inequality
  for \(\chi^2\)-divergence, this implies
  \(
    \Chisq(P\|Q) \le \Chisq(P'\|Q').
  \)
  It remains to compute the right hand side:
  \[
    \Chisq(P'\|Q') = \frac{1}{1+e^\eps} \cdot (e^\eps - 1)^2 +
    \frac{e^\eps}{1+e^\eps}\cdot (e^{-\eps}-1)^2
    = e^{-\eps}(e^\eps-1)^2.
  \]
  This completes the proof.
\end{proof}

For the our one dimensional lower bounds, we use the classical Hammersley-Chapman-Robins
bound~\cite{H50,CR51}, stated in the next lemma.
\begin{lemma}\label{lm:hcr}
  For any two probability distributions \(P\) and \(Q\) over the
  reals, and for random variables \(X, Y\) distributed, respectively, according to
  \(P\) and \(Q\), we have
  \[
    \sqrt{\var(Y)} \ge \frac{|\E[X] - \E[Y]|}{\sqrt{\Chisq(P\|Q)}}.
  \]
\end{lemma}

We first state our one dimensional lower bound for pure differential privacy.
\begin{lemma}\label{lm:puredp-var}
Let \(K \subseteq \R^d\), let $X \in K^n$ be a
dataset, and \(\mech\) be an unbiased \(\eps\)-differentially private
mechanism over $K$. Then for any $\theta \in \R^d$,  $$\sqrt{\var{[\theta^T \mech(X)]}}\gtrsim\frac{w_K(\theta)}{ne^{-0.5\varepsilon}(e^\varepsilon-1)} $$
\end{lemma}
\begin{proof}
Let $c < 2$, and let $X'$ be any dataset that is neighboring with $X$ and satisfies
\[|{\theta^T({\mu(X)}-f(X'))}|\geq \frac{c\ w_K(\theta)}{n}.\] Such a
dataset exists by Lemma~\ref{lm:neighboring}. Let $Q$ be the
probability distribution of $\theta^T\mech(X)$ and $P$ the probability
distribution $\theta^T\mech(X')$.
By Lemma~\ref{lm:hcr}, we have
\[
  \var{[\theta^T \mech(X)]}\geq
  \frac{(\theta^T\E[M(X)]-\theta^T\E[M(X')])^2}{\Chisq(P\|Q)}.
\]
By Lemma~\ref{lm:eps-to-chi2}, 
\(\Chisq(P\|Q) \le   e^{-\eps}(e^\eps-1)^2\), and, 
thus 
$$\sqrt{\var{[\theta^T
    \mech(X)]}}\gtrsim\frac{w_K(\theta)}{ne^{-0.5\varepsilon}(e^\varepsilon-1)},$$
as we needed to prove.
\end{proof}

Combining Lemmas~\ref{lm:1-to-d}~and~\ref{lm:puredp-var} immediately
gives the following theorem. 
\begin{theorem}\label{thm:PDP-bound}
  For any \(\eps\)-differentially private mechanism \(\mech\) that is
  unbiased over \(K\subseteq \R^d\),   and any dataset
  \(X \in K^n\), its error is
  bounded from below as
$$\sqrt{\E \left[{\norm{\mech (X)-{\mu(X)}}_p^2}\right]}\gtrsim \frac{\Gamma_p(K)}{ne^{-0.5\varepsilon}(e^\varepsilon-1)}.$$    
\end{theorem}

Note that Theorem~\ref{thm:PDP-bound} is usually not tight. For
example, when \(K = B_2^d\), every \(\eps\)-differentially private
mechanism has \(\ell_2\) error \(\gtrsim \frac{d}{\eps n}\), whereas
\(\Gamma_2(B_2^d) = \sqrt{d}\). Nevertheless, we record the theorem
here for ease of future reference. Moreover, Lemma~\ref{lm:puredp-var}
will be used in deriving tight lower bounds on the error of
approximately and also of locally differentially private mechanisms.

We can also get a similar lower bound for zCDP, which turns out to be
tight for small \(\eps\).
\begin{theorem}\label{thm:zCDP-bound}
    For any \(\frac{\eps^2}{2}\)-zCDP mechanism \(\mech\) that is
  unbiased over \(K\subseteq \R^d\),  and any dataset
  \(X \in K^n\), its error is
  bounded below as
$$\sqrt{\E \left[{\norm{\mech (X)-{\mu(X)}}_p^2}\right]}\gtrsim \frac{\Gamma_p(K)}{n\sqrt{e^{\eps^2}-1}}.$$    
\end{theorem}
\begin{proof}
  Let us denote by $Q$ the probability distribution of
  $\theta^T\mech(X)$ and by $P$ the probability distribution of
  $\theta^T\mech(X')$, as in the proof of Lemma~\ref{lm:puredp-var}. 
  From the definition of zCDP, the \(2\)-Renyi divergence of of \(P\)
  and \(Q\) is bounded as \(D_2(P\|Q) \le \eps^2\). We then use the
  relationship between the 2-Renyi divergence and the \(\chi^2\)-divergence to write
  \[\Chisq(P\|Q) = 2^{D_2(P\|Q)} - 1 \le e^{\eps^2}-1.\]
  The rest of the proof is analogous to the proof of
  Theorem~\ref{thm:PDP-bound} using Lemma~\ref{lm:puredp-var}.
\end{proof}
\section{Lower Bound for Approximate Differential Privacy}

Our lower bounds for approximate differential privacy do not follow
directly from the Hammersley-Chapman-Robins bound, because the
probability distributions of \(\mech(X)\) and \(\mech(X')\), for two
neighboring datasets \(X\) and \(X'\), and an \((\eps,
\delta)\)-differentially private mechanism \(\mech\), may not have the
same support. For this reason, the \(\Chisq\)-divergence between the
distributions can be infinite. This leads to some complications, both
for the one-dimensional, and for the higher-dimensional lower bounds,
presented in this section. 

\subsection{One-dimensional Lower Bound}

Let us introduce notation for the subset of  the ground set where
the ratio of densities is small. In our context, this will be the subset of
possible outputs of a mechanism for which the mechanism satisfies pure
differential privacy for a pair of neighboring inputs.

\begin{definition}
  For two probability distribution \(P\) and \(Q\) over the same
  ground set \(\Omega\), we define
  \[
    S_{P,Q,\eps} := \left\{\omega \in \Omega: e^{-\eps} \le \frac{dP}{dQ}(\omega) \le e^\eps
    \right\}.
  \]
\end{definition}

We restate Case 2 of Lemma~3.3 from
\cite{kasiviswanathan2008thesemantics} here. The lemma captures the
fact that an approximately differentially private mechanism is
``purely differentially private with high probability''. 
\begin{lemma}\label{lm:delta'}
  Let $\mech$ be an $(\eps,\delta)$-differentially private mechanism,
  let \(X,X'\) be neighboring datasets, and define \(P\) to be
  the probability distribution of \(\mech(X)\), and \(Q\) to be the
  probability distribution of \(\mech(X')\). Then,
  $$\max\left\{\Pr\left[\mech(X) \not\in S_{P,Q,2\eps} \right], \Pr\left[\mech(X') \not\in S_{P,Q,2\eps} \right]\right\}\leq \delta':= \frac{2\delta}{1-e^{-\eps}}.$$
\end{lemma}

The main challenge in applying our techniques to approximate differential
privacy is that the $\Chisq$-divergence between the output
distributions of an \((\eps, \delta)\)-differentially private
mechanism run on two neighboring datasets is not necessarily
bounded. To get around this issue, we modify one of the two
distributions slightly, so that we can fall back on bounds on the
$\Chisq$-divergence for pure differential privacy. Here we will use
Lemma~\ref{lm:delta'} crucially.

\begin{lemma}\label{lm:phat}
Let $\mech$ be an $(\eps, \delta)$-differentially private mechanism
with range \(\R^d\),
let $X$ be an arbitrary dataset, and let $X'$ be any dataset that is
neighboring with $X$. 
Let \(Q\) be the probability distribution of
$\theta^T\mech(X)$, and $P$ the probability distribution of
$\theta^T\mech(X')$. Define $S:=S_{P,Q,2\eps}$. There exists a
probability distribution $\hat{P}$ s.t.
\begin{align}
\left\lvert\log{\frac{d\hat{P}}{dQ}(y)}\right\rvert&\leq
2\eps-\log{(1-\delta')} \indent \text{for any $y$ in the range of
                                                     \(\mech\)}\label{eq:phat-dp}\\
\lvert \E_{Y \sim \hat P}[Y]-\E_{Y\sim Q}[Y]]\rvert&= \left\lvert
                                                  C_{\delta}\E_{Y\sim
                                                  P}[Y\mathbf{1}\{Y\in
                                                  S\}]-\E_{Y\sim Q}[Y\mathbf{1}\{Y\in S\}]\right\rvert  \label{eq:phat-exp}
\end{align}
where $\delta'$ is as in Lemma~\ref{lm:delta'}, and $C_\delta
:=\frac{Q(S)}{P(S)}$.
\end{lemma}
\begin{proof}
  The distribution \(\hat P\) is defined, for any (measurable) subset \(T\) of the
  range of \(\mech\), by
  \[
    \hat{P}(T):=C_\delta P(T \cap S) + Q(T \setminus S).
  \]
Then it is easy to verify that $\hat{P}$ is a valid probability
measure. This construction gives us, for any \(y\) in the range of \(\mech\),
\begin{equation*}
\frac{d\hat{P}}{dQ}(y) = \begin{cases}
C_\delta \frac{dP}{dQ}(y) &\text{$y\in S$}\\
1 &\text{otherwise}
\end{cases}.
\end{equation*}
We have the following inequality, which we prove later.
\begin{align}
  1-\delta' &\leq C_\delta \leq \frac{1}{1-\delta'}, \label{eq:Cdelta-1}
\end{align}
By \eqref{eq:Cdelta-1}, we know that for $y\in S$ $$\frac{1-\delta'}{e^{2\eps}}\leq C_\delta\frac{dP}{dQ}(y)\leq \frac{e^{2\eps}}{1-\delta'}.$$
This proves \eqref{eq:phat-dp}. Next we prove
\eqref{eq:phat-exp}. Since $\hat{P}$ and $Q$ agree outside \(S\), we have
\begin{align*}
    \lvert \E_{Y\sim\hat P}[Y]-\E_{Y\sim Q}[Y]\rvert&=\left\lvert
                                                      \int_S y\ d\hat{P}(y) - \int_S y\ dQ(y)\right\rvert\\
    &=\left\lvert C_{\delta} \int_S y\ d{P}(y) - \int_S y\ dQ(y)\right\rvert\\
    &=\left\lvert C_{\delta}\E_{Y\sim P}[Y\mathbf{1}\{Y\in S\}]-\E_{Y\sim Q}[Y\mathbf{1}\{Y\in S\}]\right\rvert.
\end{align*}
It only remains to prove \eqref{eq:Cdelta-1}. 
By Lemma~\ref{lm:delta'}, $Q(S) \geq 1-\delta'$, $P(S)\geq
1-\delta'$. Since both these probabilities are also at most $1$, we know
that \(C_\delta = \frac{Q(S)}{P(S)}\) satisfies
$1-\delta '\leq C_\delta \leq \frac{1}{1-\delta'}.$
This completes the proof of the lemma.
\end{proof}

Combining Lemma~\ref{lm:eps-to-chi2} and Lemma~\ref{lm:phat}, we get
the following bound on the  $\Chisq$-divergence.
\begin{lemma}\label{lm:epsdelta-to-chi2}
  Let \(\hat{P}\) and \(Q\) be as in Lemma~\ref{lm:phat}. Then we have
$$\Chisq(\hat P\|Q)\leq {\hat{\eps}}^{\, 2} := e^{-(2\varepsilon-\log{(1-\delta'))}}(e^{2\varepsilon-\log{(1-\delta')}}-1)^2.$$
\end{lemma}
Now we can prove a lower bound on the variance of an approximately
differentially private mechanism in a given direction. 
\begin{lemma}\label{lm:ADP-1d-bound}
  Let \(c > 0\) be a small enough absolute constant, and let $\mech$ be
  $(\eps, \delta)$-differentially private mechanism that is unbiased
  over \(K \subseteq \R^d\). Define \(\hat \eps\) and
  \(\delta'\) as in Lemmas~\ref{lm:epsdelta-to-chi2}~and~\ref{lm:delta'}. If \(\delta' \le
  \frac{c}{n}\), then for every
  $\theta \in \R^d$, and for every dataset $X \in K^n$,
  either
  $$\sqrt{\var{[\theta^T
      \mech(X)]}}\gtrsim\frac{w_K(\theta)}{n\hat{\eps}},$$ or there
  exist some $X'$ neighboring with $X$ such
  that
  $$\sqrt{\var{[\theta^T
      \mech(X')]}}\gtrsim\frac{w_K(\theta)}{n\sqrt{\delta'}}.$$
\end{lemma}
\begin{proof}
  For this proof, we will assume that \(0 \in K\). This comes without
  loss of generality, since we can shift both \(K\) and the output of
  \(\mech\) by the same vector, without affecting the privacy
  guarantees of \(\mech\) or the variance in the direction of
  \(\theta\). This allows us to assume that, for any \(x
  \in K\), \(|\ip{\theta}{x}| \le \max\{h_K(\theta), h_K(-\theta)\}
  \le w_K(\theta)\), since $h_K(\theta) \ge 0$ and $h_K(-\theta) \ge 0$.
  
  Let \(X'\) then be a neighboring dataset to \(X\) such that
  $\abs{\ip{\theta}{\mu(X)-\mu(X')}}\geq \frac{w_K(\theta)}{3n}$ (which
  we know exists by Lemma~\ref{lm:neighboring}). Let us use the
  notation from Lemma~\ref{lm:phat}, and let us abbreviate, for
  any function \(f\), \(\E_{Y\sim P}[f(Y)]\) to \(\E_{P}[f(Y)]\), and,
  similarly, \(\E_{Y\sim \hat P}[f(Y)]\) to \(\E_{\hat P}[f(Y)]\), and
  \(\E_{Y\sim Q}[f(Y)]\) to \(\E_{Q}[f(Y)]\). 
  For the distribution
  \(\hat P\) given by Lemma~\ref{lm:phat}, we have
  \begin{align}
    \lvert \E_{\hat P}[Y]-\E_{ Q}[Y]]\rvert
    &= \left\lvert C_{\delta}\E_{ P}[Y\mathbf{1}\{Y\in
      S\}]-\E_{ Q}[Y\mathbf{1}\{Y\in S\}]\right\rvert\notag\\
    &\geq \left\lvert C_{\delta}\E_{ P}[Y]-\E_{
      Q}[Y]\right\rvert-\left\lvert C_{\delta}\E_{
      P}[Y\mathbf{1}\{Y\notin S\}]\right\rvert-\left\lvert\E_{
      Q}[Y\mathbf{1}\{Y\notin S\}]\right\rvert\notag\\
    &\geq \left\lvert \E_P[y]-\E_Q[y]\right\rvert-\lvert (1-C_\delta) \E_P[y]\rvert - \left\lvert C_{\delta}\E_{
      P}[Y\mathbf{1}\{Y\notin S\}]\right\rvert-\left\lvert\E_{
      Q}[Y\mathbf{1}\{Y\notin S\}]\right\rvert\notag\\
    &\geq \frac{w_K(\theta)}{3n}-\lvert (1-C_\delta) \E_P[Y]\rvert - \left\lvert C_{\delta}\E_{
      P}[Y\mathbf{1}\{Y\notin S\}]\right\rvert-\left\lvert\E_{
      Q}[Y\mathbf{1}\{Y\notin S\}]\right\rvert.\label{eq:ADP-1d-1}
  \end{align} 
  We proceed to bound the last three terms. For the first term we have
  \begin{equation}
    \label{eq:ADP-1d-exp}
    \lvert (1-C_\delta) \E_P[Y]\rvert
    \le 2\delta' w_{K}(\theta),
  \end{equation}
  where we used that \(\mech\) is unbiased, and, therefore
  \[
    |\E_P[Y]| =  |\theta^T \E[\mech(X)]| = |\theta^T \mu(X)| \le
    w_K(\theta),
  \]
  and we also used the inequality
  \[
    \lvert 1-C_\delta \rvert \leq \frac{\delta'}{1-\delta'}\leq 2\delta'
  \]
  which follows from \eqref{eq:Cdelta-1} for \(\delta' \le \frac12\).

  Next we bound
  \begin{align}
    \left\lvert\E_{ P}[Y\mathbf{1}\{Y\notin S\}]\right\rvert
    &\le \lvert\E_{ P}[(Y-\E_P[Y])\mathbf{1}\{Y\notin S\}]\rvert
      + \lvert\E_P[Y]\rvert (1-P(S))\notag\\
    &\le \sqrt{\E_{ P}[(Y-\E_P[Y])^2](1 -P(S))}  +\lvert\E_P[Y]\rvert (1-P(S))\notag\\
    &\le \sqrt{\delta' \var[\theta^T \mech(X')]} + \delta' w_K(\theta),\label{eq:ADP-1d-CS-P}
  \end{align}
  where in the second line we used the Cauchy-Schwarz
  inequality. Analogously, we bound
  \begin{equation}
    \label{eq:ADP-1d-CS-Q}
    \left\lvert\E_{ Q}[Y\mathbf{1}\{Y\notin S\}]\right\rvert 
    \le \sqrt{\delta' \var[\theta^T \mech(X)]} + \delta' w_K(\theta).
  \end{equation}
  Combining 
  \eqref{eq:ADP-1d-1},\eqref{eq:ADP-1d-exp},\eqref{eq:ADP-1d-CS-P},
  and \eqref{eq:ADP-1d-CS-Q}, we get
  \begin{align*}
    \lvert \E_{\hat P}[Y]-\E_{ Q}[Y]]\rvert
    &\ge
    \left(\frac{1}{3n} - 4\delta'\right)w_K(\theta) - 2\sqrt{\delta'
      \var[\theta^T \mech(X')]} - \sqrt{\delta' \var[\theta^T  \mech(X)]}\\
    &\ge \frac{w_K(\theta)}{4n}  - 2\sqrt{\delta'
    \var[\theta^T \mech(X')]} - \sqrt{\delta' \var[\theta^T \mech(X)]},
  \end{align*} 
  where we used the assumption that \(\delta' \le \frac{c}{n}\) for a
  small enough \(c > 0\). Then, if both \(\sqrt{\delta'
    \var[\theta^T \mech(X)]}\) and \(\sqrt{\delta'
    \var[\theta^T \mech(X')]}\) are bounded by \(\frac{c\,
    w_K(\theta)}{n}\) for a small enough \(c > 0\), we have that \(\lvert \E_{\hat P}[Y]-\E_{
    Q}[Y]]\rvert \gtrsim \frac{w_K}{n}\), and the first lower bound in
  the lemma follows from  Lemmas~\ref{lm:hcr} and \ref{lm:epsdelta-to-chi2}. Otherwise, the second lower bound in the lemma holds. 
\end{proof}

\subsection{High-dimensional Lower Bound for $\ell_2$}

We first state a lower bound on \(\ell_2\) error that is nearly tight
when \(\delta\) is small with respect to the minimum width of the
domain \(K\). Then we will show that we can always ensure the minimum
width is not too small.

\begin{lemma}\label{lm:ADP-minwidth}
  Let \(c > 0\) be a small enough absolute constant, and let $\mech$ be
  $(\eps, \delta)$-differentially private mechanism that is unbiased
  over \(K \subseteq \R^d\). Define \(\hat \eps\) and
  \(\delta'\) as in Lemmas~\ref{lm:epsdelta-to-chi2}~and~\ref{lm:delta'}. If \(\delta' \le
  \frac{c}{n}\), and \(K\) satisfies  $\min_{\theta\in \R^d:
    \norm{\theta}_2 = 1}{w_{K}(\theta)}\geq w_0$, then for every
  dataset $X \in K^n$, either $$\sqrt{\E \left[{\norm{\mech
        (X)-{\mu(X)}}_2^2}\right]}\gtrsim \frac{\Gamma_2(K)}{n\hat{\eps}},$$
  or there exists a neighboring dataset $X'$ to $X$ s.t. 
    $$\sqrt{\E \left[{\norm{\mech (X')-{\mu(X')}}_2^2}\right]}\gtrsim \frac{w_0}{n\sqrt{\delta'}}.$$
\end{lemma}
\begin{proof}

By Lemma~\ref{lm:ADP-1d-bound}, one of the following two cases will
hold:

\textbf{Case 1.} For all $\theta\in \R^d,\norm{\theta}_2=1$, $\sqrt{\var{[\theta^T
    \mech(X)]}}\gtrsim\frac{w_K(\theta)}{n\hat{\eps}}$. Then, by Lemma~\ref{lm:1-to-d}, \[\E
\left[{\norm{\mech (X)-{\mu(X)}}_2}\right]\gtrsim
\frac{\Gamma(K)}{n\hat{\eps}}.\]

\textbf{Case 2.} There exists a $\theta^*\in \R^d,\norm{\theta^*}_2=1$
such that there exists some $X'$ neighboring to $X$ for which $\sqrt{\var{[\theta^{*T}
    \mech(X')]}}\gtrsim\frac{w_K(\theta^*)}{n\sqrt{\delta'}}$. Then, by
the Cauchy-Schwarz inequality, we have
\begin{align*}
  \sqrt{\E \left[{\norm{\mech (X')-{\mu(X')}}_2^2}\right]}
    &\geq \sqrt{\E\left[{{\ip{\theta^{*}}{\mech (X')-{\mu(X')}}}^2}\right]}\\
    &=\sqrt{\var{[\theta^{*T} \mech(X')]}}\\
    &\gtrsim \frac{w_K(\theta^*)}{n\sqrt{\delta'}}
    \geq \frac{w_0}{n\sqrt{\delta'}}.
\end{align*}

This completes the proof.
\end{proof}

In general, the minimum width of \(K\) can be \(0\) even
of \(\Gamma_2(K)\) is large. The next lemma shows that, nevertheless,
any \(K\) has a projection \(P(K)\) for which the minimum width
(within the image of \(P\))
and \(\Gamma_2(P(K))\) are within a factor linear in the dimension, and
\(\Gamma_2(P(K))\) is comparable to \(\Gamma_2(K)\). 
\begin{lemma}\label{lm:proj}
  For any \(K \subseteq \R^d\), there exists an orthogonal
  projection $P:\R^d\rightarrow \R^d$, such that $P(K)$ satisfies both
  of the following two conditions:
  \begin{align*}
    \Gamma_2(P(K))&\geq   \frac{\Gamma_2(K)}{2}\\
    \min_{\theta\in \Image(P):\norm{\theta}_2 = 1}{w_{P(K)}(\theta)}&\geq \frac{\Gamma_2(K)}{2d}
  \end{align*}
  Above, \(\Image(P)\) is the image of \(P\).
\end{lemma}
\begin{proof}
Let \(K_0 := K\), and consider the following procedure.
Set, initially, \(i := 1\). While there is a direction $\theta_i,
\norm{\theta_i}_2 = 1$ which is
orthogonal to \(\theta_1, \ldots, \theta_{i-1}\) (if \(i > 1\)), and is such
that $w_{K_{i-1}}(\theta_i)<\frac{\Gamma_2(K)}{2d}$, we set \(K_i\) to
be the projection of \(K_{i-1}\) orthogonal to \(\theta_i\), and set
\(i := i+1\). Continue until no such direction can be found, or until
we have made \(d\) projections, after which \(K_d\) is a
point. Suppose that this procedure terminates after \(k\le d\) projections.

Let $P(K) := K_k$ be the new set at the end of the procedure. \(P\) is
the orthogonal projection onto the subspace orthogonal to the span of
$\theta_1, \ldots, \theta_k$. From the construction it is clear that
for any \(\theta\) in the range of \(P\) such that $\|\theta\|_2 =1$, \(w_{P(K)}(\theta) \ge
\frac{\Gamma_2(K)}{2d}\), or the procedure would not have terminated. It remains to analyze \(\Gamma_2(P(K))\),
and here we use the triangle inequality for \(\Gamma_2\). Define the segment \(L_i := \left[-\frac{w_{K_{i-1}}(\theta_i)}{2}\theta_i,
\frac{w_{K_{i-1}}(\theta_i)}{2}\theta_i\right]\). Then \(K_{i-1} \shinclu K_{i} + L_i\), and, by
induction, we have
\(
K \shinclu K_k + \sum_{i=1}^k L_i.
\)
By the choice of \(\theta_i\),
\[
  \Gamma_2(L_i) \le w_{K_{i-1}}(\theta_i) \le \frac{\Gamma_2(K)}{2d},
\] and, using the monotonicity and triangle inequality properties from
Theorem~\ref{thm:triangle}, we have
\[
  \Gamma_2(K) \le \Gamma_2(K_k) + \sum_{i=1}^k \Gamma_2(L_i)
  \le \Gamma_2(P(K)) + \frac{\Gamma_2(K)}{2}.
\]
Rearranging gives us that \(\Gamma_2(P(K)) \ge
\frac{\Gamma_2(K)}{2}\). 
\end{proof}

Combining these two lemmas above we can prove a lower bound on the
\(\ell_2\) error for general $K$ that is tight as long as \(\delta\)
is sufficiently small with respect to \(d\).
\begin{theorem} \label{thm:main-ell2}
Let \(c > 0\) be a small enough absolute constant, and let $\mech$ be an
  $(\eps, \delta)$-differentially private mechanism that is unbiased
  over \(K \subseteq \R^d\). Define \(\hat \eps\) and
  \(\delta'\) as in Lemmas~\ref{lm:epsdelta-to-chi2}~and~\ref{lm:delta'}. If \(\delta' \le
  \frac{c}{n}\), then for every dataset \(X \in K^n\), either 
  $$\sqrt{\E \left[{\norm{\mech (X)-{\mu(X)}}_2^2}\right]}\gtrsim \frac{\Gamma_2(K)}{n\hat{\eps}}$$
  or there exists a dataset $X'$ neighboring with $X$ s.t. 
  $$\sqrt{\E \left[{\norm{\mech (X')-{\mu(X')}}_2^2}\right]}\gtrsim\frac{\Gamma_2(K)}{\sqrt{\delta'}nd}$$
\end{theorem}
\begin{proof}
  The key observation is that we can define an $(\eps, \delta)$-differentially
  private mechanism that's unbiased over \(P(K)\) using \(\mech\), so
  that the \(\ell_2\) error does not increase. To do so, we can fix,
  for any \(x \in P(K)\) a preimage \(f(x)\) so that \(P(f(x)) =
  x\). Then we apply \(f\) pointwise to any dataset \(\tilde{X} := (\tilde{x}_1,
  \ldots, \tilde{x}_n) \in (P(K))^n\) to get \(f(\tilde{X}) := (f(\tilde{x}_1), \ldots,
  f(\tilde{x}_n))\) in \(K^n\) so that \(P(f(\tilde{X})):= (P(f(\tilde{x}_1)), \ldots,
  P(f(\tilde{x}_n))) = \tilde{X}\).
  Moreover, for a fixed dataset \(X \in K^n\), we can make sure
  that \(f(P(X)) = X\). Then, given \(\mech\), we define
  \(\mech'(\tilde{X}) := P(\mech(f(\tilde{X})))\). Since \(f\) maps
  neighboring datasets to neighboring datasets, and \(\mech'\) is a
  postprocessing of \(\mech(f(\tilde{X}))\), \(\mech'\) is \((\eps,
  \delta)\)-differentially private.

  Because
  orthogonal projection does not increase the \(\ell_2\) norm, and
  since we ensured \(f(P(X)) = X\), we have
  \[
    \sqrt{\E \left[{\norm{\mech(X)-{\mu(X)}}_2^2}\right]}
    \ge
    \sqrt{\E \left[{\norm{P(\mech (X))-{P(\mu(X))}}_2^2}\right]}
    =
    \sqrt{\E \left[{\norm{\mech'(P(X)))-{\mu(P(X))}}_2^2}\right]}.
  \]
  An analogous analysis works for a dataset \(X'\) that is neighboring
  to \(X\).
  
  It is clear that $\Image(P)$ is isometric with $\R^{d-k}$, both
  endowed with the \(\ell_2\) metric. The theorem then follows from
  Lemmas~\ref{lm:ADP-minwidth}~and~\ref{lm:proj}. 
\end{proof}

\subsection{High-dimensional Lower Bound for $\ell_p$, $p > 2$}

To prove a lower bound for the non-Euclidean case \(\ell_p\), \(p >
2\), we reduce to the \(\ell_2\) case.  We do so via Lemma~\ref{lm:Gammap-via-Gamma2}.

\begin{theorem}\label{thm:main-ellp}
  Let \(c > 0\) be a small enough absolute constant, let \(p \in [2,
  \infty]\), and let $\mech$ be an
  $(\eps, \delta)$-differentially private mechanism that is unbiased
  over \(K \subseteq \R^d\). Define \(\hat \eps\) and
  \(\delta'\) as in Lemmas~\ref{lm:epsdelta-to-chi2}~and~\ref{lm:delta'}. If \(\delta' \le
  \frac{c}{n}\), then for every dataset \(X \in K^n\), either 
  $$\sqrt{\E \left[{\norm{\mech (X)-{\mu(X)}}_p^2}\right]}\gtrsim \frac{\Gamma_p(K)}{n\hat{\eps}}$$
  or there exists a dataset $X'$ neighboring with $X$ s.t. 
  $$\sqrt{\E \left[{\norm{\mech (X')-{\mu(X')}}_p^2}\right]}\gtrsim\frac{\Gamma_p(K)}{\sqrt{\delta'}nd}$$
\end{theorem}
\begin{proof}
  Let \(D\) achieve
  the maximum on the right hand side of equation \eqref{eq:Gammap-via-Gamma2}
  in Lemma~\ref{lm:Gammap-via-Gamma2}. Similarly to the proof of
  Theorem~\ref{thm:main-ell2}, we will construct a mechanism
  \(\mech'\) over \(DK\) from \(\mech\). We can define a function \(f:DK \to
  K\) such that \(Df(x) = x\) for any \(x \in DK\), and extend it to datasets \(\tilde{X} =
  (\tilde{x}_1, \ldots, \tilde{x}_n)\) by
  defining \(f(\tilde{X}) := (f(\tilde{x}_1), \ldots, f(\tilde{x}_n))\). We can also make
  sure that \(f(DX) = X\) for the dataset \(X\) in the statement of
  the theorem. Then we define the \((\eps, \delta)\)-differentially
  private mechanism \(\mech'(\tilde{X}) := D\mech(f(\tilde{X}))\).

  On the one hand, since \(\trp{q}(D^2) \le 1\) for \(q :=
  \frac{p}{p-2}\), we have, by H\"older's inequality, that
  \[
    \E \left[\norm{\mech' (DX)-\mu(DX)}_2^2\right]
    =
    \E \left[\norm{D(\mech(X)-\mu(X))}_2^2\right]
    \le
    \E \norm{\mech(X) - \mu(X)}_p^2.
  \]
  Analogously, for any dataset \(X'' \in (DK)^n\) that is neighboring to
  \(DX\), the dataset \(X' := f(X'')\), which is neighboring to
  \(X\), satisfies
  \[
    \E \left[\norm{\mech' (X'')-\mu(X'')}_2^2\right]
    =
    \E \left[\norm{D(\mech(X')-\mu(X'))}_2^2\right]
    \le
    \E \norm{\mech(X') - \mu(X')}_p^2.
  \]
  On the other hand, we have that \(\Gamma_2(DK) =
  \Gamma_p(K)\). Now, applying Theorem~\ref{thm:main-ell2} to \(DK\),
  \(\mech'\), and the dataset \(DX\) finishes the proof. 
\end{proof}

Theorem~\ref{thm:main} in the Introduction follows from
Theorems~\ref{thm:main-ell2}~and~\ref{thm:main-ellp}.

\subsection{A Counterexample}
\label{sect:counterexample}

The final theorem of the section shows why we cannot prove a lower
bound on the error for every possible input for approximately
differentially private mechanisms. 
\begin{theorem}
    For any $\delta$, any \(K \subseteq \R^d\), and any \(X_0 \in K^n\), there exists a
    $(0, \delta)$-differentially private mechanism that is unbiased
    over \(K\) and  achieves $0$ error on \(X_0\).
\end{theorem}
\begin{proof}
  We define the following mechanism:
    \begin{equation*}
\mech_{X_0}(X) = \begin{cases}
\mu(X_0) &\text{with probability $1-{\delta}$}\\
\frac{\mu(X)-(1-\delta)\mu(X_0))}{\delta} &\text{with probability ${\delta}$}
\end{cases}
\end{equation*}
It is clear that $\mech_{X_0}(X_0)=\mu(X_0)$, thus it achieves $0$ error at input $X_0$.\\
It is unbiased since $$\mu(X_0)(1-\delta)+\frac{\mu(X)-(1-{\delta})\mu(X_0)}{\delta}\cdot {\delta}=\mu(X)$$
It satisfies $(0, \delta)$-differential privacy since for any two datasets $X, X'$, the
total variation distance between \(\mech_{X_0}(X)\) and
\(\mech_{X_1}(X)\) is easily verified to be \(\delta\).    
\end{proof}

\section{Lower Bound for Local Differential Privacy}\label{sect:ldp}

Our lower bounds for local differential privacy follow from
Lemmas~\ref{lm:1-to-d}~and~\ref{lm:puredp-var} and a slight modification of Lemma 3.1
from~\cite{asi2022optimal}, which we state next.
\begin{lemma}\label{lm:canonical}
    For any $\varepsilon$-locally differentially private mechanism
    $\mech=(\mathcal{R}_1, \ldots, \mathcal{R}_n,\mathcal{A})$ that is
    unbiased over \(K\subseteq \R^d\), for any dataset
    $X \in K^n$, there
    exist  $\varepsilon$-differentially private local
    randomizers $\hatcar_i$ for every $i\in [n]$ such that each
    \(\hatcar_i\) is unbiased over \(K\), and for all $\theta\in \R^d$,
    $$\sqrt{\var[\theta^T\mech(X)]}\geq
    \frac{1}{n}\sqrt{\sum_{i=1}^n{\var[\theta^T\hatcar_i(x_i)]}}.$$
\end{lemma}
\begin{proof}
  Define $\hatcar_i$ to be
  $$\hatcar_i(x_i'):=n\E\Braket{\mathcal{A}(\mathcal{R}_1(x_1),\ldots,\mathcal{R}_i(x_i'),\ldots,\mathcal{R}_n(x_n))|\mathcal{R}_i(x'_i)}-\sum_{j\neq
    i}{x_j}.$$
  I.e., to compute \(\hatcar_i(x'_i)\), we run all other randomizers
  with their input from \(X\), and average over all random choices
  except those of \(\mathcal{R}_i\).
  For every $i$, $\hatcar_i$ is $\eps$-differentially private, since
  it is a post-processing of $\mathcal{R}_i$. Also, because \(\mech\)
  is unbiased, we have 
  \(\E[\hatcar_i(x_i')]=x_i',\)
  i.e., \(\hatcar_i\) is also unbiased over \(K\). 
  
  Now we define
    $$\hatcar_{\leq i}:=n\E\left[\mathcal{A}(\mathcal{R}_1(x_1),\ldots,\mathcal{R}_n(x_n))|\mathcal{R}_1(x_1),\ldots,\mathcal{R}_i(x_i)\right]-\sum_{j=1}^n{x_j}.$$
    We have that $\hatcar_{\leq 0}=0$ because \(\mech\) is
    unbiased. Moreover, the random variables \(\hatcar_{\leq i}\) form
    a martingale sequence by the law of total expectation. This means
    that, for any \(\theta \in \R^d\), \(\theta^T \hatcar_{\le 0},
    \ldots, \theta^T\hatcar_{\le n}\) is also a martingale sequence,
    and, therefore
    \begin{align}
    n^2\var[\theta^T\mech(X)] 
      =\E\left[\braket{\theta^T\hatcar_{\leq n}}^2\right]
      =\sum_{i=1}^n{\E\Braket{\braket{\theta^T\hatcar_{\leq i}-\theta^T\hatcar_{\leq i-1}}^2}}.\label{eq:ldp-martingale}
    \end{align}
    Applying the law of total expectation and Jensen's inequality to
    the right hand side, we get
    \begin{align}
      \E\Braket{\braket{\theta^T\hatcar_{\leq i}-\theta^T\hatcar_{\leq
      i-1}}^2}
      &= \E\Braket{\E\Braket{\braket{\theta^T\hatcar_{\leq i}-\theta^T\hatcar_{\leq i-1}}^2|\mathcal{R}_{i}(x_i)}}\notag\\
      &\ge \E\Braket{\E\Braket{\braket{\theta^T\hatcar_{\leq
        i}-\theta^T\hatcar_{\leq i-1}}|\mathcal{R}_{i}(x_i)}^2}\notag\\
      &= \E\Braket{\braket{\theta^T\hatcar_{i}(x_i)-\theta^Tx_i}^2}\label{eq:ldp-jensen}\\
      &= {\var[\theta^T\hatcar_i(x_i)]}.\label{eq:ldp-final}
    \end{align}
    The equality \eqref{eq:ldp-jensen} follows because
    \(\E[\hatcar_{\le i}|\mathcal{R}_i(x_i)] = \hatcar_i(x_i) - x_i\)
    by definition, and because     \[\E[\hatcar_{\le
      i-1}|\mathcal{R}_i(x_i)] = \E[\hatcar_{\le i-1}] = 0.\]
    The proof of the lemma is now completed by plugging
    \eqref{eq:ldp-final} into \eqref{eq:ldp-martingale}.
  \end{proof}
  
With the lemma we just proved above, and Lemmas~\ref{lm:1-to-d}~and~\ref{lm:puredp-var},  we have the following theorem.
\begin{theorem}
  For any \(\eps\)-locally differentially private mechanism \(\mech\) that is
  unbiased over \(K\subseteq \R^d\), and any dataset \(X
  \in K^n\), its error is
  bounded below as
  $$\sqrt{\E \left[{\norm{\mech (X)-{\mu(X)}}_p^2}\right]}\gtrsim \frac{\Gamma_p(K)}{\sqrt{n}e^{-0.5\varepsilon}(e^\varepsilon-1)}.$$    
\end{theorem}
\begin{proof}
  By Lemmas~\ref{lm:puredp-var}~and~\ref{lm:canonical}, for any
  \(\theta \in \R^d\),
  \[
    \sqrt{\var[\theta^T \mech(X)}] \gtrsim
    \frac{w_K(\theta)}{\sqrt{n}e^{-0.5\varepsilon}(e^\varepsilon-1)}. 
  \]
  Now the theorem follows from Lemma~\ref{lm:1-to-d}.
\end{proof}

Note that this theorem nearly matches the lower bound in
Theorem~\ref{thm:ub-ldp} for small \(\eps\).

\section{Distributionally Unbiased Mechanisms}
\label{sect:dist}

While results so far were stated for mechanisms that are unbiased with
respect to their input dataset, they can easily be extended to
mechanisms that take i.i.d.~samples from an unknown distribution \(P\)
over \(K\) and are unbiased with respect to the true mean of
\(P\). Let us define such distributionally unbiased mechanisms.

\begin{definition}
  A mechanism \(\mech\) that takes as input datasets of \(n\) points
  in some bounded domain \(K\subseteq \R^d,\) and
  outputs a vector in \(\R^d\) is distributionally unbiased
  over \(K\subseteq \R^d\) if, for every probability distribution \(P\) over \(K\),
  when \(X\in K^n\) is drawn according to the product distribution
  \(P^n\), we have
  \(
  \E[\mech(X)] = \mu(P),
  \)
  where the expectation is both over the randomness of \(\mech\) and
  the randomness of picking \(X\), and \(\mu(P)\) denotes the mean of \(P\).
\end{definition}

We can reduce proving lower bounds against distributionally unbiased
mechanisms to proving lower bounds against (empirically) unbiased
mechanisms. To this end, let us restate Lemma 6.1
from~\cite{karwa2017finite}.

\begin{lemma}\label{lm:karwavadhan}
  Suppose \(P\) and \(Q\) are two probability distribution on a common
  domain \(K\) with total
  variation distance \(\beta\). Let \(\mech\) be an \((\eps,
  \delta)\)-differentially private mechanism taking inputs from \(K\).
  Define \(\mech_P := \mech(X)\)  when \(X\) is drawn from
  \(P^n\), and \(\mech_Q := \mech(X)\) when \(X\) is drawn from \(Q^n\). Then, for every event \(E\)
  in the range of \(\mech\), we have
  \[
    \Pr[\mech_P \in E] \le e^{\eps'} \Pr[\mech_Q \in E] + \delta',
  \]
  where \(\eps':= 6\eps n \beta\), and \(\delta' := 4e^{\eps'} n\delta \beta\).
\end{lemma}

Lemma \ref{lm:karwavadhan} allows us to show the following lower bound.
\begin{theorem}\label{thm:worstcase-dist}
  Let \(c > 0\) be a small enough absolute constant, let \(p \in [2,
  \infty]\), and let $\mech$ be an
  $(\eps, \delta)$-differentially private mechanism that is distributionally unbiased
  over \(K \subseteq \R^d\). Define \(\delta'\) and \(\hat \eps\) as
  $$\delta':= \frac{8e^{6\eps}\delta}{1-e^{-6\eps}}$$
  $${\hat{\eps}}^{\, 2} := e^{-(6\varepsilon-\log{(1-\delta'))}}(e^{6\varepsilon-\log{(1-\delta')}}-1)^2.$$
  If \(\delta' \le \min\left\{\frac{c}{n},\frac{\hat{\eps}^2}{d^2}\right\}\),
  then there exists a probability distribution \(P\) over \(K\) with
  mean \(\mu(P)\) such that, when \(X\) is drawn from \(P^n\), we have
  $$\sqrt{\E \left[{\norm{\mech (X)-{\mu(P)}}_p^2}\right]}\gtrsim
  \frac{\Gamma_p(K)}{n\hat{\eps}},$$
  with the expectation taken over the randomness of \(\mech\) and the
  random choice of \(X\).
\end{theorem}
\begin{proof}
  Consider the mechanism \(\mech'\) that, on input a dataset \(X\) in
  \(K^n\), samples uniformly and independently with replacement \(n\)
  times from the empirical distribution of \(X\) to form a new dataset
  \(\hat{X}\), and returns \(\mech(\hat{X})\). Then, by
  Lemma~\ref{lm:karwavadhan}, \(\mech'\) is \(\left(6\eps,
    4e^{6\eps}\delta\right)\)-differentially private. Moreover,
  \(\mech'\) is unbiased over \(K\) by the assumption that \(\mech\)
  is distributionally unbiased. The lower bound now follows from
  Theorem~\ref{thm:main-ellp} by taking \(X\) to be any fixed dataset in
  \(K^n\) and defining \(P\) to be either its empirical distribution,
  or the empirical distribution of a neighboring dataset \(X'\),
  depending on which case of Theorem~\ref{thm:main-ellp} holds.
\end{proof}

Theorem~\ref{thm:worstcase-dist} is a worst case lower bound over
distributions. We can, instead, prove a lower bound that gives the
kind of instance optimality we have in
Theorems~\ref{thm:main-ell2}~and~\ref{thm:main-ellp}. Rather than a
reduction to the empirically unbiased case, the proof merely redoes
the proofs of these two theorems in the distributional setting. Rather
than repeating the arguments in full, we will sketch the necessary
modifications. 

\begin{theorem}\label{thm:instopt-dist}
  Let \(c, p, \mech, \hat{\eps}, \delta'\) be as in Theorem~\ref{thm:worstcase-dist}.
  If \(\delta' \le \frac{c}{n}\),
  then, for any probability distribution \(P\) over \(K\) with
  mean \(\mu(P)\), there exists a probability distribution \(Q\) at
  total variation distance at most \(\frac1n\) from \(P\) such that
  the following holds. Either 
  $$\sqrt{\E \left[{\norm{\mech (X)-{\mu(P)}}_p^2}\right]}\gtrsim
  \frac{\Gamma_p(K)}{n\hat{\eps}}$$
  when \(X\) is drawn from \(P^n\),
  or
  $$\sqrt{\E \left[{\norm{\mech (X)-{\mu(Q)}}_p^2}\right]}\gtrsim\frac{\Gamma_p(K)}{\sqrt{\delta'}nd}$$
  when \(X\) is drawn from \(Q^n\). The
  expectations above are taken over the randomness of \(\mech\) and the
  random choice of \(X\).
\end{theorem}
\begin{proof}
  Let $\theta\in \R^d$ be arbitrary.  
  Let \(x' \in K\) be the point guaranteed by Lemma~\ref{lm:farpoints}
  used with \(x := \mu(P)\). 
  We choose \(Q = \frac{n-1}{n}P + \frac1n 1_{x'}\) where \(1_{x'}\)
  is a point mass at \(x'\). This gives us \(|\theta^T(\mu(P) -
  \mu(Q))| \ge \frac{w_K(\theta)}{2n}\). We can then prove a variant
  of Lemma~\ref{lm:ADP-1d-bound} for distributionally unbiased
  mechanisms by applying Lemma~\ref{lm:phat} not to the distributions
  of \(\theta^T \mech(X)\) and \(\theta^T \mech(X')\) for neighboring
  datasets \(X\) and \(X'\), but instead to the marginal distribution
  of \(\theta^T \mech(X)\) when \(X\) is sampled from \(P^n\), and the
  marginal distribution of \(\theta^T \mech(X)\) when \(X\) is sampled
  from \(Q^n\). By Lemma~\ref{lm:karwavadhan},
 these distributions are $(6\eps,
 4e^{6\eps}\delta)$-indistinguishable, i.e., satisfy the defining
 inequality of \((6\eps, 4e^{6\eps}\delta)\)-differential privacy, so
 the proofs of Lemmas~\ref{lm:phat}~and~\ref{lm:ADP-1d-bound} go
 through. The rest of the proofs of
 Theorems~\ref{thm:main-ell2}~and~\ref{thm:main-ellp} then go
 through. 
\end{proof}

\section{Lower Bounds on \(\Gamma_p\)}
\label{sect:lbs}

In this section we compute  the \(\Gamma_p(K)\) norm of several
sets \(K\) that appear in applications. This allows us to prove the
lower bounds in Theorems~\ref{thm:lb-tensor}~and~\ref{thm:marginals}.

Our first lemma determines the \(\Gamma_p\) norm of product sets. 

\begin{lemma}\label{lm:prod}
  Let \(K_1\subseteq \R^{d_1}\) and \(K_2 \subseteq \R^{d_2}\) be
  bounded, and define
  \[
    K := K_1\times K_2 := \{(x,y) \in \R^{d_1 + d_2}: x \in K_1, y \in
    K_2\}.
  \]
  Then \(\Gamma_{2}(K) = \Gamma_2(K_1) + \Gamma_2(K_2)\), and, more
  generally,
  \(
  \Gamma_p(K) = \left(\Gamma_p(K_1)^{\frac{2p}{p+2}} + \Gamma_p(K_2)^{\frac{2p}{p+2}}\right)^{\frac{p+2}{2p}}\) for \(p
  \in [2,  \infty)\), and \(\Gamma_\infty(K) = \sqrt{\Gamma_\infty(K_1)^2
    + \Gamma_\infty(K_2)^2}\). 
\end{lemma}
\begin{proof}
  We first prove the lemma for \(p = 2\), and then reduce to this case
  via Lemma~\ref{lm:Gammap-via-Gamma2}.

  Let us first show that \(\Gamma_2(K) \le \Gamma_2(K_1) +
  \Gamma_2(K_2)\). Let \(d := d_1 + d_2\) so that \(K_1 \times K_2 
  \subseteq \R^d\). We can identify \(\R^{d_1}\) with the first \(d_1\)
  coordinates of \(\R^{d}\), and \(\R^{d_2}\) with the last \(d_2\)
  coordinates, so that \(K_1\) and \(K_2\) lie in orthogonal subspaces of
  \(\R^d\). Then \(K_1\times K_2\) can be identified with \(K_1 + K_2\), and
  the inequality follows form Theorem~\ref{thm:triangle}.

  Next we show that \(\Gamma_2(K) \ge \Gamma_2(K_1) + \Gamma_2(K_2)\)
  using Theorem~\ref{thm:duality} . Let \(P_1 \in \Delta(K_1)\) and
  \(P_2 \in \Delta(K_2)\) be such that \(\Gamma_2(K_i) =
  \tr(\cov(P_i)^{1/2})\) for \(i \in \{1,2\}\). Let \(P \in
  \Delta(K)\) be the probability distribution of \(X = (X_1, X_2)\)
  where \(X_1\sim P_1\) and \(X_2 \sim P_2\) are independent. Then,
  \[
    \cov(P) =
    \begin{pmatrix}
      \cov(P_1) & 0 \\
      0 & \cov(P_2)
    \end{pmatrix}
    \implies 
    \cov(P)^{1/2} =
    \begin{pmatrix}
      \cov(P_1)^{1/2} & 0 \\
      0 & \cov(P_2){1/2}
    \end{pmatrix}
  \]
  and, therefore, by Theorem~\ref{thm:duality},
  \[
    \Gamma_2(K) \ge
    \tr(\cov(P)^{1/2})
    =
    \tr(\cov(P_1)^{1/2}) + \tr(\cov(P_2)^{1/2})
    = \Gamma_2(K_1) + \Gamma_2(K_2).
  \]

  This proves the \(p=2\) case.
  For the case \(p > 2\), we use Lemma~\ref{lm:Gammap-via-Gamma2}. For any
  \(\alpha \in [0,1]\), by what we just proved and by
  Lemma~\ref{lm:Gammap-via-Gamma2}, we have
  \begin{align*}
    \Gamma_p(K) &= \max\{\Gamma_2(DK): D \text{ diagonal}, D \succeq 0,
    \trp{q}{(D^2)} = 1\}\\
    &= \max\{\Gamma_2(D_1K_1) + \Gamma_2(D_2K_2)\},
  \end{align*}
  where the second maximum ranges over
  diagonal matrices \(D_1\succeq 0\) and \(D_2\succeq 0\) such that
  \(\trp{q}(D_1^2)^q + \trp{q}(D_2^2)^q = 1\) for  \(q := \frac{p}{p-2}\). The first equality follows by
  Lemma~\ref{lm:Gammap-via-Gamma2}, and the second equality follows
  from the \(p=2\) case, and from observing that \(DK = (D_1K_1)
  \times (D_2K_2)\) for appropriately chosen disjoint principle
  submatriced \(D_1, D_2\) of \(D\). Replacing \(D_i\) with \(\alpha_i D'_i\) where \(\alpha_i :=
  \sqrt{\trp{q}(D_i^2)}\) and \(\trp{q}((D'_i)^2) =1 \), we get the equivalent
  formulation
  \[
    \Gamma_p(K)
    =
    \max\{{\alpha_1}\Gamma_2(D'_1K_1) + {\alpha_2}\Gamma_2(D'_2K_2)\}
  \]
  with the maximum ranging over non-negative reals \(\alpha_1,\alpha_2\)
  such that \(\alpha_1^{2q} + \alpha_2^{2q} = 1\), and over diagonal
  matrices \(D_1'\succeq 0\) and \(D_2'\succeq 0\) such that
  \(\trp{q}((D_i)^2) = 1\) for \(i \in \{1,2\}\). 
  By Lemma~\ref{lm:Gammap-via-Gamma2}, the right hand side is maximized by taking \(D'_i\) be such that \(\Gamma_2(D'_iK_i) =
  \Gamma_p(K_i)\), and, so, we have
  \[
    \Gamma_p(K)
    =
    \max\{{\alpha_1}\Gamma_p(K_1) + {\alpha_2}\Gamma_p(K_2):
    \alpha_1^{2q} + \alpha_2^{2q} = 1\}.
  \]
  The lemma now follows by H\"older's inequality and its equality case.
\end{proof}

Lemma~\ref{lm:prod} generalizes a fact about the  \(\gamma_2\) norm
proved in~\cite{disc-gamma2}: if the rows of a matrix \(W\) is the disjoint union
of the rows of the matrices \(W_1\) and \(W_2\), then
\(\gamma_2(W) \le (\gamma_2(W_1)^2 + \gamma_2(W_2)^2)^{1/2}\). This is
because the columns of \((W, -W)\) are a subset of \(K_1^{\text{sym}} \times
K_2^{\text{sym}}\), where \(K_i^{\text{sym}}\) is the set of columns of \((W_i,-W_i)\) for \(i \in
\{1,2\}\). The inequality then follows from
Theorem~\ref{thm:Gammap-as-factorization} and
Lemma~\ref{lm:prod}. More generally, for any \(p \ge 2\) we get the inequality
\[
  \gamma_{(p)}(W) \le (\gamma_{(p)}(W_1)^{\frac{2p}{p+2}} +
  \gamma_{(p)}(W_2)^{\frac{2p}{p+2}})^{\frac{p+2}{2p}},
\]
which is new. 

Lemma~\ref{lm:prod} implies immediately that, for sequences \(a_1, \ldots, a_d\)
and \(b_1, \ldots, b_d\) such that \(a_i \le b_i\) for all \(i\), we
have
\begin{equation}\label{eq:box}
  \Gamma_p([a_1,b_1]\times \ldots, [a_d,b_d])
  =
  \|b-a\|_{q},
\end{equation}
where \(q := \frac{2p}{p+2}\) for \(p \in [2, \infty)\), and \(q :=2
\) for \(p = \infty\). In particular, for any real numbers \(a \le
b\), we have
\begin{equation}\label{eq:cube}
\Gamma_p([a,b]^d) =  \frac{d^{\frac12 + \frac1p}(b-a)}{2}.
\end{equation}
Via Theorem~\ref{thm:ub}, equation \eqref{eq:box} recovers an upper
bound on the \(\ell_p\) error for mean estimation with Gaussian noise
from~\cite{plan}. Via
Theorems~\ref{thm:main-ell2}~and~\ref{thm:main-ellp}, it also implies
nearly matching lower bounds for all unbiased mechanisms. 

Next, we show that \(\Gamma_p\) is multiplicative with respect to
Kronecker products. This generalizes the multiplicativity of the $\gamma_2$ norm
with respect to Kronecker products, proved in~\cite{LeeSS08}.
To generalize to \(\Gamma_p(K)\) for general sets \(K\), we first introduce
the following notation. 

\begin{definition}
  For two sets \(K \subseteq \R^{d_1}\) and \(L \subseteq
  \R^{d_2}\), we define
  \[
    K \otimes L := \{x\otimes y: x\in K, y \in L\},
  \]
  where \(x\otimes y\) is the Kronecker product of vectors. We use
  \(K^{\otimes \ell}\) to denote \(K\otimes \ldots \otimes K\) with
  \(\ell\) terms in the product
\end{definition}

We can now state the multiplicativity lemma.
\begin{lemma}\label{lm:tensoring}
  For any \(p \in [2,\infty]\), and any two bounded sets \(K \subseteq
  \R^{d_1}\) and \(L \subseteq\R^{d_2}\), both symmetric around
  \(0\), we have \(\Gamma_{p}(K \otimes L) =
  \Gamma_{p}(K)\Gamma_{p}(L).\)   
\end{lemma}
\begin{proof}
  First we prove
  \(\Gamma_{p}(K \otimes L) \le \Gamma_{p}(K)\Gamma_{p}(L).\)
  Let \(M_K\succ 0\) and \(M_L\succ 0\) be such that \(\max_{x \in K}
  x^T M_K^{-1} x \le 1\) and \(\max_{y \in L}
  y^T M_L^{-1} y \le 1\). Letting \(M := M_K \otimes M_L\), and noting
  that \(M^{-1} = M_K^{-1}\otimes M_L^{-1}\), we have that, for any
  \(x \in K\) and \(y \in L\),
  \[
    (x\otimes y)^TM^{-1}(x\otimes y) =
    (x^T M_K^{-1} x)(y^T M_L^{-1} y) \le 1.
  \]
  Minimizing \(\trp{p/2}{(M_K)}\) over choices of \(M_K\), and
  \(\trp{p/2}{(M_L)}\) over choices of \(M_L\), using
  Lemma~\ref{lm:posdef}, and noticing that \(\trp{p/2}(M) =
  \trp{p/2}(M_K)\trp{p/2}(M_L)\) finishes the proof of the inequality. 

  In the other direction, we need to prove
  \(\Gamma_{p}(K \otimes L) \ge \Gamma_{p}(K)\Gamma_{p}(L).\)
  To this end, we use Theorem~\ref{thm:duality}. Let \(P_K\) and
  \(D_K\) achieve the
  maximum in \eqref{eq:gammap-duality-sym} for \(\Gamma_p(K)\), and
  let \(P_L\) and \(D_L\) achieve the maximum for \(\Gamma_p(L)\). Let \(X_K \sim
  P_K\) and \(X_L \sim P_L\) be independent, and define \(X := X_K
  \otimes X_L\). Then,
  \[
    \E[XX^T] =
    \E[(X_K X_K^T) \otimes (Y_K Y_K^T)] = 
    \E[X_K X_K^T] \otimes \E[X_L X_L^T],
  \]
  where the last equality follows by the independence of \(X_K\) and
  \(X_L\). Let us, finally, also define \(D := D_K \otimes D_L\), and
  notice that \(\trp{q}{(D^2)} = \trp{q}{(D_K^2)}\trp{q}{(D_L^2)} = 1\).
  We have, by Theorem~\ref{thm:duality},
  \begin{align*}
    \Gamma_p(K\otimes L)
    &\ge
    \tr((D\E_{X\sim P}[XX^T]D)^{1/2})\\
    &=
    \tr((D(\E[X_K X_K^T] \otimes \E[X_L X_L^T])D)^{1/2})\\
    &=  \tr(((D_K\E[X_K X_K^T]D_L) \otimes (D_L\E[X_L X_L^T]D_L))^{1/2})\\
    &= \tr((D_K\E[X_K X_K^T]D_L)^{1/2} \otimes (D_L\E[X_L X_L^T]D_L)^{1/2})\\
    &= \tr((D_K\E[X_K X_K^T]D_L)^{1/2}) \tr((D_L\E[X_L  X_L^T]D_L)^{1/2})\\
    &= \Gamma_p(K)\Gamma_p(L).
  \end{align*}
  This completes the proof.
\end{proof}

Finally, we show a simple lemma that allows proving a lower bound on
\(\Gamma_p(K)\) by proving a lower bound on \(\Gamma_p(PK)\) for some
orthogonal projection matrix \(P\). Recall that a coordinate
projection \(P\) is a linear transformation on \(\R^d\) of the form
\[
  \forall i \in [d]:
  P(x)_i =
  \begin{cases}
    x_i & i \in S\\
    0 & i \not \in S
  \end{cases},
\]
for some \(S \subseteq [d]\). 

\begin{lemma}\label{lm:proj-Gammap-lb}
  For any bounded set \(K \subseteq \R^d\), and every coordinate
  projection \(P\), and every \(p \in [2,\infty]\), \(\Gamma_p(PK) \le
  \Gamma_p(K)\). Moreover, for every orthogonal projection \(P\),
  \(\Gamma_2(PK) \le \Gamma_2(K)\).
\end{lemma}
\begin{proof}
  For the first claim observe that, for any matrix
  \(A\) such that \(K \shinclu AB_2^d\), we have \(PK \shinclu
  PAB_2^d\). Moreover, if \(P\) is a coordinate projection, then
  \(\trp{p/2}(PAA^TP) \le \trp{p/2}(AA^T)\). The second claim follows
  analogously, using the fact that, for any orthogonal projection
  \(P\), \(\tr(PAA^TP) \le \tr(AA^T)\). This latter
  inequality follows because
  \begin{align*}
    \tr(AA^T) &= \tr(AA^TP) + \tr(AA^T(I-P))\\
    &= \tr(AA^TP^2) + \tr(AA^T(I-P)^2)\\
    &= \tr(PAA^TP) + \tr((I-P)AA^T(I-P)) \ge \tr(PAA^TP).
  \end{align*}
  This completes the proof.
\end{proof}

Next we use Lemmas~\ref{lm:prod}~and~\ref{lm:tensoring} to
determine the \(\Gamma_p\) norm of two sets that naturally appear in
applications. In particular, we consider the domain of \(\ell\)-tensor
products of points in the unit ball \(B_2^d\), which appears in
Theorem~\ref{thm:lb-tensor}. 
\begin{theorem}\label{thm:tensor}
  Let \(K^\ell_{d,\infty} := \conv\{x^{\otimes \ell}: x \in
  \{-1,+1\}^d\)\}, and let \(K^{\ell}_{d,2}:= \conv\{x^{\otimes \ell}:
  x \in B_2^d\}\) for positive integers \(d\) and \(\ell\). Then, for
  any \(p \ge 2\), we have
  \begin{align*}
    \left(\frac{d}{\ell}\right)^{\frac{\ell}{p} + \frac{\ell}{2}} &\le \Gamma_p(K^{\ell}_{d,\infty}) \le
    d^{\frac{\ell}{p} + \frac{\ell}{2}};\\
    \left(\frac{d}{\ell}\right)^{\frac{\ell}{p}} &\le \Gamma_p(K^{\ell}_{d,2}) \le d^{\frac{\ell}{p}}.
  \end{align*}
\end{theorem}
\begin{proof}
  By   Lemma~\ref{lm:tensoring}, and equation \eqref{eq:cube},
  we have
  \[
    \Gamma_p((\{-1,+1\}^d)^{\otimes \ell})
    = \Gamma_p(\{-1,+1\}^{d})^\ell
    = \Gamma_p([-1,+1]^{d})^\ell
    = d^{\frac{\ell}{p} + \frac{\ell}{2}}.
  \]
  Clearly, \(K^\ell_{d,\infty}\subseteq (\{-1,+1\}^d)^{\otimes \ell}\) and
  the upper bound on \(\Gamma_p(K^{\ell}_{d,\infty})\) follows by
  monotonicity (Theorem~\ref{thm:triangle}). For the lower bound, we
  observe that, after identifying \(\R^{(d/\ell)^\ell}\) with a certain
  coordinate subspace of \(\R^{d^\ell}\), \((\{-1,+1\}^{d/\ell})^{\otimes
    \ell}\) is contained in a coordinate projection of \(K^\ell_{d,\infty}\).  Indeed, for any \(x, y \in
  [-1,1]^{d/2}\), we can define \(z = (x,y) \in \{-1,+1\}^d\), and notice
  that
  \[
    z\otimes z =
    (x^{\otimes 2}, x\otimes y, y \otimes x, y^{\otimes 2}).
  \]
  In particular, there is a coordinate projection of \(z\otimes
  z\) that equals \(x\otimes y\). Using this claim inductively, we see
  that for any \(x_1, \ldots, x_\ell \in \{-1,+1\}^{d/\ell}\), \(x_1
  \otimes \ldots \otimes x_\ell\) equals a coordinate projection of
  \(z^{\otimes \ell}\), where \(z:= (x_1, \ldots, x_\ell) \in \{-1,+1\}^d\). This means that \((\{-1,+1\}^{d/\ell})^{\otimes
    \ell}\) is contained in a coordinate projection 
  \(P(K^\ell_{d,\infty})\) of \(K^\ell_{d,\infty}\), as
  claimed. Therefore, by Lemmas~\ref{lm:tensoring}~and~\ref{lm:proj-Gammap-lb}, and equation \eqref{eq:cube},
  \[
    \Gamma_p(K^{\ell}_{d,\infty})
    \ge \Gamma_p((\{-1,+1\}^{d/\ell})^{\otimes \ell})
    = \Gamma_p(\{-1,+1\}^{d/\ell})^{\ell}
    =\Gamma_p([-1,+1]^{d/\ell})^{\ell}
    = \left(\frac{d}{\ell}\right)^{\frac{\ell}{p} + \frac{\ell}{2}}.
  \]

  The inequalities for \(K_{d,2}^\ell\) follow as a corollary. Notice
  that \(\{-1,+1\}^d \subseteq \sqrt{d}B_2^d\), and, therefore,
  \(K_{d,\infty}^\ell \subseteq d^{\frac{\ell}{2}}
  K_{d,2}^\ell\). Then, by Theorem~\ref{thm:triangle},
  \(\Gamma_p(K_{d,2}^\ell) \ge d^{-\frac{\ell}{2}} \Gamma_p(K_{d,\infty}^\ell)\).
  This implies the lower bound. For the upper bound, analogously to \(K_{d,\infty}\)
  we have, using Lemma~\ref{lm:Gammap-ellipsoid},
  \[
    \Gamma_p(K_{d,2}^\ell)
    \le \Gamma_p((B_2^d)^{\otimes \ell})
    = \Gamma_p(B_2^{d})^\ell
    = d^{\frac{\ell}{p}}.
  \]
  This completes the proof.
\end{proof}

Theorem~\ref{thm:lb-tensor} in the Introduction follows from Theorem~\ref{thm:tensor} and
from Theorems~\ref{thm:main-ell2}~and~\ref{thm:main-ellp}.

Finally, we consider the problem of releasing \(\ell\)-way marginal
queries, and prove Theorem~\ref{thm:marginals}. Recall that
\(\queries^{\text{marg}}_{d,\ell}\) is the set of
statistical queries over the universe \(\uni := \{0,1\}^d\) where each
query \(\query_{s, \beta}\) is defined by a sequence of \(\ell\)
indices \(s := (i_1, \ldots, i_\ell) \in [d]^\ell\) and a sequence of
\(\ell\) bits \(\beta := (\beta_1, \ldots, \beta_\ell) \in
\{0,1\}^d\), and has value \(\query_{s,\beta}(x) = \prod_{j = 1}^\ell
|x_{i_j} - \beta_{j}|\) on every \(x \in \uni\). As mentioned in the
introduction, releasing answers to \(\queries^{\text{marg}}_{d,\ell}(X)\)
is equivalent to mean estimation over the set \(K^{\text{marg}}_{d,\ell}
:= \{\queries^{\text{marg}}_{d,\ell}(x): x \in \{0,1\}^d\}\). Let us
give a more explicit geometric description of this set. When
\(\ell=1\), \(K^{\text{marg}}_{d,1} = L\times \ldots \times L :=
L^{\times d}\) where the cross product is taken \(d\) times, and
\[
L := \left\{
  \begin{pmatrix}1\\0\end{pmatrix},
  \begin{pmatrix}0\\1\end{pmatrix}
\right\}.
\]
For larger \(\ell > 1\), we have \(K^{\text{marg}}_{d,\ell} =
\{y^{\otimes \ell}: y \in K^{\text{marg}}_{d,1}\}.\)

\begin{theorem}\label{thm:marginals-K}
  For any positive integer \(d\), we have
  \[
    \Gamma_p(K^{\text{marg}}_{d,1}) = 2^{\frac1p - 1} d^{\frac1p + \frac12}.
  \]
  Moreover, for any integer \(\ell \ge 2\), we have
  \[
  \frac{d^{\frac{\ell}{2} + \frac{\ell}{p}}}{(2\sqrt{2} \ell)^\ell}\le
  \Gamma_p(K^{\text{marg}}_{d,\ell}) \le d^{\frac{\ell}{p}+\frac{\ell}{2}}.
  \]
\end{theorem}
\begin{proof}
  The formula for \(\Gamma_p(K^{\text{marg}}_{d,1})\) follows from
  Lemma~\ref{lm:prod} since \(\Gamma_p(L) = 2^{\frac1p - 1}.\) To
  justify the last equality, note that, for any \(p \ge 2\),
  \(\Gamma_p(L)\) is achieved by the (degenerate) ellipsoid
  \[
    \begin{pmatrix}
      \frac12\\
      \frac12
    \end{pmatrix}
    +
    \begin{pmatrix}
      \frac{1}{2\sqrt{2}} & -\frac{1}{2\sqrt{2}} \\
      -\frac{1}{2\sqrt{2}} & \frac{1}{2\sqrt{2}} \\
    \end{pmatrix}
    B_2^2.
  \]
  This can be verified, either using Lemma~\ref{lm:Gammap-ellipsoid},
  or using Theorem~\ref{thm:duality} with \(P\) the uniform
  distribution on \(L\) and \(D := 2^{\frac1p -\frac12}I\). 

  The upper bound on \(\Gamma_p(K^{\text{marg}}_{d,\ell})\) follows by Lemma~\ref{lm:Gammap-ellipsoid}
  because \(K^{\text{marg}}_{d,\ell} \subseteq
  d^{\frac{\ell}{2}} B_2^{(2d)^\ell}\). One way to see
  this, is to note that \(L \subseteq B_2^2\), so
  \(K^{\text{marg}}_{d,1} \subseteq \sqrt{d}B_2^{2d}\), and,
  therefore, \(\|y^{\otimes \ell}\|_2 = \|y\|_2^\ell \le
  d^{\frac{\ell}{2}}\) for all \(y \in K^{\text{marg}}_{d,1}\).

  To prove the lower bound, we use Lemma~\ref{lm:proj-Gammap-lb} and
  Theorem~\ref{thm:tensor}.   We will exhibit orthogonal projections
  \(P\) and \(Q\) such that \(QPK^{\text{marg}}_{d,\ell}\) equals \(K^\ell_{d,\infty}\) after
  identifying \(\R^{d^\ell}\) with a subspace of \(\R^{(2d)^\ell}\)
  (in fact, a coordinate subspace). Let \(P_0 := I - \frac{1}{2d} 11^T\), where
  \(1\) is the \(2d\)-dimensional all-ones vector. I.e., \(P_0\)
  projects orthogonal to the all-ones vector. Every point in
  \(y \in K^{\text{marg}}_{d,1}\) satisfies \(1^T y = d\), and
  therefore, \(P_0 y = y - \frac12 1\). This implies that
  \(P_0K^{\text{marg}}_{d,1}\) equals the set \((L')^{\times d}\) where
  \[
    L' := \left\{
      \begin{pmatrix}\frac{1}{2}\\-\frac{1}{2}\end{pmatrix},
      \begin{pmatrix}-\frac{1}{2}\\\frac{1}{2}\end{pmatrix}
    \right\}.
  \]
  Defining \(P := P_0^{\otimes d}\), we have
  \begin{align*}
    PK^{\text{marg}}_{d,\ell}
    &= \{Py^{\otimes d}: y \in K^{\text{marg}}_{d,1}\}\\
    &= \{P_0^{\otimes d} y^{\otimes d}: y \in K^{\text{marg}}_{d,1}\}\\
    &= \{(P_0 y)^{\otimes d}: y \in K^{\text{marg}}_{d,1}\}\\
    &= \{y^{\otimes d}: y \in (L')^{\times d}\}.
  \end{align*}
  Notice that when projected to the first coordinate, \(L'\) is just
  \(\{-\frac12, \frac12\}\). Therefore, there is a coordinate
  projection of \((L')^{\times d}\) (taking the first coordinate from
  every copy of \(L'\)) that equals \(\frac12 \{-1,+1\}^d\). As a
  consequence, there is also a coordinate projection \(Q\) of
  \(PK^{\text{marg}}_{d,\ell}\) that equals \(\frac{1}{2^\ell} K^\ell_{d,\infty}\).  By
  applying Lemma~\ref{lm:proj-Gammap-lb} twice, we have
  \[
    \Gamma_2(K^{\text{marg}}_{d,\ell}) \ge \frac{1}{2^\ell} \Gamma_2(K^\ell_{d,\infty}).
  \]
  The lower bound for \(\Gamma_2(K^{\text{marg}}_{d,\ell})\) then
  follows from Theorem~\ref{thm:tensor}. 
  To prove the lower bound on \(\Gamma_p(K^{\text{marg}}_{d,\ell})\)
  for \(p > 2 \), we make the observation that, for any \(p \in
  [2,\infty]\), and any bounded set \(K\subseteq \R^d\), by Lemma~\ref{lm:trp-norm},
  \(
  \Gamma_2(K) \le d^{\frac12 - \frac1p}\Gamma_p(K).
  \)
  Applying this inequality to \(K^{\text{marg}}_{d,\ell} \subseteq
  \R^{(2d)^\ell}\) finishes the proof. 
\end{proof}

Theorem~\ref{thm:marginals} in the Introduction follows from Theorem~\ref{thm:marginals-K} and
from Theorems~\ref{thm:main-ell2}~and~\ref{thm:main-ellp}. 

\bibliographystyle{alpha}
\bibliography{Privacy}

\newcommand{\etalchar}[1]{$^{#1}$}
\begin{thebibliography}{MMHM18}

\bibitem[AD20a]{AsiD20}
Hilal Asi and John~C. Duchi.
\newblock Instance-optimality in differential privacy via approximate inverse
  sensitivity mechanisms.
\newblock In {\em Advances in Neural Information Processing Systems 33: Annual
  Conference on Neural Information Processing Systems 2020, NeurIPS 2020,
  December 6-12, 2020, virtual}, 2020.

\bibitem[AD20b]{asi20optimal}
Hilal Asi and John~C. Duchi.
\newblock Near instance-optimality in differential privacy.
\newblock {\em CoRR}, abs/2005.10630, 2020.

\bibitem[AFT22]{asi2022optimal}
Hilal Asi, Vitaly Feldman, and Kunal Talwar.
\newblock Optimal algorithms for mean estimation under local differential
  privacy.
\newblock In {\em International Conference on Machine Learning, {ICML} 2022,
  17-23 July 2022, Baltimore, Maryland, {USA}}, volume 162 of {\em Proceedings
  of Machine Learning Research}, pages 1046--1056. {PMLR}, 2022.

\bibitem[ALNP23]{plan}
Martin Aum{\"{u}}ller, Christian~Janos Lebeda, Boel Nelson, and Rasmus Pagh.
\newblock {PLAN:} variance-aware private mean estimation.
\newblock {\em CoRR}, abs/2306.08745, 2023.

\bibitem[AN04]{AlonN04}
Noga Alon and Assaf Naor.
\newblock {Approximating the cut-norm via Grothendieck's inequality}.
\newblock In {\em ACM Symposium on Theory of Computing}, pages 72--80, 2004.

\bibitem[And79]{Ando79}
T.~Ando.
\newblock Concavity of certain maps on positive definite matrices and
  applications to {H}adamard products.
\newblock {\em Linear Algebra Appl.}, 26:203--241, 1979.

\bibitem[BBNS19]{cdp}
Jaros{\l}aw B{\l}asiok, Mark Bun, Aleksandar Nikolov, and Thomas Steinke.
\newblock Towards instance-optimal private query release.
\newblock In {\em Proceedings of the {T}hirtieth {A}nnual {ACM}-{SIAM}
  {S}ymposium on {D}iscrete {A}lgorithms. {SODA} 2019}, pages 2480--2497. SIAM,
  Philadelphia, PA, 2019.

\bibitem[BDKT12]{BhaskaraDKT12}
Aditya Bhaskara, Daniel Dadush, Ravishankar Krishnaswamy, and Kunal Talwar.
\newblock Unconditional differentially private mechanisms for linear queries.
\newblock In Howard~J. Karloff and Toniann Pitassi, editors, {\em Proceedings
  of the 44th Symposium on Theory of Computing Conference, {STOC} 2012, New
  York, NY, USA, May 19 - 22, 2012}, pages 1269--1284. {ACM}, 2012.

\bibitem[BLN21]{BLN21}
Vijay Bhattiprolu, Euiwoong Lee, and Assaf Naor.
\newblock A framework for quadratic form maximization over convex sets through
  nonconvex relaxations.
\newblock In Samir Khuller and Virginia~Vassilevska Williams, editors, {\em
  {STOC} '21: 53rd Annual {ACM} {SIGACT} Symposium on Theory of Computing,
  Virtual Event, Italy, June 21-25, 2021}, pages 870--881. {ACM}, 2021.

\bibitem[BS16]{BunS16}
Mark Bun and Thomas Steinke.
\newblock Concentrated differential privacy: Simplifications, extensions, and
  lower bounds.
\newblock In {\em Theory of Cryptography - 14th International Conference, {TCC}
  2016-B, Beijing, China, October 31 - November 3, 2016, Proceedings, Part
  {I}}, volume 9985 of {\em Lecture Notes in Computer Science}, pages 635--658,
  2016.

\bibitem[BST14]{BassilyST14}
Raef Bassily, Adam~D. Smith, and Abhradeep Thakurta.
\newblock Private empirical risk minimization: Efficient algorithms and tight
  error bounds.
\newblock In {\em 55th {IEEE} Annual Symposium on Foundations of Computer
  Science, {FOCS} 2014, Philadelphia, PA, USA, October 18-21, 2014}, pages
  464--473. {IEEE} Computer Society, 2014.

\bibitem[BUV14]{BunUV14}
Mark Bun, Jonathan Ullman, and Salil~P. Vadhan.
\newblock Fingerprinting codes and the price of approximate differential
  privacy.
\newblock In David~B. Shmoys, editor, {\em Symposium on Theory of Computing,
  {STOC} 2014, New York, NY, USA, May 31 - June 03, 2014}, pages 1--10. {ACM},
  2014.

\bibitem[Car10]{Carlen10}
Eric Carlen.
\newblock Trace inequalities and quantum entropy: an introductory course.
\newblock In {\em Entropy and the quantum}, volume 529 of {\em Contemp. Math.},
  pages 73--140. Amer. Math. Soc., Providence, RI, 2010.

\bibitem[CMRT22]{CMRT22}
Christopher~A. Choquette{-}Choo, H.~Brendan McMahan, Keith Rush, and Abhradeep
  Thakurta.
\newblock Multi-epoch matrix factorization mechanisms for private machine
  learning.
\newblock {\em CoRR}, abs/2211.06530, 2022.

\bibitem[CR51]{CR51}
Douglas~G. Chapman and Herbert Robbins.
\newblock Minimum variance estimation without regularity assumptions.
\newblock {\em Ann. Math. Statistics}, 22:581--586, 1951.

\bibitem[DJW18]{DJW}
John~C. Duchi, Michael~I. Jordan, and Martin~J. Wainwright.
\newblock Minimax optimal procedures for locally private estimation.
\newblock {\em J. Amer. Statist. Assoc.}, 113(521):182--201, 2018.

\bibitem[DKM{\etalchar{+}}06]{DworkKMMN06}
Cynthia Dwork, Krishnaram Kenthapadi, Frank McSherry, Ilya Mironov, and Moni
  Naor.
\newblock Our data, ourselves: Privacy via distributed noise generation.
\newblock In {\em Advances in Cryptology - {EUROCRYPT} 2006, 25th Annual
  International Conference on the Theory and Applications of Cryptographic
  Techniques, St. Petersburg, Russia, May 28 - June 1, 2006, Proceedings},
  pages 486--503, 2006.

\bibitem[DKSS23]{DKSS23}
Travis Dick, Alex Kulesza, Ziteng Sun, and Ananda~Theertha Suresh.
\newblock Subset-based instance optimality in private estimation.
\newblock In Andreas Krause, Emma Brunskill, Kyunghyun Cho, Barbara Engelhardt,
  Sivan Sabato, and Jonathan Scarlett, editors, {\em International Conference
  on Machine Learning, {ICML} 2023, 23-29 July 2023, Honolulu, Hawaii, {USA}},
  volume 202 of {\em Proceedings of Machine Learning Research}, pages
  7992--8014. {PMLR}, 2023.

\bibitem[DLY22]{DLY22}
Wei Dong, Yuting Liang, and Ke~Yi.
\newblock Differentially private covariance revisited.
\newblock {\em CoRR}, abs/2205.14324, 2022.

\bibitem[DMNS06]{DworkMNS06}
Cynthia Dwork, Frank McSherry, Kobbi Nissim, and Adam Smith.
\newblock Calibrating noise to sensitivity in private data analysis.
\newblock In {\em Proceedings of the Third Conference on Theory of
  Cryptography}, TCC'06, pages 265--284, Berlin, Heidelberg, 2006.
  Springer-Verlag.

\bibitem[DNT15]{conjunctions}
Cynthia Dwork, Aleksandar Nikolov, and Kunal Talwar.
\newblock Efficient algorithms for privately releasing marginals via convex
  relaxations.
\newblock {\em Discrete Comput. Geom.}, 53(3):650--673, 2015.

\bibitem[DR16]{DworkR16}
Cynthia Dwork and Guy~N. Rothblum.
\newblock Concentrated differential privacy.
\newblock {\em CoRR}, abs/1603.01887, 2016.

\bibitem[DR18]{DuchiRuan18}
John~C. Duchi and Feng Ruan.
\newblock The right complexity measure in locally private estimation: It is not
  the fisher information.
\newblock {\em CoRR}, abs/1806.05756, 2018.

\bibitem[DY22]{DongYi22}
Wei Dong and Ke~Yi.
\newblock A nearly instance-optimal differentially private mechanism for
  conjunctive queries.
\newblock In {\em {PODS} '22: International Conference on Management of Data,
  Philadelphia, PA, USA, June 12 - 17, 2022}, pages 213--225. {ACM}, 2022.

\bibitem[EG92]{EvansG92}
Lawrence~C. Evans and Ronald~F. Gariepy.
\newblock {\em Measure theory and fine properties of functions}.
\newblock Studies in Advanced Mathematics. CRC Press, Boca Raton, FL, 1992.

\bibitem[EGS03]{EvfimievskiGS03}
Alexandre~V. Evfimievski, Johannes Gehrke, and Ramakrishnan Srikant.
\newblock Limiting privacy breaches in privacy preserving data mining.
\newblock In {\em {PODS}}, pages 211--222. {ACM}, 2003.

\bibitem[ENU20]{factormech}
Alexander Edmonds, Aleksandar Nikolov, and Jonathan Ullman.
\newblock The power of factorization mechanisms in local and central
  differential privacy.
\newblock In {\em S{TOC}'20---{P}roceedings of the 52n {A}nnual {ACM} {SIGACT}
  {S}ymposium on {T}heory of {C}omputing}, pages 425--438. {ACM}, 2020.

\bibitem[Gro53]{grothendieck}
Alexandre Grothendieck.
\newblock R{\'e}sum{\'e} de la th{\'e}orie m{\'e}trique des produits tensoriels
  topologiques.
\newblock {\em Bol. Soc. Mat. Sao Paulo}, 8(1-79):88, 1953.

\bibitem[Ham50]{H50}
J.~M. Hammersley.
\newblock On estimating restricted parameters.
\newblock {\em J. Roy. Statist. Soc. Ser. B}, 12:192--229; discussion,
  230--240, 1950.

\bibitem[HLY21]{HuangLY21}
Ziyue Huang, Yuting Liang, and Ke~Yi.
\newblock Instance-optimal mean estimation under differential privacy.
\newblock In Marc'Aurelio Ranzato, Alina Beygelzimer, Yann~N. Dauphin, Percy
  Liang, and Jennifer~Wortman Vaughan, editors, {\em Advances in Neural
  Information Processing Systems 34: Annual Conference on Neural Information
  Processing Systems 2021, NeurIPS 2021, December 6-14, 2021, virtual}, pages
  25993--26004, 2021.

\bibitem[HU22]{HU22}
Monika Henzinger and Jalaj Upadhyay.
\newblock Constant matters: Fine-grained complexity of differentially private
  continual observation using completely bounded norms.
\newblock {\em CoRR}, abs/2202.11205, 2022.

\bibitem[HUU22]{HUU22}
Monika Henzinger, Jalaj Upadhyay, and Sarvagya Upadhyay.
\newblock Almost tight error bounds on differentially private continual
  counting.
\newblock {\em CoRR}, abs/2211.05006, 2022.

\bibitem[KLN{\etalchar{+}}08]{KLNRS}
Shiva~Prasad Kasiviswanathan, Homin~K. Lee, Kobbi Nissim, Sofya Raskhodnikova,
  and Adam Smith.
\newblock What can we learn privately?
\newblock In {\em FOCS}, pages 531--540. {IEEE}, Oct 25--28 2008.

\bibitem[KMR{\etalchar{+}}23]{trilemma23}
Gautam Kamath, Argyris Mouzakis, Matthew Regehr, Vikrant Singhal, Thomas
  Steinke, and Jonathan~R. Ullman.
\newblock A bias-variance-privacy trilemma for statistical estimation.
\newblock {\em CoRR}, abs/2301.13334, 2023.

\bibitem[KN12]{KhotNaor12}
Subhash Khot and Assaf Naor.
\newblock Grothendieck-type inequalities in combinatorial optimization.
\newblock {\em Comm. Pure Appl. Math.}, 65(7):992--1035, 2012.

\bibitem[Kom88]{Komiya88}
Hidetoshi Komiya.
\newblock Elementary proof for {S}ion's minimax theorem.
\newblock {\em Kodai Math. J.}, 11(1):5--7, 1988.

\bibitem[KOV17]{KOV17}
Peter Kairouz, Sewoong Oh, and Pramod Viswanath.
\newblock The composition theorem for differential privacy.
\newblock {\em {IEEE} Trans. Inf. Theory}, 63(6):4037--4049, 2017.

\bibitem[Kri74]{Krivine74}
J.~L. Krivine.
\newblock Th\'{e}or\`emes de factorisation dans les espaces r\'{e}ticul\'{e}s.
\newblock In {\em S\'{e}minaire {M}aurey-{S}chwartz 1973--1974: {E}spaces
  {$L^p$}, applications radonifiantes et g\'{e}om\'{e}trie des espaces de
  {B}anach}, pages Exp. Nos. 22 et 23, 22. \'{E}cole Polytech., Paris, 1974.

\bibitem[KS14]{kasiviswanathan2008thesemantics}
Shiva~Prasad Kasiviswanathan and Adam~D. Smith.
\newblock On the 'semantics' of differential privacy: {A} bayesian formulation.
\newblock {\em J. Priv. Confidentiality}, 6(1), 2014.

\bibitem[KV18]{karwa2017finite}
Vishesh Karwa and Salil~P. Vadhan.
\newblock Finite sample differentially private confidence intervals.
\newblock In Anna~R. Karlin, editor, {\em 9th Innovations in Theoretical
  Computer Science Conference, {ITCS} 2018, January 11-14, 2018, Cambridge, MA,
  {USA}}, volume~94 of {\em LIPIcs}, pages 44:1--44:9. Schloss Dagstuhl -
  Leibniz-Zentrum f{\"{u}}r Informatik, 2018.

\bibitem[LHR{\etalchar{+}}10]{LiHRMM10}
Chao Li, Michael Hay, Vibhor Rastogi, Gerome Miklau, and Andrew McGregor.
\newblock Optimizing linear counting queries under differential privacy.
\newblock In {\em Proceedings of the 29th ACM Symposium on Principles of
  Database Systems}, PODS'10, pages 123--134. ACM, 2010.

\bibitem[LMH{\etalchar{+}}15]{MM}
Chao Li, Gerome Miklau, Michael Hay, Andrew McGregor, and Vibhor Rastogi.
\newblock The matrix mechanism: optimizing linear counting queries under
  differential privacy.
\newblock {\em {VLDB} J.}, 24(6):757--781, 2015.

\bibitem[LSS08]{LeeSS08}
Troy Lee, Adi Shraibman, and Robert Spalek.
\newblock A direct product theorem for discrepancy.
\newblock In {\em Proceedings of the 23rd Annual {IEEE} Conference on
  Computational Complexity, {CCC} 2008, 23-26 June 2008, College Park,
  Maryland, {USA}}, pages 71--80. {IEEE} Computer Society, 2008.

\bibitem[MMHM18]{HDMM}
Ryan McKenna, Gerome Miklau, Michael Hay, and Ashwin Machanavajjhala.
\newblock Optimizing error of high-dimensional statistical queries under
  differential privacy.
\newblock {\em Proc. {VLDB} Endow.}, 11(10):1206--1219, 2018.

\bibitem[MNT20]{disc-gamma2}
Ji\v{r}{\'i} Matou\v{s}ek, Aleksandar Nikolov, and Kunal Talwar.
\newblock Factorization norms and hereditary discrepancy.
\newblock {\em Int. Math. Res. Not. IMRN}, 2020(3):751--780, 2020.

\bibitem[MRT22]{MRT22}
Brendan McMahan, Keith Rush, and Abhradeep~Guha Thakurta.
\newblock Private online prefix sums via optimal matrix factorizations.
\newblock {\em CoRR}, abs/2202.08312, 2022.

\bibitem[MSU22]{MSU22}
Audra McMillan, Adam~D. Smith, and Jonathan~R. Ullman.
\newblock Instance-optimal differentially private estimation.
\newblock {\em CoRR}, abs/2210.15819, 2022.

\bibitem[Nes98]{Nesterov98}
Yu. Nesterov.
\newblock Semidefinite relaxation and nonconvex quadratic optimization.
\newblock {\em Optim. Methods Softw.}, 9(1-3):141--160, 1998.

\bibitem[Nik14]{nikolov-thesis}
Aleksandar Nikolov.
\newblock {\em New Computational Aspects of Discrepancy Theory}.
\newblock PhD thesis, Rutgers, The State University of New Jersey, 2014.

\bibitem[Nik23]{jlquery}
Aleksandar Nikolov.
\newblock Private query release via the johnson-lindenstrauss transform.
\newblock In {\em Proceedings of the 2023 {ACM-SIAM} Symposium on Discrete
  Algorithms, {SODA} 2023, Florence, Italy, January 22-25, 2023}, pages
  4982--5002. {SIAM}, 2023.

\bibitem[NRV14]{NRV-noncomm}
Assaf Naor, Oded Regev, and Thomas Vidick.
\newblock Efficient rounding for the noncommutative {G}rothendieck inequality.
\newblock {\em Theory Comput.}, 10:257--295, 2014.

\bibitem[NTZ13]{NTZ-stoc}
Aleksandar Nikolov, Kunal Talwar, and Li~Zhang.
\newblock The geometry of differential privacy: the sparse and approximate
  cases.
\newblock In {\em S{TOC}'13---{P}roceedings of the 2013 {ACM} {S}ymposium on
  {T}heory of {C}omputing}, pages 351--360. ACM, New York, 2013.

\bibitem[NTZ16]{NTZ}
Aleksandar Nikolov, Kunal Talwar, and Li~Zhang.
\newblock The geometry of differential privacy: The small database and
  approximate cases.
\newblock {\em {SIAM} J. Comput.}, 45(2):575--616, 2016.

\bibitem[Pis78]{Pisier-noncomm}
Gilles Pisier.
\newblock Grothendieck's theorem for noncommutative {$C\sp{\ast} $}-algebras,
  with an appendix on {G}rothendieck's constants.
\newblock {\em J. Functional Analysis}, 29(3):397--415, 1978.

\bibitem[Roc70]{Rockafellar}
R.~Tyrrell Rockafellar.
\newblock {\em Convex analysis}.
\newblock Princeton Mathematical Series, No. 28. Princeton University Press,
  Princeton, N.J., 1970.

\bibitem[RT09]{MathStats09}
Kandethody~M. Ramachandran and Chris~P. Tsokos.
\newblock {\em Mathematical statistics with applications}.
\newblock Elsevier/Academic Press, Amsterdam, 2009.

\bibitem[Sio58]{Sion58}
Maurice Sion.
\newblock On general minimax theorems.
\newblock {\em Pacific J. Math.}, 8:171--176, 1958.

\bibitem[TJ89]{TJ-book}
N.~Tomczak-Jaegermann.
\newblock {\em Banach-Mazur Distances and Finite-Dimensional Operator Ideals}.
\newblock Pitman Monographs and Surveys in Pure and Applied Mathematics 38. J.
  Wiley, New York, 1989.

\bibitem[Ver18]{Vershynin-HDP}
Roman Vershynin.
\newblock {\em High-dimensional probability}, volume~47 of {\em Cambridge
  Series in Statistical and Probabilistic Mathematics}.
\newblock Cambridge University Press, Cambridge, 2018.
\newblock An introduction with applications in data science, With a foreword by
  Sara van de Geer.

\bibitem[VN93]{VoinovNikulin-vol1}
V.~G. Voinov and M.~S. Nikulin.
\newblock {\em Unbiased estimators and their applications. {V}ol. 1}, volume
  263 of {\em Mathematics and its Applications}.
\newblock Kluwer Academic Publishers, Dordrecht, 1993.
\newblock Univariate case, Translated from the 1989 Russian original by L. E.
  Strautman and revised by the authors.

\bibitem[VN96]{VoinovNikulin-vol2}
V.~G. Voinov and M.~S. Nikulin.
\newblock {\em Unbiased estimators and their applications. {V}ol. 2}, volume
  362 of {\em Mathematics and its Applications}.
\newblock Kluwer Academic Publishers Group, Dordrecht, 1996.
\newblock Multivariate case.

\bibitem[War65]{Warner65}
Stanley~L. Warner.
\newblock Randomized response: A survey technique for eliminating evasive
  answer bias.
\newblock {\em Journal of the American Statistical Association},
  60(309):63--69, 1965.

\bibitem[XHZK23]{XHZK23}
Yingtai Xiao, Guanlin He, Danfeng Zhang, and Daniel Kifer.
\newblock An optimal and scalable matrix mechanism for noisy marginals under
  convex loss functions.
\newblock {\em CoRR}, abs/2305.08175, 2023.

\bibitem[ZFH{\etalchar{+}}23]{ZFHDT23}
Keyu Zhu, Ferdinando Fioretto, Pascal~Van Hentenryck, Saswat Das, and Christine
  Task.
\newblock Privacy and bias analysis of disclosure avoidance systems.
\newblock {\em CoRR}, abs/2301.12204, 2023.

\bibitem[ZHF21]{ZhuHF21}
Keyu Zhu, Pascal~Van Hentenryck, and Ferdinando Fioretto.
\newblock Bias and variance of post-processing in differential privacy.
\newblock In {\em Thirty-Fifth {AAAI} Conference on Artificial Intelligence,
  {AAAI} 2021, Virtual Event, February 2-9, 2021}, pages 11177--11184. {AAAI}
  Press, 2021.

\end{thebibliography}

\appendix

\section{Proof of the Dual Characterization}
\label{app:duality}

Before we proceed to the proofs of the main duality results, we recall some
important preliminaries, and introduce useful notation. First we recall a special case of Sion's minimax
theorem~\cite{Sion58} (see also a simple proof by Komiya~\cite{Komiya88}).
\begin{theorem}\label{thm:sion}
  Let \(A\) and \(B\) be convex subsets of two topological vector
  spaces. Let \(f:A \times B\to \R\) be such that
  \begin{itemize}
  \item \(f(a,b)\) is continuous and convex in \(b\) for each \(a \in
    A\);
  \item \(f(a,b)\) is continuous and concave in \(a\) for each \(b \in
    B\).
  \end{itemize}
  If at least one of \(A\) or \(B\) is compact, then
  \[
    \sup_{a \in A}\inf_{b \in B} f(a,b) =
    \inf_{b \in B} \sup_{a \in A} f(a,b).
  \]
\end{theorem}
Sion's theorem in fact allows weaker conditions on the continuity and
convexity-concavity properties of \(f\), but we will not need this
extra power. 

We also recall the notion of weak* convergence: a sequence of measures
\(P_1, P_2, \ldots \in \Delta(K)\) weak* converges to \(P\) if for all
continuous functions \(f:K \to \R\) the sequence
\(\E_{X \sim P_1}[f(X)], \E_{X \sim P_2}[f(X)], \ldots\) converges to
\(E_{X \sim P}[f(X)]\). A standard application of the Riesz
representation theorem shows that, for a compact set \(K\),
\(\Delta(K)\) is sequentially compact in the weak* topology, i.e., any
sequence of measures in \(\Delta(K)\) has a subsequence that weak*
converges. For a proof (of a more general result) see, e.g.~Section
1.9 in~\cite{EvansG92}.

With these preliminaries out of the way, we can now prove our duality
results. 
Our first lemma gives a dual characterization of \(\Gamma_2(K)\), which
generalizes a known
duality result for the \(\gamma_F\) factorization
norm~\cite{LeeSS08,HUU22}. Later in this subsection we also derive a
characterization of \(\Gamma_p(K)\) as a consequence of the one for \(\Gamma_2(K)\).

\begin{lemma}\label{lm:gamma2-duality}
  Let \(K \subseteq \R^d\) be a compact set. We have the identity
  \begin{equation}\label{eq:gamma2-duality}
    \Gamma_2(K) = \sup\{\tr(\cov(P)^{1/2}): P \in \Delta(K)\}.
  \end{equation}
\end{lemma}
\begin{proof}
  By Lemma~\ref{lm:posdef} (in the case \(p = 2\)),
  \begin{equation}\label{eq:gamma2-sqrt}
    \Gamma_2(K) =
    \inf\left\{\sqrt{\tr(M)}: v \in \R^d, M\succ 0, (x+v)^T M^{-1}(x+v)
      \le 1 \ \forall x \in K\right\}.
  \end{equation}
  Letting \(t = \tr(M)\), and dividing \(M\) by \(t\), we can rewrite
  \eqref{eq:gamma2-sqrt} as 
  \begin{align*}
    \Gamma_2(K) &= \inf\left\{t: v \in \R^d, M\succ 0, \tr(M) \le 1, (x+v)^T M^{-1}(x+v)
      \le t\right\}\\
    &= \inf\left\{\max_{x \in K} \sqrt{(x+v)^T
                  M^{-1}(x+v)}: v\in \R^d,  M\succ 0, \tr(M) \le 1\right\}\\
                &= \left(\inf_{M \in S, v \in \R^d}\max_{P \in \Delta(K)} \E_{X \sim P}[(X+v)^T M^{-1}(X+v)]\right)^{1/2},
  \end{align*}
  where \(S := \{M\succ 0, \tr(M) \le 1\}\).
  We are now ready to apply Theorem~\ref{thm:sion} to the expression inside the square root on
  the right  hand side. The set \(\Delta(K)\) is clearly convex, and,
  as noted above, compact in the weak* topology. The set \(S\times \R^d\) is
  also convex. Moreover, the function \(f(M,v,P):= \E_{X \sim
    P}[(X+v)^T M^{-1}(X+v)]\) is continuous in
  \((M,v)\) for each \(P\) since \(K\) is bounded and \(M^{-1}\) is
  continuous on \(S\). This function is also jointly convex in \(M\)
  and \(v\) for each
  \(P\), since the function \((x+v)^T M^{-1} (x+v)\) is jointly convex
  in \(M\) and \(v\) for each \(x \in \R^d\): see Theorem~1 of
  \cite{Ando79} or Section 3.1 of \cite{Carlen10}. Finally, for each
  \(M \in S\) and \(v \in \R^d\), \(f(M,v,P)\) is linear
  in \(P\), and continuous in the weak* topology. Thus all
  assumptions of Theorem~\ref{thm:sion} are satisfied, and we can
  switch the \(\inf\) and the \(\max\), getting
  \begin{equation}\label{eq:gamma2-dual-saddle}
    \Gamma_2(K) =
    \left(\max_{P \in \Delta(K)} \inf_{M \in S, v \in \R^d}\E_{X \sim P}[(X+v)^T M^{-1}(X+v)]\right)^{1/2}.
  \end{equation}
  
  To complete the proof of \eqref{eq:gamma2-duality}, we need to show that
  \begin{multline}\label{eq:gamma2-dual-inner}
    \inf_{M \in S, v \in \R^d}\E_{X \sim P}[(X+v)^T M^{-1}(X+v)]\\
    =
    \inf_{M \in S}\tr(\E_{X \sim P}[(X+v)(X+v)^T] M^{-1})
    =
    \tr(\cov(P)^{1/2})^2.
  \end{multline}
  We first perform the minimization over \(v \in \R^d\). We claim
  that, for a fixed \(M \succ 0\), the infimum over \(v \in \R^d\) is achieved by \(v = -\E_{X'\sim P}[X']\). It is enough to show that, for any \(v \in
  \R^d\),
  \begin{equation}\label{eq:matrix-var}
  \E_{X \sim P}[(X+v)(X +v)^T] \succeq \cov(P).
  \end{equation}
  Let  \(\theta \in \R^d\) be arbitrary. We have
  \begin{align*}
    \theta^T \E_{X \sim P}[(X+v)(X +v)^T] \theta
    &= \E_{X\sim P}[(\ip{X}{\theta} + \ip{v}{\theta})^2]\\
    &\ge \E_{X\sim P}[(\ip{X}{\theta} - \E_{X' \sim P}[\ip{X'}{\theta}])^2]\\
    &=
    \E_{X\sim P}[(\ip{X}{\theta} - \ip{\E_{X' \sim  P}[X']}{\theta}])^2]\\
    &=
     \theta^T \cov(P)\theta,
   \end{align*}
   where the inequality follows because for any real-valued
   square-integrable random variable \(Z\), \(\E[(Z + t)^2]\) is
   minimized at \(t = - \E[Z]\). 
   This proves \eqref{eq:matrix-var} and shows that the left hand side
   of \eqref{eq:gamma2-dual-inner} equals
   \(
   \inf_{M \in S}\tr(\cov(P)M^{-1}).
   \)

   Notice next that, by the matrix Cauchy-Schwarz inequality, for any
  \(P \in \Delta(K)\) and \(M \in S\),
  \begin{align*}
    \tr(\cov(P)^{1/2})^2
    &=  \tr((\cov(P)^{1/2}M^{-1/2})M^{1/2})^2\\
    &\le
    \tr(\cov(P) M^{-1}) \tr(M)\\
    &\le \tr(\cov(P) M^{-1}).
  \end{align*}
  Moreover, if \(\cov(P) \succ 0\), then equality is achieved for
  \(
  M := \frac{1}{\tr(\cov(P)^{1/2})} \cov(P)^{1/2}.
  \)
  In case \(\cov(P)\) is singular, we can define, for \(\beta > 0\),
  \[
    M := \frac{1}{\tr(\cov(P)^{1/2} + \beta I)} (\cov(P)^{1/2} + \beta I).
  \]
  Letting \(\beta\) go to \(0\), we achieve equality again in the
  limit.
  This finishes the proof of \eqref{eq:gamma2-dual-inner}, and,
  together with \eqref{eq:gamma2-dual-saddle}, also the proof of
  \eqref{eq:gamma2-duality}. 
\end{proof}

Next, we prove Lemma~\ref{lm:Gammap-via-Gamma2}, which relates
\(\Gamma_p\) for \(p > 2\) to \(\Gamma_2\). This allows translating
many facts about \(\Gamma_p\) into facts about \(\Gamma_2\) which can
be easier to establish. In particular, it allows us to extend
Lemma~\ref{lm:gamma2-duality} to \(p > 2\).

\begin{proof}[Proof of Lemma~\ref{lm:Gammap-via-Gamma2}]
  Note that it is enough to prove
  \[
    \Gamma_p(K) = \max\{\Gamma_2(DK): D \text{ diagonal}, D \succeq 0,
    \trp{q}{(D^2)} \le 1\},
  \]
  since the right hand side is maximized at some $D$
  satisfying $\trp{q}{D^2} = 1$ by the homogeneity property of
  $\Gamma_2$ (Theorem~\ref{thm:triangle}).
  We first show that, for any non-negative diagonal matrix \(D\) such that
  \(\trp{q}{(D)} \le 1\) we have
  \[
    \Gamma_p(K) \ge \Gamma_2(DK).
  \]
  Indeed, for any \(d \times d\) matrix \(A\) such that \(K \subseteq
  AB_2^d\), we have that \(DK \subseteq DAB_2^d\). Moreover, by
  H\"older's inequality,
  \begin{align*}
    \tr((DA)(DA)^T) = \tr(DAA^TD) = \tr(D^2AA^T)
    &= \sum_{i = 1}^d D_{ii}^2(AA^T)_{ii}\\
    &\le \trp{q}(D^2)\trp{p/2}(AA^T)
      \le \trp{p/2}(AA^T).
  \end{align*}
  Taking an infimum over all choices of \(A\) gives the inequality.

  In the other direction we use Theorem~\ref{thm:sion}. We first note that, for
  any non-negative diagonal matrix \(D\), we can write 
  \begin{equation}\label{eq:gamma2-diag}
    \Gamma_2(DK)^2 = \inf\{\tr(D^2M): M \succeq 0, K \shinclu \sqrt{M} B_2^d\}.
  \end{equation}
  This identity follows trivially from Lemma~\ref{lm:posdef} when \(D\) is invertible, since then \(K
  \shinclu \sqrt{M} B_2^d\) if and only if \(DK \shinclu D\sqrt{M}B_2^d\). The
  general case follows by a continuity argument. In particular,
  consider \(D_\beta := ((1-\beta)D^2 + \beta I)^{1/2}\) for \(\beta \in
  [0,1]\). It follows easily from Theorem~\ref{thm:triangle} that
  \(\Gamma_2(D_\beta K)^2\) is continuous in \(\beta\). Moreover, for
  any positive semidefinite matrix \(M\), \(\tr(D^2M)\) is linear in
  \(D^2\), and, so, it follows from
  Corollary 7.5.1.~in~\cite{Rockafellar} that \(\inf\{\tr(D_\beta^2M): K \shinclu
  AB_2^d\}\)  is continuous from the right as \(\beta \to 0\). Since \(D_\beta\) is
  invertible for any \(\beta>0\), the general case of
  \eqref{eq:gamma2-diag} follows. 

  Using \eqref{eq:gamma2-diag}, we can write the right hand
  side of \eqref{eq:Gammap-via-Gamma2} as
  \[
    \left(\max_{D\in S} \inf_{A \in T} \tr(D^2M)\right)^{1/2},
  \]
  where \(S\) is the set of diagonal matrices \(D\) such that
  \(\trp{q}(D^2) \le 1\), and \(T\) is the set of matrices
  \(M\succeq 0\) such that \(K\shinclu \sqrt{M}B_2^d\). Let us reparameterize by
  taking \(C := D^2\) and rewrite the right hand
  side of \eqref{eq:Gammap-via-Gamma2}  again as
  \begin{equation}\label{eq:saddle}
    \left(\max_{C \in S'} \inf_{M \in T} \tr(CM)\right)^{1/2},
  \end{equation}
  where \(S'\) is the set of non-negative diagonal matrices \(C\)
  such that \(\trp{q}(C) \le 1\), and \(T\) is as above. Note that \(S'\) is convex and compact, since \(\trp{q}(C)\) is just the \(\ell_q\) norm
  of the diagonal of \(C\), and, therefore, is a convex function. 
  The
  following claim establishes that \(T\) is also convex.
  \begin{claim}
    The set \(T\) of positive semidefinite matrices \(M\) for which \(K \shinclu \sqrt{M}
    B_2^d\) is convex.
  \end{claim}
  \begin{proof}
    A standard argument, that we recalled in the proof of Lemma~\ref{lm:posdef}, shows that if \(K
    \shinclu \sqrt{M} B_2^d\), then there is some \(v \in -\clconv K\) such that
    \(K + v \subseteq \sqrt{M}B_2^d\). This implies that $0 \in \clconv(K+v)$,
    and, therefore
    \[
      h_{K+v}(\theta)  = h_{\clconv(K+v)}\ge 0.
    \]
    Then, \(K \shinclu \sqrt{M}
    B_2^d\) implies that there exists a $v \in \R^d$ such that for every \(\theta \in \R^d\),
    \[
      \theta^T M \theta = h_{\sqrt{M}B_2^d}(\theta)^2 \ge
      h_{K + v}(\theta)^2.
    \]
    Let us take two positive semidefinite \(M_1, M_2\) and vectors
    \(v_1, v_2\in \R^d\) satisfying these conditions, and some \(\lambda
    \in (0,1)\), and define \(M := \lambda M_1 + (1-\lambda)M_2)\) and
    \(v := \lambda v_1 + (1-\lambda)v_2\). By Jensen's inequality, we
    have that for any \(\theta \in \R^d\), 
    \begin{align*}
      \theta^T M \theta
      &= \lambda \theta^T M_1 \theta + (1-\lambda) \theta^T M_2 \theta\\
      &\ge \lambda h_{K+v_1}(\theta)^2 + (1-\lambda) h_{K+v_2}(\theta)^2\\
      &\ge (\lambda h_{K+v_1}(\theta) + (1-\lambda) h_{K+v_2}(\theta))^2
      = h_{K + v}(\theta)^2,
    \end{align*}
    which implies \(K \shinclu \sqrt{M} B_2^d\). This completes the
    proof of the claim
  \end{proof}
  Since \(S'\) is convex and compact, \(T\) is convex, and
  \(\tr(CM)\) is linear in \(C\) and in \(M\), by Theorem~\ref{thm:sion} we have
  \[
    \max_{C\in S'} \inf_{M \in T} \tr(CM) = \inf_{M\in T} \max_{C
      \in S'} \tr(CM).
  \]
  But now, by the equality case of H\"older's inequality, we
  have that \(\max_{C \in S'} \tr(CM) = \trp{p/2}(M)\) for any matrix \(M\)
  with non-negative diagonal,
  and, therefore, the right hand side of \eqref{eq:Gammap-via-Gamma2}
  can be written as
  \begin{equation*}
    \inf\left\{\sqrt{\trp{p/2}(M)}: M\succeq 0, K\shinclu
      \sqrt{M}B_2^d\right\}
    =
    \Gamma_p(K),
  \end{equation*}
  where the equality follows from Lemma~\ref{lm:posdef}.
\end{proof}

The proof of our dual characterization in Theorem~\ref{thm:duality}
follows from combining
Lemmas~\ref{lm:gamma2-duality}~and~\ref{lm:Gammap-via-Gamma2}. 

\begin{proof}[Proof of Theorem~\ref{thm:duality}]
  Note that both the left and the right hand side of \eqref{eq:gammap-duality} remains
  unchanged if we replace \(K\) with its closure \(\bar{K}\). For the left hand
  side, this follows from the definition of \(\Gamma_p(K)\), since
  \(AB_2^d\) is closed for any matrix \(A\). For the right hand side,
  this follows because any measure \(P \in \Delta(\bar{K})\) can be written as
  the weak* limit of measures in \(\Delta(K)\). We assume for the rest
  of the proof, therefore, that \(K\) is compact. 
  
  Lemmas~\ref{lm:gamma2-duality}~and~\ref{lm:Gammap-via-Gamma2} then imply
  that
  \[
    \Gamma_p(K) = \sup\{\tr(\cov(P)^{1/2}):D \text{ diagonal}, D \succeq 0, \trp{q}{(D^2)} = 1, P \in \Delta(DK)\}.
  \]
  Note that, for any \(P \in \Delta(DK)\), we can find a \(Q \in
  \Delta(K)\) such that, if \(Y\sim Q\), then \(X := DK\) is
  distributed according to \(P\). Since \(\cov(P) = D\cov(Q)D\), the
  formula above is the equivalent to \eqref{eq:gammap-duality}.

  To prove the claim after ``moreover'', we show that, if \(K\) is
  symmetric around \(0\), then the supremum in
  \eqref{eq:gammap-duality} can be taken over measures that are
  symmetric around \(0\) as well. Since such measures have mean \(0\),
  the claim follows. Let \(P \in \Delta\) be arbitrary,
  and let \(X\) be distributed according
  to \(P\). Set \(Y := sX\), where \(s\) is chosen uniformly at
  random in \(\{-1, +1\}\), and independently of \(X\). Let \(Q\) be
  the probability distribution of \(Y\). Then \(Q \in \Delta(K)\) by
  the symmetry assumption, and, since \(\E[Y] = 0\) and \(YY^T = s^2
  XX^T = XX^T\), 
  \[
    \cov(Q) = \E[YY^T] = \E[XX^T] \succeq \cov(P),
  \]
  where the final inequality holds because
  \(
  \E[XX^T] = \cov(P) + \E[X]\E[X]^T.
  \)
  This means that \(D\cov(Q) D \succeq D\cov(P)D\) for any diagonal matrix
  \(D\), and, therefore,
  \[
    \tr((D\E[YY^T]D)
    = \tr(D\cov(Q)D)
    \ge \tr(D\cov(P)D).
  \]
  This completes the proof of the claim.
\end{proof}

\section{Computing $\Gamma_p(K)$ using the Ellipsoid Method}
\label{app:opt}

Here we supply the missing details of the proof of
Theorem~\ref{thm:opt}. Recall that we need to show that, under the
assumptions of the theorem, \eqref{eq:cp-obj}--\eqref{eq:cp-psd} can
be solved approximately using the ellipsoid method.  First, we argue
that the set
\(\lambda S_{p/2} := \{M\succeq 0: \trp{p/2}(M) \le \lambda\}\) has a
separation oracle for any \(\lambda\). If \(M\) is not positive
semidefinite, then standard techniques allow us to compute some
positive semidefinite matrix \(X\) such that \(\ip{X}{M} < 0\),
whereas \(\ip{X}{M'} \ge 0\) for any \(M' \succ 0\). If
\(M \succeq 0\) but \(\trp{p/2}(M) > \lambda\), then convexity implies
that for any matrix \(M' \succeq 0\)
\begin{equation}\label{eq:sep-p}
  \trp{p/2}(M')^{p/2}
  \ge
  \trp{p/2}(M)^{p/2} + \ip{X}{M'-M},
\end{equation}
where \(X\) is the gradient of \(\trp{p/2}(M)^{p/2}\) with respect
to \(M\), i.e., a diagonal matrix with entries
\(X_{ii} = \frac{p}{2}M_{ii}^{(p-2)/2}\). Then, for
\(t := \ip{X}{M} + \lambda^{p/2} - \trp{p/2}(M)\), we have that
\[
  \ip{X}{M} = t + \trp{p/2}(M) - \lambda^{p/2} > t.
\]
On the other hand, \(\ip{X}{M'} \le t\) for all \(M' \in \lambda S_{p/2}\)
since, otherwise
\[    \trp{p/2}(M)^{p/2} + \ip{X}{M'-M} > \trp{p/2}(M)^{p/2} +t - \ip{X}{M}
  = \lambda^{p/2},
\]
and
\eqref{eq:sep-p} would imply \(\trp{p/2}(M') >  \lambda\) in
contradiction with \(M' \in \lambda S_{p/2}\).

Next we show that the existence of an approximate quadratic oracle
for \(K\) implies an approximate separation oracle for the set
\[
  S_K := \{(M,v): M \succ 0, v \in RB_2^d, (x + v)^T M^{-1} (x+v) \le 1 \ \
  \forall x \in K\}. 
\]
In particular, we show that, if \(K\) has an \(\alpha\)-approximate
quadratic oracle with running time \(T\), then there is an algorithm \(\mathcal{SO}\)
that takes as input a \(d\times d\) matrix \(M\) and \(v \in \R^d\),
runs in time polynomial in \(T,d\) and the bit description of \(M\)
and \(v\), and satisfies the conditions
\begin{itemize}
\item if \(M\) is not positive definite, then \(\mathcal{SO}(M,v)\)
  outputs a nonzero positive 
  semidefinite matrix \(X\) such that \(\ip{X}{M} \le 0\); note that
  \(\ip{X}{M'} > 0\) for all \(M\succ 0\);
\item if \(\|v\|_2 > R\), then \(\mathcal{SO}(M,v)\) outputs \(u:=
  \frac{1}{\|v\|_2} v \in B_2^d\); note that \(\ip{u}{v} > R\), but,
  by the Cauchy-Schwarz inequality, \(\ip{u}{v'} \le R\) for all \(v'
  \in RB_2^d\);
\item if \(M \succ 0\) and \(v \in RB_2^d\), but
  \((\alpha M, v) \not \in S_K\), then \(\mathcal{SO}(M,v)\) outputs
  a \(d\times d\) matrix \(X\), a vector \(u \in \R^d\), and a real
  number \(t\) such that \(\ip{X}{M} + \ip{u}{v} > t\) but
  \(\ip{X}{M'} + \ip{u}{v'} \le t\) for all \((M',v') \in S_K\);
\item if \((M,v) \in S_K\), then \(\mathcal{SO}(M,v)\) outputs ``Inside''.
\end{itemize}
Note that we make no requirements when $(M,v) \not \in S_K$ but
\((\alpha M, v) \in S_K\). An algorithm \(\mathcal{SO}\) with these
properties is an approximate inverse separation oracle in the sense
of~\cite{BLN21}, except that we only scale \(M\) by \(\alpha\), and
\(S_K\) is, in their terminology, inverse star-shaped with respect
to \(M\) but not \(v\).

The first two requirements we make on \(\mathcal{SO}(M,v)\), which amount to separating \(M\) from
the set of positive definite matrices, or \(v\) from \(RB_2^d\), can
be satisfied using standard techniques. Suppose, then, that
\(M \succ 0\), \(v \in RB_2^d\), but \((\alpha M,v) \not \in
K\). This means that, for some \(x \in K\), \((x + v)^T M^{-1} (x+v) >
\alpha\). Then the \(\alpha\)-approximate quadratic oracle, run with
\(M^{-1}\) and \(v\), would
find a \(\tilde{x} \in K\) such that
\(f_{\tilde{x}}(M,v) := (\tilde{x} + v)^T M^{-1} (\tilde{x}+v) >
1\). Since \(f_{\tilde{x}}(M,v)\) is jointly convex in \(M\) and
\(v\), for any \(d\times d\) matrix \(M'\) and any \(v' \in \R^d\)
we have
\begin{equation}\label{eq:sep-grad}
  f_{\tilde{x}}(M',v') \ge f_{\tilde{x}}(M,v) + \ip{X}{M'-M} + \ip{u}{v'-v},
\end{equation}
where  \(X:= -M^{-1}(\tilde{x}+v)(\tilde{x}+v)^TM^{-1}\) is the gradient
of \(f_{\tilde{x}}(M,v)\) with respect to \(M\), and \(u := 2M^{-1}
(\tilde{x}+v)\) is the gradient with respect to \(v\). We let
\(\mathcal{SO}(M,v)\) output \(X\) and \(u\) as just defined, and \(t:=
\ip{X}{M} + \ip{u}{v} + 1-f_{\tilde{x}}(M,v)\). We have
\[
  \ip{X}{M} + \ip{u}{v} = t + f_{\tilde{x}}(M,v) - 1 > t.
\]
Moreover, any \(M', v'\) for which
\[
  \ip{X}{M'} + \ip{u}{v'} > t
  \iff
  \ip{X}{M'-M} + \ip{u}{v'-v} > 1-f_{\tilde{x}}(M,v) 
\]
satisfy \(f_{\tilde{x}}(M',v') > 1\) by \eqref{eq:sep-grad}, so \((M',v')
\not \in S_K\). In the contrapositive, this means that \(S_K
\subseteq \{(M',v'): \ip{X}{M'} + \ip{u}{v'} \le t\}\), as
required. 

Finally, if, the approximate quadratic oracle called with \(M^{-1}\) and
\(v\)  returns some \(\tilde{x}\) such that \((\tilde{x} + v)^T
M^{-1} (\tilde{x}+v)\le 1\), \(\mathcal{SO}(M,v)\) returns
``Inside''. It is clear that this will be the case when \((M,v) \in S_K\).
This completes the description of \(\mathcal{SO}\). It is easy to
verify that the running time of the algorithm is polynomial in
\(T,d\) and the bit length of the description of \(M\) and \(n\).

The theorem now follows from (a slight modification of) Propositions 3.5 and 3.6
in~\cite{BLN21}. In particular, an argument analogous to Proposition
3.5 shows that, for any choice of \(\lambda\), running the ellipsoid method with
the separation oracles above for number of steps polynomial in
\(d\) and \(\log(R/\eta)\) either returns some \(M\) and \(v\) such
that
\(M \in \lambda S_{p/2}\)
and
\((\alpha M,v) \in S_K\)
or certifies that \(\lambda S_{p/2} \cap S_K\) does not contain a
Euclidean ball of radius \(\eta\). Here, \(\eta > 0\) is a parameter
we can choose. Then, to solve
\eqref{eq:cp-obj}--\eqref{eq:cp-psd} we do binary search on
\(\lambda\), starting with \(\lambda := d^{2/p}R^2\) (which is
feasible with \(R^2I \in \lambda S_{p/2}, (R^2I,0) \in S_K\)), and ending with
some \(\lambda^* \ge \beta\) such that
\begin{itemize}
\item there exist \(M\) and \(v\) for which \(M \in \lambda^*
  S_{p/2}\) and \((\alpha M, v) \in S_K\);
\item \((\lambda^* - \beta) S_{p/2} \cap S_K\) does not contain a Euclidean ball
  of radius \(\beta\).
\end{itemize}
Note that this binary search will terminate in \(O(\log(dR/\beta))\)
iterations. 
The \(M\) and \(v\) in the first item are such that \(\alpha M\) and
\(v\) are feasible for   \eqref{eq:cp-obj}--\eqref{eq:cp-psd} and
achieve objective value \(\alpha \lambda^*\).   
Let \(\mathrm{OPT}\) be the optimal value of
\eqref{eq:cp-obj}--\eqref{eq:cp-psd}, 
achieved by some \((M^*,v^*) \in
S_K\) such that \(\trp{p/2}(M) = \mathrm{OPT}\). It remains to
relate \(\lambda^*\) to \(\mathrm{OPT}\).

Let \(M' := (1+\beta)M^* +
\beta I.\) We will show that, for a small enough \(\eta > 0\), \(S_K\)
contains a Euclidean ball of radius \(\eta\) around \((M', v^*)\).
Since \((M^*,v^*) \in S_K\), and \(M' \succeq (1+\beta)M^*\), we
know that, for any \(x \in K\), \((x+v^*)^T (M')^{-1} (x+v^*) \le
\frac{1}{1+\beta}\). The result then follows from the inequality
\[
  \|(M'')^{-1} - (M')^{-1}\|_{2\to 2}
  \le
  \frac{\|(M')^{-1}\|^2_{2\to 2} \|M'-M''\|_{2\to 2}}{1-\|(M')^{-1}\|_{2\to 2}\|M'-M''\|_{2\to 2}}
  \le
  \frac{\|M'-M''\|_{2\to 2}}{\beta(\beta-\|M'-M''\|_{2\to 2})}.
\]
which holds for any positive semidefinite matrix \(M''\) such that
\(\|M'-M''\|_{2\to 2} < \frac{1}{\|(M')^{-1}\|_{2\to 2}}\). This
inequality can be verified using the formula, $(I + A)^{-1} = I - A +
A^2 - A^3 + \ldots$, valid whenever $\|A\|_{2\to 2}  < 1$.
Moreover, we only
need to take \(\eta\) so that \(\frac{1}{\eta}\) is polynomial in
\(d\), \(R\), and \(\frac{1}{\beta}\). We omit the tedious but
straightforward calculations.

Finally, we need to also show that \((\lambda - \beta)S_{p/2}\)
contains a Euclidean ball of radius \(\eta\) around \(M\) if
\(\lambda\) is large enough. This would then imply that \(\lambda^*
\le \lambda\). We have that,
for any \(M''\) such that \(\|M'' - M'\|_F \le \beta\), by the
triangle inequality,
\[
  \trp{p/2}(M'') \le \trp{p/2}(M')  + \trp{p/2}(M''-M')
  \le (1+\beta)\mathrm{OPT} + \beta d^{p/2} + \beta \max\{d^{2/p-1/2},1\},
\]
where we used the fact that for any \(x \in \R^d\), \(\|x\|_{p/2}
\le d^{2/p - 1/2}\|x\|_2\) if \(2 \le p \le 4\), and \(\|x\|_{p/2} \le
\|x\|_2\) otherwise. 
This means that, whenever \(\lambda \ge
(1+\beta)\mathrm{OPT} + \beta (d^{p/2} +  \max\{d^{2/p-1/2},1\} +
1)\), \((\lambda - \beta)S_{p/2}\) contains a ball of radius
\(\beta\) around \(M'\). Therefore,
\[
  \lambda^* \le (1+\beta)\mathrm{OPT} + \beta (d^{p/2} +  \max\{d^{2/p-1/2},1\} +
  1)
  \le \mathrm{OPT} + \beta (R^2d^{p/2} + d^{p/2} +  \max\{d^{2/p-1/2},1\} +   1)
\]
The theorem now follows by rescaling \(\beta\).

\end{document}